\documentclass[11pt,a4paper,reqno]{amsart}
\usepackage[margin=1.0in]{geometry}
\usepackage{amsmath,amsthm,amssymb,amsxtra,comment,graphicx,psfrag}
\usepackage{enumerate,enumitem}
\usepackage{bm,mathrsfs}
\usepackage{mathtools}
\usepackage{xspace}
\usepackage{color}
\usepackage{cite}
\usepackage{subcaption}
\usepackage{multirow}
\usepackage{longtable}
\usepackage{color}
\usepackage{hyperref}
\hypersetup{%
  colorlinks = true,
  citecolor=blue, 
  linkcolor  = black
}

\allowdisplaybreaks

\usepackage{titlesec}
\titleformat{\section}{\vskip10pt\large\bfseries}{\thesection.}{0.5em}{\centering\vspace{5pt}}
\titleformat{\subsection}{\vskip10pt\normalsize\bfseries}{\thesubsection.}{0.5em}{}
\newtheorem{theorem}{Theorem}[section]

\newtheorem{lemma}[theorem]{Lemma}

\theoremstyle{definition}

\def\bfx{{\bf x}}
\def\bfe{{\bf e}}

\def\nu{n}

\def\vb{{\bf v}}

\def\d{{\mathrm d}}

\def\R{{\mathbb R}}

\def\ud{\underline{D}}
\def\md{\partial_{t}^\bullet}

\def\mdc{\partial_{\theta}^\bullet} % theta
\def\mda{\partial_{\alpha}^\bullet} % alpha
\def\ehm{{\hat{e}_h^m}}
\def\eM{{e_h^{m+1}}}
\def\eem{{e_h^{m}}}
\def\ehM{{\hat{e}_h^{m+1}}}

\def\Gm{{\Gamma^m}}
\def\GM{{\Gamma^{m+1}}}
\def\Ghm{{\Gamma^m_h}}
\def\GhM{{\Gamma^{m+1}_h}}
\def\Ghsm{{\hat\Gamma^m_{h, *}}}
\def\GhsM{{\hat\Gamma^{m+1}_{h, *}}}

\def\Ghso{{\Gamma_{h,\rm f}^0}}

\def\Hm{{H^m}}

\def\nm{{n^m}}

\def\nhm{{n^m_h}}

\def\nhsm{{\hat n^m_{h, *}}}

\def\nbhsm{{\bar n^m_{h, *}}}

\def\nsm{{n^m_{*}}}

\def\Thsm{{\hat T^m_{h, *}}}

\def\Tbhsm{{\bar T^m_{h, *}}}

\def\Tsm{{T^m_{*}}}
\def\TsM{{T^{m+1}_{*}}}

\def\Nhsm{{\hat N^m_{h, *}}}

\def\Nsm{{N^m_{*}}}
\def\NsM{{N^{m+1}_{*}}}

\def\Nbhsm{{\bar N^m_{h, *}}}

\def\Tbhm{{\bar{T}^m_h}}
\def\Nbhm{{\bar{N}^m_h}}
\def\nbhm{{\bar{n}^m_h}}

\def\kl{{\kappa_l}}

\def\dtXm{{\delta_\tau X_h^m}}
\def\dtem{{\delta_\tau \hat e_h^m}}

\newcommand{\id}{\operatorname{id}}

\numberwithin{equation}{section}

\begin{document}
\title[]{\parbox[b]{\linewidth+20pt}{\centering Super-approximation and optimal $L^2$ convergence of the stabilized BGN method with piecewise linear parametric finite elements for curve-shortening flow}}

\author[]{Genming Bai}
\address{Genming Bai:  Department of Mathematics and Statistics, Old Dominion University, Norfolk, VA 23508. {\it Email address: \tt gbai@odu.edu}}

%\author[]{Buyang Li}
%\address{Buyang Li (corresponding author): 
%Department of Applied Mathematics, The Hong Kong Polytechnic University,
%Hong Kong. {\it Email address: \tt buyang.li@polyu.edu.hk}}



%\subjclass[2010]{65M12, 65M60, 53E10, 53A04, 35R01, 35R35}
%
%\keywords{Mean curvature flow, curve shortening flow, geometric evolution equation, parametric finite element method, tangential motion, stability, convergence, trajectory, mass lumping, distance projection.} 
\allowdisplaybreaks

\maketitle
\vspace{-20pt}

\begin{abstract}\noindent
{\small 
%In this paper, we develop a bihomotopy argument to estimate the discrete time derivative of the linear forms of consistency errors.
%Together with a summation-by-parts trick,
%we discover a directional super-convergence estimate for the consistency error tested with error velocity.
%Based on these new techniques, we are able to prove the super-approximation and the optimal $L^2$ convergence of a stabilized BGN method for curve-shortening flow with linear parametic finite elements.
We develop a bihomotopy argument to estimate the discrete material time
derivatives of consistency functionals.
Combined with summation-by-parts, this argument yields a
directional super-convergence estimate for the consistency contribution
when tested with the normal error velocity.
Together with the super-approximation of the reversely weighted discrete
normal, these techniques establish $L_t^\infty H_x^1$ super-convergence of the projection error and optimal $L_t^\infty L_x^2$ convergence of the trajectory error of a stabilized BGN method for curve-shortening flow with piecewise linear parametric finite elements.
}
\end{abstract} 
%\tableofcontents

\setlength\abovedisplayskip{3.5pt}
\setlength\belowdisplayskip{3.5pt}

\section{Introduction}\label{section:intro}

The tangential motion of mesh points serves as an important source of numerical stability in parametric approximations of geometric flows. Although the graph of geometric evolution is determined by the normal velocity, the distribution of the computational nodes depends on their tangential velocity. The parametric finite element
methods of Barrett, Garcke and N\"urnberg (BGN) \cite{BGN2007JCP,BGN2008JCP,BGN2015a,BGN2016,BGN2020} incorporate a tangential smoothing mechanism into their discrete
weak formulations: The normal motion approximates the prescribed geometric
velocity, while the tangential motion is determined implicitly by the scheme. The underlying mechanism gives favorable mesh redistribution properties, but also presents a difficulty for convergence analysis because of the implicitness of the tangential gauge. 
%, since the
%numerical particle trajectories need not approximate those of a
%parametrization moving only in the normal direction.

%Besides the BGN gauge, there are other tangential gauges which are helpful to maintain the mesh distribution. 
The BGN method can also be viewed as a limiting case of the family of DeTurck-based parametric finite element methods introduced by Elliott and Fritz \cite{EF2017}. Schemes corresponding to other members of this family have been studied and analyzed in \cite{DD1994,BDS2017,DN2026}.
%We also note that the BGN method can be identified as the endpoint case of a more general DeTurck-gauged family of methods proposed by Elliott and Fritz \cite{EF2017}. Other instances of the family of Elliott and Fritz have been considered and analyzed in \cite{DD1994,BDS2017,DN2026}.

Let $\Gamma(t)\subset\mathbb R^2$, $t\in[0,T]$, be a smooth family of closed curves, with unit normal $n$ and curvature $H$ satisfying $Hn=-\Delta_{\Gamma(t)}{\rm id}$. The curve-shortening law prescribes the normal velocity $v\cdot n=-H$.
Parametric finite element approximations for this evolution originate from Dziuk's method \cite{Dziuk1991}. In contrast, BGN-type methods differ from Dziuk's method by retaining the normal component of the mass bilinear form.
%Parametric finite element approximations of this evolution originate
%from Dziuk's method \cite{Dziuk1991}, while the BGN-type methods only differ from Dziuk's method by only keeping the normal part of the mass bilinear form.
%Convergence analyses for semidiscrete and fully discrete versions were
%developed in \cite{Dziuk1994,PS2017,Li2020,YC2021}, respectively.
%Another approach introduces a tangential reparametrization at the
%continuous level, as in the DeTurck-based methods of Elliott and Fritz
%\cite{EF2017}.
%For BGN-type methods, the tangential motion is instead specified through
%the discrete equations, and its stability must be analyzed together with
%the geometric approximation.

%We now describe the numerical method. 
Define $t_m=m\tau$, where $\tau>0$ is the time step, and let $\Ghm$ be a polygonal approximation of
$\Gm=\Gamma(t_m)$. We denote by $S_h(\Ghm)$ the space of continuous
piecewise linear functions on $\Ghm$, by $I_h$ the nodal interpolation
operator, and by a superscript $h$ on an integral the use of mass
lumping defined in Section \ref{subsec:Lagrange-interpolation}. At a node $p$ shared by two adjacent edges, let
$\nhm(p-)$ and $\nhm(p+)$ be their unit normals on two sides, and let
$\ell_{h,-}^m$ and $\ell_{h,+}^m$ be their respective lengths.
The reversely weighted normal proposed recently in \cite{Bai2026} is defined by
\begin{equation}\label{eq:intro-reverse-normal}
	\nbhm(p)
	=
	\frac{
		\ell_{h,+}^m \nhm(p-)
		+
		\ell_{h,-}^m \nhm(p+)
	}{
		\ell_{h,-}^m+\ell_{h,+}^m
	}.
\end{equation}
Thus, the normal of each edge is weighted by the length of the opposite edge. These nodal values define $\nbhm\in[S_h(\Ghm)]^2$.
Its counterpart $\nbhsm$ on the consistency curve satisfies the super-approximation estimate in Lemma~\ref{lemma:n_bar_app}, which plays a central role in the convergence analysis.
%Its consistency counterpart $\nbhsm$ enjoys a crucial super-approximation property (Lemma \ref{lemma:n_bar_app}), which is a key ingredient in the convergence proof.

Given $\Ghm$, the proposed stabilized BGN method determines
$X_h^{m+1}\in[S_h(\Ghm)]^2$ by
%Given $\Ghm$, we propose the following stabilized BGN method. We solve for $X_h^{m+1}\in[S_h(\Ghm)]^2$ from
\begin{align}\label{eq:BGN-stab}
	&\int_{\Ghm}^{h}
	\frac{X_h^{m+1}-{\rm id}}{\tau}
	\cdot\nbhm\,\phi_h\cdot\nbhm
	+
	\int_{\Ghm}
	\nabla_{\Ghm}X_h^{m+1}\cdot
	\nabla_{\Ghm}\phi_h
	\notag\\
	&\qquad=
	\int_{\Ghm}
	\nabla_{\Ghm}{\rm id}\cdot
	\nabla_{\Ghm}
	I_h\!\left[
	\phi_h-(\phi_h\cdot\nbhm)\nbhm
	\right]
	\qquad
	\forall\,\phi_h\in[S_h(\Ghm)]^2,
\end{align}
and sets $\GhM:=X_h^{m+1}(\Ghm)$.
The averaged normal $\nbhm$ is used without normalization in both the mass-lumped term and the stabilization term.
All coefficients are evaluated on the known curve $\Ghm$, so each time step requires the solution of a linear system.

The stabilization on the right-hand side of \eqref{eq:BGN-stab} has the
same form as that originally introduced in \cite{BL2025}.
Its continuous counterpart vanishes, since
\[
\int_{\Gamma}
\nabla_{\Gamma}{\rm id}\cdot
\nabla_{\Gamma}\bigl[\phi-(\phi\cdot n)n\bigr]
=
\int_{\Gamma}
Hn\cdot\bigl[\phi-(\phi\cdot n)n\bigr]
=0.
\]
At the discrete level, it provides the source of the stability for the tangential velocity, compensating the loss for tangential control in the mass bilinear form. Indeed, if $\phi_h\cdot\nbhm=0$ at every node, then
\eqref{eq:BGN-stab} gives
\[
\int_{\Ghm}
\nabla_{\Ghm}
\frac{X_h^{m+1}-{\rm id}}{\tau}
\cdot\nabla_{\Ghm}\phi_h
=0.
\]
This identity underlies the tangential stability estimates in
Section~\ref{sec:tan_stab}.
As shown below, the continuous comparison velocity $v$ is determined
by the elliptic velocity system \eqref{subeq:system},
\[
v\cdot n=-H,
\qquad
-\Delta_\Gamma v=\kappa n.
\]
This is the Euler--Lagrange system for the pointwise constrained minimization
problem
\begin{subequations}
	\label{subeq:var}
	\begin{align}\label{eq:energy_I}
		\min_{v\in\mathcal H}
		\int_\Gamma |\nabla_\Gamma v|^2,
	\end{align}
	where the convex admissible set is
	\begin{align}\label{eq:admissible_set}
		\mathcal H
		=
		\bigl\{
		v\in[H^1(\Gamma)]^2:
		v\cdot n=-H
		\quad\mbox{a.e. on }\Gamma
		\bigr\}.
	\end{align}
\end{subequations}
Thus, the comparison velocity minimizes the Dirichlet energy subject
to the prescribed normal velocity.
This variational characterization explains the tangential smoothing
mechanism associated with the stabilized BGN method
\eqref{eq:BGN-stab}.
%This identity yields the tangential stability estimates in Section \ref{sec:tan_stab}. As we will show later,
%the continuous comparison velocity is the solution $v$ of the elliptic velocity system
%\eqref{subeq:system}, which satisfies both $v\cdot n=-H$ and
%$-\Delta_{\Gamma}v=\kappa n$.
%Indeed, this system is the Euler--Lagrange equation of the functional convex program
%\begin{subequations}\label{subeq:var}
%	\begin{align}\label{eq:energy_I}
%		\min_{v} \int_\Gamma |\nabla_\Gamma v|^2 
%	\end{align}
%	over the convex admissible set
%	\begin{align}\label{eq:admissible_set}
%		\mathcal{H} = \{ v\in H^1(\Gamma): v\cdot n = u\cdot n \quad\mbox{a.e.} \} .
%	\end{align}
%\end{subequations}
%The variational characterization above accounts for the tangential smoothing mechanism
%associated with the stabilized BGN method \eqref{eq:BGN-stab}.

The main purpose of this paper is to systematically develop
super-convergence estimates within the projection-error framework
introduced in \cite{BL2024} and subsequently to establish convergence of the proposed
stabilized BGN method \eqref{eq:BGN-stab} with piecewise linear finite elements (finite element degree $k=1$) for the first time.

Super-convergence estimates are particularly effective in the analysis of lowest-order parametric finite element methods; see, for example,
\cite{PS2017,DN2026}.
Establishing such estimates for the projection error requires additional
care because the nodal closest-point projection introduces geometric
remainders that are absent from the standard trajectory-error analysis.
Once these remainders are controlled, the projection-error framework also yields optimal $L^2$
convergence of the numerical particle trajectories; see
Section~\ref{Proof-THM-2}.
In the present lowest-order stability argument, the use of the projection error is indispensable because the blow-up factor
$h^{-2}\|\nabla_\Ghsm\ehm\|_{L^2(\Ghsm)}$
in Lemma~\ref{lemma:NT_stab} cannot be handled using the ordinary trajectory-error estimates alone.
For convergence analyses of higher-order parametric finite element methods of BGN type, we refer to \cite{BL2025,BGV2026}.

%The main purpose of this paper is to systemically investigate super-convergence estimates under the framework of projection error, originally proposed in \cite{BL2025}, and prove the convergence of a BGN-type method in the regime of lowest finite element degree $k=1$ for the first time. 
%Super-convergence estimates are known to be extremely powerful dealing with lowest-degree parametric finite element methods; cf. \cite{PS2017,DN2026}.
%We should note that the super-convergence results for the projection error are more difficulty to establish because of the presence of the nodal projection. Therefore, our analysis naturally carries over to the ordinary trajectory error analysis. Nevertheless, our approach of the projection error does not lose any generality. We can still recover the optimal $L^2$ convergence rate along the particle trajectory in Theorem {\r???}.
%We also note that, for the lowest finite element degree $k=1$, the use of the projection error is necessary, otherwise keeping track of the trajectory error directly cannot resolve the unstable factor $h^{-2}\| \nabla_\Ghsm \ehm \|_{L^2(\Ghsm)}$ in Lemma \ref{lemma:NT_stab}.
%For the convergence analysis of higher-order parametric finite element methods for the BGN-type methods, we refer the readers to \cite{BL2025,BGV2026}.

We now outline the main super-convergence argument. At time level $m$, the consistency functional satisfies (see Lemma~\ref{proposition-consistency})
\[
|\mathscr D^m(\phi_h)|
\lesssim
\tau\|\phi_h\|_{L^2(\Ghsm)}
+
h^2\|\phi_h\|_{H^1(\Ghsm)} ,
\]
%where $\Ghsm$ is the consistency curve defined in Section \ref{subsec:numerical-and-interpolated-curves}.
%Unfortunately, such $L^2$-suboptimal rate is not sufficient to overcome the blow-up of $h^{-2}\| \nabla_\Ghsm \ehm \|_{L^2(\Ghsm)}$ in Lemma \ref{lemma:NT_stab}.
where $\Ghsm$ is the consistency curve defined in Section~\ref{subsec:numerical-and-interpolated-curves}.
By the inverse inequality, this gives only an $O(\tau+h)$ bound
when the test function is measured in $L^2$.
The resulting loss of one spatial order prevents the velocity estimate
from closing the stability argument involving the factor
$h^{-2}\|\nabla_\Ghsm\ehm\|_{L^2(\Ghsm)}$
in Lemma~\ref{lemma:NT_stab}.

To retain the optimal-order consistency, we test with the normal error velocity
\[
I_h\Nbhsm\dtem
=
I_h\Nbhsm
\Big(
\frac{\eM-\ehm}{\tau}
-
I_h\Tsm v^m
\Big).
\]
Here $\Tsm$ and $\Nsm$ are the exact tangential and normal projectors
pulled back to $\Ghsm$, and $\Nbhsm$ is the orthogonal projector onto
the direction of the reversely weighted normal $\nbhsm$.
Moreover, $\eM$ is the one-step trajectory error relative to the
interpolated exact normal flow, $\ehm$ is the projection error, and
$v^m$ is the continuous comparison velocity.
Their precise definitions are given in
Sections~\ref{subsec:averaged-normal-basic}
and~\ref{section:geometry}.
Functions on different curves are compared through their canonical
nodal identification.

The summation-by-parts argument in Section~\ref{sec:bbd_vel_sup}
gives the decomposition
\begin{align*}
	\mathscr D^m(I_h\Nbhsm\dtem)
	&=
	\frac{
		\tilde{\mathscr D}^{m+1}(\ehM)
		-
		\tilde{\mathscr D}^{m}(\ehm)
	}{\tau}
	+
	\sum_{i=1}^5 D_i^m,
\end{align*}
where $\tilde{\mathscr D}^m$ defined in Section~\ref{sec:cons_err} collects the critical consistency contributions whose estimates involve the $H^1$ test norm.
The first term telescopes after multiplication by $\tau$ and summation
over time.
Two of the remainders require particular attention.

%To carry out the super-convergence argument to the situation where the consistency functional is tested with the normal error velocity $I_h\Nbhsm\dtem := I_h\Nbhsm \Big(
%\frac{\eM-\ehm}{\tau}
%-
%I_h\Tsm v^m
%\Big)$, where $\Tsm$, $\Nsm$ and $\Nbhsm$ are projection operators, $\eM$ is the one-step consistency errror, $\ehm$ is the projection error and $v^m$ is comparison velocity; see Sections \ref{subsec:averaged-normal-basic} and \ref{section:geometry} for their detailed definitions.
%We apply a summation-by-parts argument to make up a telescopic sum up to several remainders (see Section \ref{sec:bbd_vel_sup})
%\begin{align}
%	{\mathscr D}^m(I_h\Nbhsm\dtem )
%	&=
%	\frac{\tilde{\mathscr D}^{m+1}( \hat e_h^{m+1})- \tilde{\mathscr D}^{m} ( \hat e_h^m)}{\tau} 
%	+
%	\sum_{i=1}^5 D_i^m ,
%	\notag
%\end{align}
%where $\tilde{\mathscr D}^m$, defined in Section \ref{sec:cons_err}, is collection of the bad part of ${\mathscr D}^m$.
%Most importantly,
\begin{itemize}
	\item 
%	The term
%	\[
%	D_3^m
%	:=
%	-\frac{
%		\tilde{\mathscr D}^{m+1}(I_h\NsM\ehM)
%		-
%		\tilde{\mathscr D}^{m}(I_h\NsM\ehM)
%	}{\tau}
%	\]
%	involves the discrete material derivative of the consistency
%	functional, with the test function transported through its nodal
%	values.
%	To estimate this derivative, we construct a bihomotopy family of
%	curves in Section~\ref{sec:bbd_vel_Abel}.
%	One parameter connects consecutive time levels, while the other
%	connects the piecewise linear interpolant to its smooth projected
%	parametrization.
%	In Lemma \ref{bhk1:main}, we are able to show that the time derivative will not deteriorate the consistency rate.
	The term
	\[
	D_3^m
	:=
	-\frac{
		\tilde{\mathscr D}^{m+1}(I_h\NsM\ehM)
		-
		\tilde{\mathscr D}^{m}(I_h\NsM\ehM)
	}{\tau}
	\]
	involves the discrete material derivative of the consistency
	functional, evaluated at a test function with the same nodal
	values at both time levels.
	To estimate this derivative, we construct a bihomotopy family
	of curves in Section~\ref{sec:bbd_vel_Abel}.
	One parameter connects consecutive time levels, while the other
	connects the piecewise linear interpolant to its smooth projected
	parametrization.
	Lemma~\ref{bhk1:main} shows that material differentiation preserves
	the optimal order in the consistency estimate.
%	Test functions are transported through this family, so that the
%	changes in the integration domain, the geometric coefficients,
%	and the quadrature weights are included in the differentiation.
%	Lemma~\ref{bhk1:main} gives
%	\[
%	|\delta_\tau\tilde{\mathscr D}^m(\phi_h)|
%	\lesssim
%	h^2
%	\left(
%	1+
%	\|\nabla_\Ghsm I_h\Tsm\dtem\|_{L^2(\Ghsm)}
%	\right)
%	\|\phi_h\|_{H^1(\Ghsm)}.
%	\]
%	Thus, material differentiation preserves the explicit factor $h^2$,
%	with a dependence on the tangential error velocity that is handled
%	together with the tangential stability estimates.
	
%	The term $D_3^m := -
%	\frac{1}{\tau} \Big( \tilde{\mathscr D}^{m+1}(I_h\NsM \hat e_h^{m+1})- \tilde{\mathscr D}^{m}(I_h\NsM \hat e_h^{m+1}) \Big)$ is the discrete time derivative of $\tilde{\mathscr D}^{m}$. To handle this term, we systematically develop a bihomotopy argument in Section \ref{sec:bihomo}. In Lemma \ref{bhk1:main}, we are able to show that the time derivative will not deteriorate the consistency rate.
	
	\item 
	The term
	\[
	D_5^m
	:=
	\tilde{\mathscr D}^m
	\Big(
	I_h\Nsm
	\Big(
	\frac{\eM-\ehM}{\tau}
	-
	I_h\Tsm v^m
	\Big)
	\Big)
	\]
	is a geometric remainder caused by the nodal reprojection in $\ehM$.
	It measures the difference between the corrected one-step error $\eM-\tau I_h\Tsm v^m$ and the consecutive projection error $\ehM$, divided by $\tau$.
	Lemma~\ref{lemma:corrected-displacement-nodal} establishes the
	super-approximation estimate
	\[
	\Big\|
	I_h\Nsm
	\Big(
	\frac{\eM-\ehM}{\tau}
	-
	I_h\Tsm v^m
	\Big)
	\Big\|_{H^1(\Ghsm)}
	\lesssim \tau ,
	\]
	reflecting the geometric rigidity of the above error displacement.
	Lemma~\ref{lemma:corrected-displacement-nodal} is also used to control
	the crucial term $A_3^m$ in Section~\ref{sec:bbd_vel_sup}.
\end{itemize}
%The other remainders $D_1^m$, $D_2^m$ and $D_4^m$ are trivially harmless. So, we can recover the $L^2$-optimality of the consistency error in the following sense:
The remaining terms $D_1^m$, $D_2^m$, and $D_4^m$ can be controlled using the standard estimates. Consequently, the optimal consistency order is retained in the following sense,
\begin{align}
	\Big|{\mathscr D}^m(I_h\Nbhsm\dtem )
	-
	\frac{\tilde{\mathscr D}^{m+1}( \hat e_h^{m+1})- \tilde{\mathscr D}^{m} ( \hat e_h^m)}{\tau} 
	\Big|
	\lesssim
	(\tau+h^2)^2
	+
	\text{stability terms}
	,
	\notag
\end{align}
which can be interpreted as a commutator estimate for exchanging the order of discrete time differentiation and the consistency functional.
%This serves as the main source of super-convergence for the error velocity along the normal direction.
%As a consequence, we are able to show the super-convergence of error under the $L_t^\infty H_x^1$ norm (see Theorem \ref{thm:main}) for the lowest order finite element $k=1$.
After summation in time, the telescoping contribution reduces to
optimal-order endpoint terms.
This structure is the main source of super-convergence for the normal
error velocity and ultimately yields the $L_t^\infty H_x^1$
super-convergence estimate for the projection error with piecewise
linear finite elements; see Theorem~\ref{thm:main}.

The remainder of the paper is organized as follows.
Section~\ref{section:preparation} introduces the notation and geometric
framework and states the main convergence results.
Section~\ref{sec:cons_err} establishes the consistency estimates.
Section~\ref{Proof-THM-1} derives the error equation and estimates the
associated linear and bilinear forms.
Section~\ref{sec:tan_stab0} establishes tangential
stability, basic velocity estimates, and bounds for the errors and
nodal displacements.
Section~\ref{sec:bbd_vel_Abel} develops the bihomotopy argument and
estimates the discrete material derivatives of the consistency
functionals. These estimates are combined with summation by parts
to prove velocity super-convergence and the main error estimates.
The convergence proof is completed by establishing uniform mesh
regularity in Section~\ref{sec:bbd} and trajectory convergence
in Section~\ref{Proof-THM-2}.
The appendices summarize the notation and collect surface calculus
formulas, super-approximation estimates, discrete norm equivalences,
and Poincar\'e inequalities.

\section{Notation and the main result}
\label{section:preparation}

\subsection{Polygonal curves and finite element functions}
\label{subsec:parametric-finite-elements}

We work with polygonal curves in $\R^2$, whose edges form the finite
element mesh $\mathcal T$. Writing
\[
D(\mathcal T):=\bigcup_{K\in\mathcal T}K,
\]
we denote the set of vertices by $\mathcal N(\mathcal T)$ and the mesh
size by $h:=\max_{K\in\mathcal T}\operatorname{diam}K$. Adjacent edges
share their endpoints, so their affine parametrizations agree there.
The resulting curve is a piecewise regular, globally $C^0$ embedding.

Each edge is parametrized by an affine map
$\hat F_K:\hat K\rightarrow K$, where $\hat K=[0,1]$.
With $\mathbb P_1(\hat K)$ denoting the polynomials of degree at most
one, the scalar finite element space is
\[
S_h(\mathcal T)
:=
\bigl\{
v_h\in C^0(D(\mathcal T)):
v_h|_K\circ\hat F_K\in\mathbb P_1(\hat K)
\quad\forall K\in\mathcal T
\bigr\}.
\]
Thus $S_h(\mathcal T)$ consists of continuous functions that are affine
on every edge. We use $[S_h(\mathcal T)]^q$ for $\R^q$-valued functions.
The notation $W_h^{s,p}(D(\mathcal T)), s>0,$ refers to the broken Sobolev
space with elementwise Sobolev--Slobodeckij norms, and $H^s=W^{s,2}$.
The discrete norms $L_h^p$ are defined in Appendix~\ref{sec:disc-norm}.

The different polygonal curves are compared through their nodal values.
Two meshes $\mathcal T_1$ and $\mathcal T_2$ with the same number of
edges are called equivalent, written $\mathcal T_1\cong\mathcal T_2$,
if a $W^{1,\infty}$ homeomorphism
$f_h\in[S_h(\mathcal T_1)]^2$ maps $D(\mathcal T_1)$ onto
$D(\mathcal T_2)$ and maps each edge of $\mathcal T_1$ onto its
corresponding edge of $\mathcal T_2$. The induced identification is
\begin{equation}
	v_h\in S_h(\mathcal T_1)
	\quad\longleftrightarrow\quad
	v_h\circ f_h^{-1}\in S_h(\mathcal T_2).
	\label{eq:nodal-identification}
\end{equation}
Both realizations have the same nodal vector. Throughout the paper,
a finite element function is identified with this common vector;
the domain of an integral or norm specifies its realization.
In particular, the notation $v_h$ may represent the same nodal
vector $\vb$ on different equivalent curves without identifying
their physical domains.

For example, the two integrands of $$\int_{D(\mathcal T_1)} v_h \quad\mbox{and}\quad \int_{D(\mathcal T_2)} v_h$$ have the same vector of nodal values, denoted by $\vb$, but are defined on different domains $D(\mathcal T_1)$ and $D(\mathcal T_2)$. When the underlying domain is specified, $\vb$ is automatically realized to a finite element function $v_h$ on that domain. Since all of the quantitative computations in this paper involve either integrals or norms, our notation for finite element functions will always have a unique and clear meaning. For another example, $\| v_h \|_{L^2(D(\mathcal T_1))}$ and $\| v_h \|_{L^2(D(\mathcal T_2))}$ denote the norms of finite element functions with the same nodal vector on the two different domains $D(\mathcal T_1)$ and $D(\mathcal T_2)$, respectively.

\subsection{Lagrange interpolation and mass lumping}
\label{subsec:Lagrange-interpolation}

Let $I_{\hat K}:C^0(\hat K)\rightarrow\mathbb P_1(\hat K)$ be the
linear interpolant at the endpoints of $\hat K$. On a polygonal
mesh $\mathcal T$, the nodal interpolant
$I_h(\mathcal T):C^0(D(\mathcal T))\rightarrow S_h(\mathcal T)$ is
defined by
\[
\bigl(I_h(\mathcal T)v\bigr)|_K
=
I_{\hat K}[v\circ\hat F_K]\circ\hat F_K^{-1},
\qquad K\in\mathcal T.
\]
Interpolation of vector- and matrix-valued functions is understood
componentwise. Under the identification
\eqref{eq:nodal-identification},
\[
I_h(\mathcal T_2)(v\circ f_h^{-1})
=
\bigl(I_h(\mathcal T_1)v\bigr)\circ f_h^{-1}.
\]
We simply write $I_h$ when the underlying mesh is clear from the context.

The mass-lumped integral is the integral of the Lagrange interpolant.
Since every element is a line segment, it takes the form
\begin{equation}
	\int_{D(\mathcal T)}^h f
	:=
	\int_{D(\mathcal T)}I_hf
	=
	\sum_{K\in\mathcal T}\frac{|K|}{2}
	\sum_{p\in\mathcal N(K)}f(p).
	\notag
%	\label{eq:mass-lumped-inner-product}
\end{equation}

\subsection{The numerical and nodally projected curves}
\label{subsec:numerical-and-interpolated-curves}

Let $\mathcal T_h^0$ be the mesh of the initial polygonal curve
$\Gamma_h^0$. For each edge $K^0$, write $K_{\rm f}^0=K^0$ and retain
the reference-curve notation
\begin{equation}
	\Ghso
	:=
	\bigcup_{K^0\subset\Gamma_h^0}K_{\rm f}^0
	=
	\Gamma_h^0.
	\notag
%	\label{eq:flat-reference-curve}
\end{equation}
Its mesh is denoted by $\mathcal T_{h,{\rm f}}^0$.

At time $t_m=m\tau$, $m=0,\ldots,\lfloor T/\tau\rfloor$, let
$\bfx^m=(x_1^m,\ldots,x_J^m)^\top$ be the nodal vector generated
by \eqref{eq:BGN-stab}. Its realization on the reference curve is
$X_h^m\in[S_h(\Ghso)]^2$, and the numerical mesh and curve are
\begin{equation}
	\mathcal T_h^m
	:=
	\bigl\{
	X_h^m(K_{\rm f}^0):
	K_{\rm f}^0\in\mathcal T_{h,{\rm f}}^0
	\bigr\},
	\qquad
	\Ghm:=D(\mathcal T_h^m)=X_h^m(\Ghso).
	\notag
%	\label{eq:numerical-parametric-curve}
\end{equation}
Whenever $X_h^m$ is regular, elementwise injective and free of
self-intersections, this mesh is equivalent to
$\mathcal T_{h,{\rm f}}^0$. By
\eqref{eq:nodal-identification}, $X_h^m$ realized on $\Ghm$ is
$ {\rm id}_{\Gamma_h^m}$, whereas $X_h^{m+1}$ realized on $\Ghm$
is the local discrete flow map in \eqref{eq:BGN-stab}.

To define the consistency curve, let $\Gm=\Gamma(t_m)$ and let
\[
a^m:U_\delta(\Gm)\longrightarrow\Gm
\]
be the distance retraction from a tubular neighborhood of $\Gm$.
Whenever $\Ghm\subset U_\delta(\Gm)$, project its vertices by setting
\begin{equation}
	\hat{\bf x}_*^m
	:=
	\bigl(a^m(x_1^m),\ldots,a^m(x_J^m)\bigr)^\top.
	\notag
%	\label{eq:projected-nodal-vector}
\end{equation}
The finite element function on $\Ghso$ with this nodal vector is denoted by
$\hat X_{h,*}^m$. Equivalently,
\begin{equation}
	\hat X_{h,*}^m
	=
	I_h(a^m\circ X_h^m)
	\qquad\hbox{on }\Ghso.
	\notag
%	\label{eq:interpolated-parametric-map}
\end{equation}
Its image defines the consistency mesh and curve:
\begin{equation}
	\hat{\mathcal T}_{h,*}^m
	:=
	\bigl\{
	\hat X_{h,*}^m(K_{\rm f}^0):
	K_{\rm f}^0\in\mathcal T_{h,{\rm f}}^0
	\bigr\},
	\qquad
	\Ghsm
	:=
	D(\hat{\mathcal T}_{h,*}^m)
	=
	\hat X_{h,*}^m(\Ghso).
	\notag
%	\label{eq:interpolated-parametric-curve}
\end{equation}

\subsection{Lift and inverse lift}
\label{subsec:lift-and-inverse-lift}

When $\Ghsm$ is sufficiently close to $\Gm$, the restriction
$a^m|_{\Ghsm}:\Ghsm\rightarrow\Gm$ is bijective. We use it to
transfer functions between the polygonal consistency curve and
the exact curve. For a function $v$ on $\Ghsm$, its lift is
\[
v^\ell:=v\circ(a^m|_{\Ghsm})^{-1}.
\]
For a function $f$ on $\Gm$, the inverse lift is
\[
f^{-\ell}:=f\circ a^m|_{\Ghsm} .
\]
The same convention applies to finite element functions identified
with functions on $\Ghsm$ through their nodal vectors.

Every vertex of $\Ghsm$ lies on $\Gm$. Consequently, if
$K=F_K(K_{\rm f}^0)\subset\Ghsm$, where $F_K$ is the affine
parametrization over $K_{\rm f}^0\in\mathcal T_{h,{\rm f}}^0$,
then $a^m\circ F_K$ and $F_K$ agree at the endpoints. Hence
\begin{equation}
	I_{K_{\rm f}^0}[a^m\circ F_K]=F_K,
	\qquad
	I_h(a^m|_K)
	=
	I_{K_{\rm f}^0}[a^m\circ F_K]\circ F_K^{-1}
	=
	{\rm id}|_K
	\quad\hbox{on }K.
	\label{eq:interpolation-of-distance-projection}
\end{equation}

\subsection{Mesh regularity and interpolation error estimates}

We measure the regularity of a polygonal mesh through its
piecewise affine parametrization. Let
$F(\mathcal T)\in[S_h(\mathcal T_0)]^2$ parametrize $\mathcal T$
over a reference mesh $\mathcal T_0$, taken here to be the mesh
of $\Ghso$. Define
\begin{align}\label{P}
	\kappa_*(\mathcal T)
	:=
	\|F(\mathcal T)\|_{W^{1,\infty}(D(\mathcal T_0))}
	+
	\|F(\mathcal T)^{-1}\|_{W^{1,\infty}(D(\mathcal T))}.
\end{align}
The linear interpolation estimate on each reference edge,
together with the chain rule, gives the following result.

\begin{lemma}\label{lemma:Ih}
	For a polygonal mesh $\mathcal T$, let
	$f\in W_h^{2,p}(D(\mathcal T))\cap C^0(D(\mathcal T))$,
	where $1\leq p\leq\infty$. Then
	\begin{align*}
		&\|(1-I_h)f\|_{L^p(D(\mathcal T))}
		+h\|\nabla_{D(\mathcal T)}(1-I_h)f\|_{L^p(D(\mathcal T))}
		\\
		&\qquad\leq
		C_{\kappa_*(\mathcal T)}
		h^2\|f\|_{W_h^{2,p}(D(\mathcal T))}.
	\end{align*}
\end{lemma}

For finite element functions identified by their nodal vectors,
pullback by $F(\mathcal T)$ and its inverse also yields
\[
C_{\kappa_*}^{-1}
\|v_h\|_{W^{1,p}(D(\mathcal T_0))}
\leq
\|v_h\|_{W^{1,p}(D(\mathcal T))}
\leq
C_{\kappa_*}
\|v_h\|_{W^{1,p}(D(\mathcal T_0))}.
\]

We define the shape regularity constant of the evolving consistency mesh $\Ghsm$ as
\begin{align}\label{PP}
	\begin{split}
		\kappa_l
		:=&
		\max_{0\leq m\leq l}
		\big(
		\|\hat X_{h,*}^m\|_{W^{1,\infty}(\Ghso)}
		+
		\|(\hat X_{h,*}^m)^{-1}\|_{W^{1,\infty}(\Ghsm)}
		\big)
		\\
		=&
		\max_{0\leq m\leq l}
		\big(
		\max_{K\subset\Ghsm}
		\|F_K\|_{W^{1,\infty}(K_{\rm f}^0)}
		+
		\max_{K\subset\Ghsm}
		\|F_K^{-1}\|_{W^{1,\infty}(K)}
		\big),
	\end{split}
\end{align}
where $0\leq l\leq M=\lfloor T/\tau\rfloor$. This is the
time-dependent counterpart of \eqref{P}. We use $C$ for a generic
positive constant independent of $h$, $\tau$ and $m$, which may
initially depend on $\kappa_l$, $T$ and the smooth exact solution.
The notation $C_{\kappa_l}$ displays the mesh dependence when needed,
whereas $C_0$ is independent of $\kappa_l$. We write
$A\lesssim B$ if $A\leq CB$.

The identity
\eqref{eq:interpolation-of-distance-projection} shows that
$\hat X_{h,*}^m$ is the nodal interpolant of
$a^m\circ\hat X_{h,*}^m$ on $\Ghso$.
Lemma~\ref{lemma:Ih} and the chain rule therefore give
\begin{align*}
	\|a^m\circ\hat X_{h,*}^m-\hat X_{h,*}^m\|_{L^\infty(\Ghso)}
	%	\\
	%	&\qquad
	+h\|a^m\circ\hat X_{h,*}^m-\hat X_{h,*}^m
	\|_{W^{1,\infty}(\Ghso)}
	\leq C_{\kappa_l}h^2.
\end{align*}
Let $n^m$ and $H^m$ be the unit normal and curvature of $\Gm$,
with normal extensions
\[
\nsm:=n^m\circ a^m,
\qquad
H_*^m:=H^m\circ a^m
\]
to $U_\delta(\Gm)$. The edgewise unit outward normal $\nhsm$ of $\Ghsm$
satisfies the corresponding first-order approximation
\begin{align}\label{normal-intpl}
	\|\nhsm-\nsm\|_{L^\infty(\Ghsm)}
	\leq C_{\kappa_l}h.
\end{align}

\subsection{Norm equivalence on nearby polygonal curves}
\label{subsec:norm-equivalence}

Let $\mathcal T_1\cong\mathcal T_2$ with transition map $f_h$.
The intermediate polygonal meshes
\[
\mathcal T_{1,2}^\theta
:=
(1-\theta)\mathcal T_1+\theta\mathcal T_2,
\qquad 0\leq\theta\leq1,
\]
are parametrized over $D(\mathcal T_1)$ by
$f_h^\theta:=(1-\theta){\rm id}_{D(\mathcal T_1)}+\theta f_h$.
A fixed nodal vector determines a finite element function on
each of these meshes. The following comparison of its norms
is proved in \cite[Lemma 4.3]{KLL17} and
\cite[Lemma 7.2]{KLL19}.

\begin{lemma}\label{lemma:norm-equiv}
	Suppose that
	\[
	\|\nabla_{D(\mathcal T_1)}
	(f_h-{\rm id}_{D(\mathcal T_1)})\|_{L^\infty(D(\mathcal T_1))}
	\leq\frac12.
	\]
	Then, for $0\leq\theta\leq1$ and $1\leq p\leq\infty$,
	every finite element function $v_h$ with a common nodal
	vector satisfies
	\begin{align*}
		c^{-1}\|v_h\|_{L^p(D(\mathcal T_1))}
		&\leq
		\|v_h\|_{L^p(D(\mathcal T_{1,2}^\theta))}
		\leq
		c\|v_h\|_{L^p(D(\mathcal T_1))},
		\\
		c^{-1}\|\nabla_{D(\mathcal T_1)}v_h\|_{L^p(D(\mathcal T_1))}
		&\leq
		\|\nabla_{D(\mathcal T_{1,2}^\theta)}v_h
		\|_{L^p(D(\mathcal T_{1,2}^\theta))}
%		\\
%		&
		\leq
		c\|\nabla_{D(\mathcal T_1)}v_h\|_{L^p(D(\mathcal T_1))},
	\end{align*}
	where $c$ is a universal constant independent of
	$\theta$, $p$, $\mathcal T_1$ and $\mathcal T_2$.
\end{lemma}

%\subsection{Geometry of the consistency curve}
%\label{section:interpolated}

\subsection{Reversely weighted normals and the super-approximation of $\nbhsm$}
\label{subsec:averaged-normal-basic}

At a vertex $p$ of $\Ghsm$, denote the adjacent edges by $K_-$
and $K_+$ and their one-sided traces by $p-$ and $p+$.
Their lengths are written as
\[
\ell_{h,*,\pm}^m
:=
|w_{K_\pm}(p)|\,|K_{\pm,{\rm f}}^0|,
\qquad
w_{K_\pm}
:=
\nabla_{K_{\pm,{\rm f}}^0}F_{K_\pm}\circ F_{K_\pm}^{-1}.
\]
The finite element normal $\nbhsm\in[S_h(\Ghsm)]^2$ is specified
at the vertices by
\begin{align}
	\nbhsm(p)
	=
	\frac{
		\ell_{h,*,+}^m\,\nhsm(p-)
		+\ell_{h,*,-}^m\,\nhsm(p+)
	}{
		\ell_{h,*,-}^m+\ell_{h,*,+}^m
	}.
	\label{bar-n-hat-n}
\end{align}
Each one-sided normal is thus weighted by the length of the
opposite edge. On the numerical curve $\Ghm$, write $\nhm$
for the edgewise unit normal and $\ell_{h,\pm}^m$ for the
adjacent edge lengths. The analogous nodal definition is
\begin{align}
	\nbhm(p)
	=
	\frac{
		\ell_{h,+}^m\,\nhm(p-)
		+\ell_{h,-}^m\,\nhm(p+)
	}{
		\ell_{h,-}^m+\ell_{h,+}^m
	}.
	\label{eq:bar_n}
\end{align}
This is the averaged normal appearing in \eqref{eq:BGN-stab}.

The unit length of the edge normals gives the following nodal
bounds; cf. \cite[Eqs. (3.18) and (4.33)]{BL2024}:
\begin{subequations}
	\begin{align}
		|\nbhsm(p)|&\leq1,
		&
		\big||\nbhsm(p)|-1\big|
		&\leq C|\nhsm(p+)-\nhsm(p-)|^2
		\leq C_{\kappa_l}h^2,
		\label{eq:nhs_bar_len0}
		\\
		|\nbhm(p)|&\leq1,
		&
		\big||\nbhm(p)|-1\big|
		&\leq C|\nhm(p+)-\nhm(p-)|^2.
		\label{eq:nhs_bar_len}
	\end{align}
\end{subequations}

Normal vectors on the numerical and consistency curves are compared
using the intermediate curves of
Section~\ref{subsec:norm-equivalence}. Let $\ehm:=X_h^m - \hat X_{h,*}^m$ be the projection
error; also see the relevant definitions in Section~\ref{section:geometry}.
The fundamental theorem of calculus, Lemma~\ref{lemma:ud},
item~7, Lemma~\ref{lemma:lump}, and
Lemma~\ref{lemma:norm-equiv} yield
\begin{align}\label{normal-intpl-2}
	&\|\nhm-\nhsm\|_{L^p(\Ghsm)}
	+\|\nhm-\nhsm\|_{L_h^p(\Ghsm)}
%	\notag\\
%	&\qquad
	\lesssim
	\|\nabla_\Ghsm\ehm\|_{L^p(\Ghsm)}
	\qquad\forall p\in[1,\infty].
\end{align}
The edgewise constant function $\nhm$ is realized on $\Ghsm$
by the same element correspondence.
Applying this argument to \eqref{bar-n-hat-n} and
\eqref{eq:bar_n}, with the variation of the length weights included,
also gives
\begin{align}\label{normal-intpl-3}
	\|\nbhm-\nbhsm\|_{L^p(\Ghsm)}
	\lesssim
	\|\nabla_\Ghsm\ehm\|_{L^p(\Ghsm)}
	\qquad\forall p\in[1,\infty];
\end{align}
see \cite[Eq. (4.31)]{BL2025}.

We associate normal and tangential projections with these normals:
\[
\Nsm:=\nsm(\nsm)^\top,
\qquad
\Tsm:=I-\Nsm,
\qquad
\Nhsm:=\nhsm(\nhsm)^\top,
\qquad
\Thsm:=I-\Nhsm,
\]
and
\[
\bar N_{h,*}^m
:=\frac{\bar n_{h,*}^m}{|\bar n_{h,*}^m|}
\left(\frac{\bar n_{h,*}^m}{|\bar n_{h,*}^m|}\right)^\top,
\qquad
\bar T_{h,*}^m:=I-\bar N_{h,*}^m, 
\qquad
\bar N_{h}^m
:=\frac{\bar n_{h}^m}{|\bar n_{h}^m|}
\left(\frac{\bar n_{h}^m}{|\bar n_{h}^m|}\right)^\top,
\qquad
\bar T_{h}^m:=I-\bar N_{h}^m
.
\]
%\begin{align*}
%	\Nbhsm
%	&:=
%	\frac{\nbhsm}{|\nbhsm|}
%	\left(\frac{\nbhsm}{|\nbhsm|}\right)^\top,
%	&
%	\Tbhsm&:=I-\Nbhsm,
%	\\
%	\Nbhm
%	&:=
%	\frac{\nbhm}{|\nbhm|}
%	\left(\frac{\nbhm}{|\nbhm|}\right)^\top,
%	&
%	\Tbhm&:=I-\Nbhm.
%\end{align*}
The matrix $I-\nbhm(\nbhm)^\top$ used in
\eqref{eq:BGN-stab} differs from the normalized tangential
projector $\Tbhm$.

%\subsection{Super-approximation of $\nbhsm$}
%\label{subsec:edge-normal-expansion}

As proved in \cite[Lemma 2.3]{Bai2026}, the reversely weighted normals satisfy the following approximation
estimates.

\begin{lemma}\label{lemma:n_bar_app}
	For $0\leq m\leq l$, we have
	\begin{align}
		\|\nbhsm-I_h\nsm\|_{L^\infty(\Ghsm)}
		&\lesssim h^2,
		\notag\\
		\|\nbhm-I_h\nsm\|_{L^2(\Ghsm)}
		&\lesssim
		h^2+\|\nabla_\Ghsm\ehm\|_{L^2(\Ghsm)},
		\notag\\
		\|\nbhsm-\nhsm\|_{L^\infty(\Ghsm)}
		&\lesssim h.
		\label{eq:averaged-normal-approximation}
	\end{align}
\end{lemma}

%\begin{proof}
%	See \cite[Lemma 2.3]{Bai2026}.
%\end{proof}

\subsection{Some geometric identities}
\label{section:geometry}

In the remainder of this section, we will identify all finite element functions as elements either in $S_h(\Ghsm)$ or $[S_h(\Ghsm)]^2$, via the canonical nodal vector identification described in Section \ref{subsec:parametric-finite-elements}.
We define $X^{m+1}:\Gamma^0\rightarrow\Gamma^{m+1}$ and $Y^{m+1}:= X^{m+1}\circ (X^m)^{-1} :\Gamma^m\rightarrow\Gamma^{m+1}$ to be the exact global and local flow maps along $-H(t)n(t)$, respectively.
Let the interface finite element function $X_{h,*}^{m+1}:\Ghsm\rightarrow \Gamma_{h,*}^{m+1}$ be the interpolation of the local flow, which is uniquely determined by the relation $$X_{h,*}^{m+1}(p) - \hat X_{h,*}^m(p) = Y^{m+1}(p) - {\rm id_\Gm}(p)\quad\forall p\in\mathcal N(\Ghsm)\subset\Gm .$$
Then, it follows that 
\begin{align}
	&X_{h,*}^{m+1} - \hat X_{h,*}^m =I_h \big((Y^{m+1} - {\rm id_\Gm}) \circ a^m|_\Ghsm \big) &&\mbox{on}\,\,\, \Ghsm , \label{eq:X-id1} \\
	&Y^{m+1} - {\rm id}_\Gm = \tau ( {{-H^m n^m }}  +  g^m )  &&\mbox{on}\,\,\,  \Gamma^m, \label{eq:X-id2}
\end{align}
where $g^m$ is a smooth correction from the Taylor expansion, satisfying the following $W^{1,\infty}$ estimate: 
\begin{align}\label{W1infty-g}
	\|g^m\|_{W^{1,\infty}(\Gamma^m)}\le C\tau . 
\end{align}

%The projection error at the time level $m$ is defined as $\ehm := {\rm id}|_\Ghm -  \hat X_{h,*}^m$ where $\hat X_{h,*}^m$ is the interface finite element function whose nodal values coincide with the projected nodes of $\Gamma_h^m$ onto $\Gm$ via $\Nsm$.
%Let $\Ghsm$ be the discrete surface whose parametrization map is $\hat X_{h,*}^m$.

The local trajectory error and the projection error at time level $m$ are defined as $e_h^m := X_h^m -  X_{h,*}^m$ and $\ehm := X_h^m -  \hat X_{h,*}^m$, respectively.
According to \cite[Eqs. (3.12)--(3.13)]{BL2024}, we have the following nodal relation
\begin{align}\label{eq:geo_rel_1}
	\ehm = I_h\big[ (e_h^m\cdot \nsm)\nsm \big] + r_h^m ,
\end{align}
with $r_h^m := \ehm - I_h\big[ (e_h^m\cdot \nsm)\nsm \big]$ satisfying
\begin{align}\label{eq:geo_rel_2}
	|r_h^m| \lesssim |[I - \nsm (\nsm)^\top ] e_h^m|^2
	\quad\mbox{at the nodes of $\Ghsm$}.	
\end{align}
$r_h^m$ can be interpreted as a quadratic remainder of the nodal orthogonal projection due to the presence of curvature.

Therefore, we deduce from \eqref{eq:X-id1}, \eqref{eq:X-id2} and \eqref{eq:geo_rel_1} that
\begin{align}\label{eq:geo_rel_3}
	\begin{aligned}
		X_h^{m+1} - X_h^m 
		&= \eM - \ehm + X_{h,*}^{m+1} - \hat X_{h,*}^m \\
		&= \eM - \ehm + \tau I_h\big((-H^m n^m + g^m) \circ a^m|_\Ghsm \big) ,
	\end{aligned}
\end{align}
at the finite element nodes in $\mathcal N(\Ghsm)$.
%This relation helps us convert the numerical displacement $X_h^{m+1} - X_h^m$ to the error displacement $\eM - \ehm$.
Let $(v , \kappa)$ be the unique smooth solution to the following elliptic system on the smooth curve $\Gamma=\Gamma(t)$: 
\begin{subequations}\label{subeq:system}
	\begin{align}
		v\cdot n&=-H \label{eq:ell_sys_v1} , \\
		-\Delta_\Gamma v&=\kappa n  . \label{eq:ell_sys_v2}
	\end{align}
\end{subequations}
Denote $\delta_\tau X_h^m:=\big(X_h^{m+1}-X_h^m-\tau I_h[(v^m) \circ a^m|_\Ghsm ]\big)/\tau$ and $\delta_\tau \ehm:=(\eM-\ehm-\tau I_h\Tsm[(v^m) \circ a^m|_\Ghsm ])/\tau$ with $v^m:=v(t_m)$. Then \eqref{eq:geo_rel_3} becomes
\begin{align}\label{eq:geo_rel_31}
	\delta_\tau X_h^m = \delta_\tau \ehm +  I_h\big(g^m \circ a^m|_\Ghsm \big) .
\end{align}

The following geometric identities related to the projection error are well-known (cf. \cite[Eqs. (A.15)--(A.17)]{BL2024}): 
If we define $\rho_h^m := I_h(\Nsm (\hat X_{h,*}^{m + 1} -\hat X_{h,*}^{m})) - I_h((Y^{m + 1} - {\rm id}_\Gm)\circ a^m|_\Ghsm)\in [S_h(\Ghsm)]^2$, then at all finite element nodes in $\mathcal N(\Ghsm)$, it holds that
\begin{align}
	\Nsm (\hat X_{h,*}^{m + 1} -\hat X_{h,*}^{m})
	&= (Y^{m + 1} - {\rm id}_\Gm)\circ a^m|_\Ghsm + \rho_h^m  ,
	%	&& \mbox{at the nodes} ,
	\label{eq:geo_rel_4}
	\\
	\mbox{where}\,\,\, | \rho_h^m| 
	&\le C_0 \tau^2 + C_0 |T_*^m (\hat X_{h,*}^{m + 1} -\hat X_{h,*}^{m})|^2
	%	&& \mbox{at the nodes}
	, \label{eq:geo_rel_5}\\
	T_*^m (\hat X_{h,*}^{m + 1} -\hat X_{h,*}^{m})
	&= T_*^m (X_{h}^{m + 1} - X_{h}^{m})
	- T_*^m (N_*^{m+1}\circ\hat X_{h,*}^{m+1}-N_*^{m}\circ\hat X_{h,*}^{m})  \ehM
	%	&& \mbox{at the nodes} 
	. \label{eq:geo_rel_6}
\end{align}

The nodal orthogonality of $\hat e_h^m$ to $\Gamma^m$ implies that its tangential component away from the nodes is of higher order. This can be quantified via the super-approximation estimates in Lemma~\ref{lemma:T<=N2}.

\subsection{The main theorem and the induction hypothesis}
\label{subsec:main-result-induction}

For the initial mesh $\mathcal T_h^0$ of $\Ghso$, we assume
shape regularity
\begin{align}\label{P0}
	\kappa_0\leq C_{\rm sh}
\end{align}
and quasi-uniformity: there are constants $c_{\rm q},C_{\rm q}>0$,
independent of $h$, such that
\begin{equation}
	c_{\rm q}h\leq|K|\leq C_{\rm q}h
	\qquad\forall K\in\mathcal T_h^0.
	\label{P1}
\end{equation}
The convergence result is stated in terms of the projection error
and the error velocity defined in Section~\ref{section:geometry}.

\begin{theorem}[Super-convergence of the stabilized BGN method]
	\label{thm:main}
	Let $X:\Gamma^0\times[0,T]\rightarrow\mathbb R^2$ be the
	flow map of the curve-shortening flow. Assume that
	$X(\cdot,t)$ and its inverse
	$X(\cdot,t)^{-1}:\Gamma(t)\rightarrow\Gamma^0$
	are sufficiently smooth, uniformly for $t\in[0,T]$.
	Suppose that the initial discrete curve $\Gamma_h^0$
	satisfies \eqref{P0}--\eqref{P1} and obeys
	\[
	\|\hat e_h^0\|_{L^2(\hat\Gamma_{h,*}^0)}=0 .
	\]
	Let $X_h^m$ be the finite element solution of
	\eqref{eq:BGN-stab}, initialized by
	$X_h^0={\rm id}|_{\Gamma_h^0}$ on $\Gamma_h^0$.
	
	For any fixed constant $0<c_0$, there exists $h_0>0$
	such that
	\[
	\tau\leq c_0h^2,
	\qquad h\leq h_0,
	\]
	implies
	\begin{align}
		\label{eq:err_est_1}
		&\max_{0\leq m\leq\lfloor T/\tau\rfloor}
		\|\hat e_h^m\|_{H^1(\Ghsm)}^2
		+
		\sum\limits_{m=0}^{\lfloor T/\tau\rfloor-1}
		\tau\|\delta_\tau\ehm\|_{L^2(\Ghsm)}^2
		\leq C(\tau+h^2)^2.
	\end{align}
	The constant $C$ is independent of $h$ and $\tau$, but may
	depend on $c_0,c_{\rm q},C_{\rm q},C_{\rm sh},T$
	and the exact solution.
\end{theorem}

%\begin{remark}\label{rmk:thm}\upshape
%	The condition $\tau\leq c_0h^2$ enters the proof of uniform
%	mesh regularity in Section~\ref{sec:bbd}.
%\end{remark}

The projection-error estimate also yields optimal $L^2$ convergence
of the particle trajectories associated with the comparison velocity
$v$ in \eqref{subeq:system}; see Section \ref{Proof-THM-2} for the proof.
For each initial node $x_j^0\in\mathcal N(\Gamma_h^0)$,
let $x_{j,\#}(t)$ be the particle trajectory satisfying
\begin{align*}
	\partial_t x_{j,\#}(t)
	&=v\bigl(t,x_{j,\#}(t)\bigr),
	\qquad t\in[0,T],
	\\
	x_{j,\#}(0)&=x_j^0,
	\qquad j=1,\ldots,J.
\end{align*}
Since $v\cdot n=-H$, these trajectories satisfy
$x_{j,\#}(t)\in\Gamma(t)$.
Define $X_{h,\#}^m\in[S_h(\Gamma_h^0)]^2$ by
\[
X_{h,\#}^m(x_j^0)=x_{j,\#}(t_m),
\qquad j=1,\ldots,J,
\]
and define the trajectory error on the initial reference curve by
\[
e_{h,\#}^m:=X_{h,\#}^m-X_h^m.
\]
Then, we have the following convergence result for the trajectory error $e_{h,\#}^m$.
\begin{theorem}[Convergence of the trajectory error]
	\label{thm:trajectory}
	Under the assumptions of Theorem~\ref{thm:main}, the trajectory error converges optimally under the $L_t^\infty L_x^2$ norm:
	\begin{align}
		\max_{0\leq m\leq\lfloor T/\tau\rfloor}
		\|e_{h,\#}^m\|_{L^2(\Ghsm)}
		\leq C(\tau+h^2). \notag
%		\label{eq:trajectory-error}
	\end{align}
	The constant $C$ is independent of $h$ and $\tau$, with the
	same allowed dependencies as in Theorem~\ref{thm:main}.
\end{theorem}

%The proof of Theorem~\ref{thm:trajectory} is given in
%Section~\ref{Proof-THM-2}.

The continuation argument proceeds as follows. Set
$M:=\lfloor T/\tau\rfloor$ and fix $l\in\{0,\ldots,M-1\}$.
For every $m=0,\ldots,l$, assume
\begin{align}
	&\|\ehm\|_{L^2(\Ghsm)}
	+\|e_h^m\|_{L^2(\Ghsm)}
	+\big(
	\|\ehm\|_{L^\infty(\Ghsm)}
	+\|e_h^m\|_{L^\infty(\Ghsm)}
	\big)
	\notag\\
	&\quad
	+\big(
	\|\ehm\|_{H^1(\Ghsm)}
	+\|e_h^m\|_{H^1(\Ghsm)}
	\big)
	+h^{1/2}\big(
	\|\ehm\|_{W^{1,\infty}(\Ghsm)}
	+\|e_h^m\|_{W^{1,\infty}(\Ghsm)}
	\big)
	\leq h^{7/4}.
	\label{eq:ind_hypo1}
\end{align}
Here $e_h^0:=\hat e_h^0$, and $e_h^m$ is realized on $\Ghsm$
through the canonical nodal identification.
Under this hypothesis, we prove \eqref{eq:err_est_1} and recover
\eqref{eq:ind_hypo1} at the next time level $m=l+1$.

In particular, \eqref{eq:ind_hypo1} gives
\[
\|\nabla_\Ghsm\ehm\|_{L^\infty(\Ghsm)}\leq h^{5/4}.
\]
The intermediate curves between $\Ghsm$ and $\Ghm$ are
\[
\hat\Gamma_{h,\theta}^m
:=
({\rm id}|_\Ghsm+\theta\ehm)(\Ghsm)
%=
%(1-\theta)\Ghsm+\theta\Ghm
,
\]
with transition maps
$({\rm id}|_\Ghsm+\theta\ehm)=(1-\theta)\hat X_{h,*}^m+\theta X_h^m=: \hat X_{h,\theta}^m :\Ghsm\rightarrow\hat\Gamma_{h,\theta}^m$.
For sufficiently small $h$, Lemma~\ref{lemma:norm-equiv}
therefore makes the $L^p$ and $W^{1,p}$ norms of functions with
a common nodal vector uniformly equivalent on
$\Ghsm$, $\hat\Gamma_{h,\theta}^m$ and $\Ghm$.

The same induction hypothesis, together with
Lemma~\ref{lemma:n_bar_app}, the smoothness of $n_*^m$ and
inverse estimates, yields
\begin{align}\label{eq:n_bar_bbd}
	\begin{split}
		\|\nbhsm\|_{W^{1,\infty}(\Ghsm)}
		&\lesssim1+h\lesssim1,
		\\
		\|\nbhm\|_{L^\infty(\Ghsm)}
		&\lesssim
		1+h^2+h^{-1/2}\|\nabla_\Ghsm\ehm\|_{L^2(\Ghsm)}
		\lesssim1,
		\\
		\|\nbhm\|_{W^{1,\infty}(\Ghsm)}
		&\lesssim
		1+h+h^{-3/2}\|\nabla_\Ghsm\ehm\|_{L^2(\Ghsm)}
		\lesssim1
%		+h^{1/4}
		.
	\end{split}
\end{align}

\section{Consistency analysis}\label{sec:cons_err}

%We insert the interpolated exact one-step flow into the discrete scheme
%and separate the contributions from time discretization, mass lumping,
%geometric perturbation and stabilization. We follow the decomposition
%in \cite{Bai2026}, with the modification required by the normal mass
%bilinear form.

\subsection{Geometry perturbation estimates}

We recall the standard geometric perturbation estimates; see
\cite[Lemma~3.1]{Bai2026} and \cite[Lemma~5.6]{Kov18}.

\begin{lemma}\label{lemma:geo-pert}
	Given $1/p+1/q=1$, the following estimates hold:
	\begin{align*}
		\Big|\int_{\Ghsm} f_1f_2-\int_{\Gm}f_1^\ell f_2^\ell\Big|
		&\le C_{\kappa_l}h^2
		\|f_1\|_{L^p(\Ghsm)}\|f_2\|_{L^q(\Ghsm)},\\
		\Big|\int_{\Ghsm}\nabla_{\Ghsm}f_1\cdot\nabla_{\Ghsm}f_2
		-\int_{\Gm}\nabla_{\Gm}f_1^\ell\cdot\nabla_{\Gm}f_2^\ell\Big|
		&\le C_{\kappa_l}h^2
		\|\nabla_{\Ghsm}f_1\|_{L^p(\Ghsm)}
		\|\nabla_{\Ghsm}f_2\|_{L^q(\Ghsm)},\\
		\Big|\int_{\Ghsm}\nabla_{\Ghsm}f_1\cdot f_2
		-\int_{\Gm}\nabla_{\Gm}f_1^\ell\cdot f_2^\ell\Big|
		&\le C_{\kappa_l}h
		\|\nabla_{\Ghsm}f_1\|_{L^p(\Ghsm)}
		\|f_2\|_{L^q(\Ghsm)}.
	\end{align*}
\end{lemma}

\subsection{Definition of consistency errors}

The consistency error associated with \eqref{eq:BGN-stab} is the linear
functional on $[S_h(\Ghsm)]^2$ defined by
\begin{align}\label{eq:cons_eq}
	\mathscr D^m(\phi_h)
	&:=\int_{\Ghsm}^h
	\frac{X_{h,*}^{m+1}-{\rm id}}{\tau}
	\cdot\nbhsm\,\phi_h\cdot\nbhsm
	+\int_{\Ghsm}\nabla_{\Ghsm}X_{h,*}^{m+1}
	\cdot\nabla_{\Ghsm}\phi_h
	\notag\\
	&\quad-\int_{\Ghsm}\nabla_{\Ghsm}\hat X_{h,*}^m
	\cdot\nabla_{\Ghsm}I_h[(I-\nbhsm(\nbhsm)^\top)\phi_h]
	\notag\\
	&=\int_{\Ghsm}^h
	\frac{X_{h,*}^{m+1}-{\rm id}}{\tau}
	\cdot\nbhsm\,\phi_h\cdot\nbhsm
	+\int_{\Gm}\Hm\nm\cdot\phi_h^\ell
	\notag\\
	&\quad-\int_{\Gm}\nabla_{\Gm}{\rm id}
	\cdot\nabla_{\Gm}\phi_h^\ell
	+\int_{\Ghsm}\nabla_{\Ghsm}X_{h,*}^{m+1}
	\cdot\nabla_{\Ghsm}\phi_h
	\notag\\
	&\quad-\int_{\Ghsm}\nabla_{\Ghsm}\hat X_{h,*}^m
	\cdot\nabla_{\Ghsm}I_h[(I-\nbhsm(\nbhsm)^\top)\phi_h]
	\notag\\
	&=:\mathscr D_1^m(\phi_h)+\mathscr D_2^m(\phi_h)
	+\mathscr D_3^m(\phi_h),
\end{align}
where we have used
\[
\int_{\Gm}\nabla_{\Gm}{\rm id}\cdot\nabla_{\Gm}\phi_h^\ell
=\int_{\Gm}\Hm\nm\cdot\phi_h^\ell.
\]

\subsection{Estimates for consistency errors}

\begin{lemma}\label{proposition-consistency}
	The consistency error $\mathscr D^m(\cdot)$ satisfies
	\begin{align*}
		|\mathscr D^m(\phi_h)|
		&\lesssim
		\tau\|\phi_h\|_{L^2(\Ghsm)}
		+h^2\|\phi_h\|_{H^1(\Ghsm)}
		\quad\forall\,\phi_h\in[S_h(\Ghsm)]^2.
	\end{align*}
\end{lemma}

\begin{proof}
	Using \eqref{eq:X-id1}--\eqref{eq:X-id2}, we decompose the mass
	contribution into five terms:
	\begin{align*}
		\mathscr D_1^m(\phi_h)
		&=\int_{\Ghsm}^h
		\frac{X_{h,*}^{m+1}-{\rm id}}{\tau}
		\cdot\nbhsm\,\phi_h\cdot\nbhsm
		+\int_{\Gm}\Hm\nm\cdot\phi_h^\ell\\
		&=\int_{\Ghsm}^h I_h(g^m)^{-\ell}
		\cdot\nbhsm\,\phi_h\cdot\nbhsm\\
		&\quad-\int_{\Ghsm}^h
		\big(I_h(H^mn^m)^{-\ell}-(H^mn^m)^{-\ell}\big)
		\cdot\phi_h\\
		&\quad-\Big(\int_{\Ghsm}^h-\int_{\Ghsm}\Big)
		(H^mn^m)^{-\ell}\cdot\phi_h\\
		&\quad-\int_{\Ghsm}(H^mn^m)^{-\ell}\cdot\phi_h
		+\int_{\Gm}H^mn^m\cdot\phi_h^\ell\\
		&\quad+\int_{\Ghsm}^h
		(H^mn^m)^{-\ell}\cdot
		[I-\nbhsm(\nbhsm)^\top]\phi_h\\
		&=:\sum_{i=1}^5\mathscr D_{1i}^m(\phi_h).
	\end{align*}
	Here
	$\big(\int_{\Ghsm}^h-\int_{\Ghsm}\big)f
	=\int_{\Ghsm}^h f-\int_{\Ghsm}f$
	denotes the quadrature error.
	
	By \eqref{W1infty-g}, the boundedness of $\nbhsm$ and discrete
	norm equivalence,
	\begin{align*}
		|\mathscr D_{11}^m(\phi_h)|
		&\lesssim
		\|g^m\|_{L^\infty(\Gm)}
		\|\phi_h\|_{L^2(\Ghsm)}
		\lesssim\tau\|\phi_h\|_{L^2(\Ghsm)}.
	\end{align*}
	The integrand defining $\mathscr D_{12}^m$ vanishes at every
	quadrature node, and hence
	\[
	\mathscr D_{12}^m(\phi_h)=0.
	\]
	The quadrature estimate in Lemma~\ref{lemma:super_conv2}, together
	with the piecewise linearity of $\phi_h$, gives
	\begin{align*}
		|\mathscr D_{13}^m(\phi_h)|
		\lesssim
		h^2\|(H^mn^m)^{-\ell}\|_{H_h^2(\Ghsm)}
		\|\phi_h\|_{H^1(\Ghsm)}
%		\\
%		&
		\lesssim h^2\|\phi_h\|_{H^1(\Ghsm)}.
	\end{align*}
	Lemma~\ref{lemma:geo-pert} and norm equivalence yield
	\begin{align*}
		|\mathscr D_{14}^m(\phi_h)|
		&\lesssim
		h^2\|(H^mn^m)^{-\ell}\|_{L^2(\Ghsm)}
		\|\phi_h\|_{L^2(\Ghsm)}
		\lesssim h^2\|\phi_h\|_{L^2(\Ghsm)}.
	\end{align*}
	For the last contribution, the super-approximation
	in Lemma~\ref{lemma:n_bar_app} implies
	\[
	\big|[I-\nbhsm(\nbhsm)^\top]\nsm\big|
	\lesssim|\nbhsm-\nsm|
	\lesssim h^2
	\quad\mbox{at the nodes of $\Ghsm$},
	\]
	and, subsequently,
	\begin{align}
		|\mathscr D_{15}^m(\phi_h)|
		\lesssim h^2\|\phi_h\|_{L^2(\Ghsm)}.
	\end{align}
	Combining the five bounds gives
	\begin{align*}
		|\mathscr D_1^m(\phi_h)|
		\lesssim
		\tau\|\phi_h\|_{L^2(\Ghsm)}
		+h^2\|\phi_h\|_{H^1(\Ghsm)}.
	\end{align*}
	
	Next, the stiffness contribution is decomposed as
	\begin{align*}
		\mathscr D_2^m(\phi_h)
		&=\int_{\Ghsm}\nabla_{\Ghsm}X_{h,*}^{m+1}
		\cdot\nabla_{\Ghsm}\phi_h
		-\int_{\Gm}\nabla_{\Gm}{\rm id}
		\cdot\nabla_{\Gm}\phi_h^\ell\\
		&=\int_{\Ghsm}
		\nabla_{\Ghsm}(X_{h,*}^{m+1}-\hat X_{h,*}^m)
		\cdot\nabla_{\Ghsm}\phi_h\\
		&\quad+\int_{\Ghsm}\nabla_{\Ghsm}\hat X_{h,*}^m
		\cdot\nabla_{\Ghsm}\phi_h
		-\int_{\Gm}\nabla_{\Gm}(\hat X_{h,*}^m)^\ell
		\cdot\nabla_{\Gm}\phi_h^\ell\\
		&\quad+\int_{\Gm}
		\nabla_{\Gm}[(I_ha^m|_\Ghsm)^\ell-a^m]
		\cdot\nabla_{\Gm}\phi_h^\ell\\
		&=:\mathscr D_{21}^m(\phi_h)
		+\mathscr D_{22}^m(\phi_h)
		+\mathscr D_{23}^m(\phi_h),
	\end{align*}
	where we have used
	\[
	(\hat X_{h,*}^m)^\ell=(I_ha^m|_\Ghsm)^\ell,
	\qquad
	{\rm id}=a^m\quad\mbox{on $\Gm$}.
	\]
	
	For $\mathscr D_{21}^m$, we use the argument of
	\cite[Lemma~4.3]{BL2024}, and get the bound
%	.
%	By \eqref{eq:X-id1}, the displacement
%	$X_{h,*}^{m+1}-\hat X_{h,*}^m$ is the nodal interpolant of
%	$(Y^{m+1}-{\rm id}_{\Gm})^{-\ell}$.
%	On each straight element, $\nabla_{\Ghsm}\phi_h$ is constant,
%	while the interpolation error vanishes at the endpoints.
%	Elementwise integration by parts therefore gives
	\begin{align*}
		\mathscr D_{21}^m(\phi_h)
		&=\int_{\Ghsm}
		\nabla_{\Ghsm}(Y^{m+1}-{\rm id}_{\Gm})^{-\ell}
		\cdot\nabla_{\Ghsm}\phi_h.
	\end{align*}
	Smoothness of the exact flow ensures
	$\|Y^{m+1}-{\rm id}_{\Gm}\|_{H^2(\Gm)}\lesssim\tau$.
	Thus, Lemma~\ref{lemma:geo-pert}, integration by parts on $\Gm$
	and the inverse inequality imply
	\begin{align*}
		|\mathscr D_{21}^m(\phi_h)|
%		&\lesssim
%		\Big|\int_{\Gm}
%		\nabla_{\Gm}(Y^{m+1}-{\rm id}_{\Gm})
%		\cdot\nabla_{\Gm}\phi_h^\ell\Big|
%		+\tau h^2\|\phi_h\|_{H^1(\Ghsm)}\\
%		&\lesssim
%		\tau\|\phi_h\|_{L^2(\Ghsm)}
%		+\tau h^2\|\phi_h\|_{H^1(\Ghsm)}\\
%		&
		\lesssim\tau\|\phi_h\|_{L^2(\Ghsm)}.
	\end{align*}
	Another application of Lemma~\ref{lemma:geo-pert} yields
	\begin{align*}
		|\mathscr D_{22}^m(\phi_h)|
		\lesssim
		h^2\|\nabla_{\Ghsm}\hat X_{h,*}^m\|_{L^2(\Ghsm)}
		\|\nabla_{\Ghsm}\phi_h\|_{L^2(\Ghsm)}
%		\\
%		&
		\lesssim h^2\|\phi_h\|_{H^1(\Ghsm)}.
	\end{align*}
	The super-approximation estimate in Lemma~\ref{Lemma-GLW},
	applied componentwise to ${\rm id}_{\Gm}$, gives
	\begin{align*}
		|\mathscr D_{23}^m(\phi_h)|
		&\lesssim
		h^2\|{\rm id}_{\Gm}\|_{H^2(\Gm)}
		\|\phi_h\|_{H^1(\Ghsm)}
		\lesssim h^2\|\phi_h\|_{H^1(\Ghsm)}.
	\end{align*}
	Consequently,
	\begin{align*}
		|\mathscr D_2^m(\phi_h)|
		\lesssim
		\tau\|\phi_h\|_{L^2(\Ghsm)}
		+h^2\|\phi_h\|_{H^1(\Ghsm)}.
	\end{align*}
	
	It remains to estimate the stabilization contribution. As in the
	proof of \cite[Lemma~3.2]{Bai2026}, we write
	\begin{align*}
		\mathscr D_3^m(\phi_h)
		&=-\int_{\Ghsm}\nabla_{\Ghsm}I_ha^m
		\cdot\nabla_{\Ghsm}I_h[(I-\nbhsm(\nbhsm)^\top)\phi_h]\\
		&=-\int_{\Gm}\nabla_{\Gm}{\rm id}_{\Gm}
		\cdot\nabla_{\Gm}
		\big(I_h[(I-\nbhsm(\nbhsm)^\top)\phi_h]\big)^\ell\\
		&\quad-\Big(
		\int_{\Ghsm}\nabla_{\Ghsm}({\rm id}_{\Gm})^{-\ell}
		\cdot\nabla_{\Ghsm}I_h[(I-\nbhsm(\nbhsm)^\top)\phi_h]\\
		&\hspace{4.5em}
		-\int_{\Gm}\nabla_{\Gm}{\rm id}_{\Gm}
		\cdot\nabla_{\Gm}
		\big(I_h[(I-\nbhsm(\nbhsm)^\top)\phi_h]\big)^\ell
		\Big)\\
		&\quad+\int_{\Ghsm}\nabla_{\Ghsm}(1-I_h)a^m
		\cdot\nabla_{\Ghsm}I_h[(I-\nbhsm(\nbhsm)^\top)\phi_h]\\
		&=:\mathscr D_{31}^m(\phi_h)
		+\mathscr D_{32}^m(\phi_h)
		+\mathscr D_{33}^m(\phi_h).
	\end{align*}
	Integration by parts on $\Gm$ gives
	\begin{align*}
		|\mathscr D_{31}^m(\phi_h)|
		&=\Big|\int_{\Gm}H^mn^m\cdot
		\big(I_h[(I-\nbhsm(\nbhsm)^\top)\phi_h]\big)^\ell\Big|\\
		&=\Big|
		\int_{\Gm}H^mn^m\cdot
		\big(I_h[(\Tbhsm-\Tsm)\phi_h]\big)^\ell\\
		&\qquad-\int_{\Gm}H^mn^m\cdot
		[(1-I_h)(\Tsm\phi_h)]^\ell\\
		&\qquad+\int_{\Gm}H^mn^m\cdot
		\Big(I_h\Big[
		\Big(\frac{\nbhsm(\nbhsm)^\top}{|\nbhsm|^2}
		-\nbhsm(\nbhsm)^\top\Big)\phi_h
		\Big]\Big)^\ell
		\Big|.
	\end{align*}
	Here the second equality uses
	$n^m\cdot(\Tsm\phi_h)^\ell=0$.
	Lemma~\ref{lemma:n_bar_app} and
	\eqref{eq:nhs_bar_len0} bound the first and third terms by
	$Ch^2\|\phi_h\|_{L^2(\Ghsm)}$.
	For the second term, the standard interpolation estimates (Lemma \ref{lemma:Ih}) give
	\begin{align*}
		\|(1-I_h)(\Tsm\phi_h)\|_{L^2(\Ghsm)}
		\lesssim h^2\|\phi_h\|_{H^1(\Ghsm)}.
	\end{align*}
	Therefore,
	\begin{align*}
		|\mathscr D_{31}^m(\phi_h)|
		\lesssim h^2\|\phi_h\|_{H^1(\Ghsm)}.
	\end{align*}
	
	By Lemma~\ref{lemma:geo-pert} and the boundedness
	\eqref{eq:n_bar_bbd},
	\begin{align*}
		|\mathscr D_{32}^m(\phi_h)|
		\lesssim
		h^2\|\nabla_{\Ghsm}
		I_h[(I-\nbhsm(\nbhsm)^\top)\phi_h]\|_{L^2(\Ghsm)}
%		\\
%		&
		\lesssim h^2\|\phi_h\|_{H^1(\Ghsm)}.
	\end{align*}
	Finally, $(1-I_h)a^m$ vanishes at every element endpoint, whereas
	the gradient of
	$I_h[(I-\nbhsm(\nbhsm)^\top)\phi_h]$
	is constant on each straight element.
	Elementwise integration by parts thus gives
	\[
	\mathscr D_{33}^m(\phi_h)=0.
	\]
	Combining these three estimates yields
	\begin{align*}
		|\mathscr D_3^m(\phi_h)|
		\lesssim h^2\|\phi_h\|_{H^1(\Ghsm)}.
	\end{align*}
	The proof is complete upon collecting the bounds for
	$\mathscr D_1^m$, $\mathscr D_2^m$ and $\mathscr D_3^m$.
\end{proof}

For a nodally normal test function, the stabilization contribution admits
a stronger estimate. The nodal projection identity gives
\begin{align}
	I_h[(I-\nbhsm(\nbhsm)^\top)I_h\Nbhsm\phi_h]^\ell
	=
	I_h[(1-|\nbhsm|^2)I_h\Nbhsm\phi_h]^\ell.
\end{align}
Since $\hat X_{h,*}^m={\rm id}$ on $\Ghsm$, elementwise integration by
parts implies
\begin{align*}
	\mathscr D_3^m(I_h\Nbhsm\phi_h)
	&
	=
	-\int_{\Ghsm}\nabla_{\Ghsm}\hat X_{h,*}^m
	\cdot\nabla_{\Ghsm}
	I_h[(1-|\nbhsm|^2)I_h\Nbhsm\phi_h]\\
	&\quad=
	-\sum_{p\in\mathcal N(\Ghsm)}
	\big(\hat\mu_{h,*}^m(p+)+\hat\mu_{h,*}^m(p-)\big)
	\cdot
	\big[(1-|\nbhsm(p)|^2)(I_h\Nbhsm\phi_h)(p)\big],
\end{align*}
where $\hat\mu_{h,*}^m(p\pm)$ are the outward co-normals of the adjacent
edges. Smoothness of the exact curve and
\eqref{eq:nhs_bar_len0} give
\[
|\hat\mu_{h,*}^m(p+)+\hat\mu_{h,*}^m(p-)|
\lesssim h,
\qquad
|1-|\nbhsm(p)|^2|\lesssim h^2.
\]
Using quasi-uniformity and discrete norm equivalence, we conclude that
\begin{align}
	|\mathscr D_3^m(I_h\Nbhsm\phi_h)|
	\lesssim
	h^3\sum_{p\in\mathcal N(\Ghsm)}
	|(I_h\Nbhsm\phi_h)(p)|
%	\notag\\
%	&
	\lesssim
	h^2\|I_h\Nbhsm\phi_h\|_{L^2(\Ghsm)}.
\end{align}
%This is also the consistency-side stabilization estimate in
%\cite[Eq.~(4.17)]{Bai2026} applied to the nodally normal test function.

We define the critical part by
\[
\tilde{\mathscr D}^m(\phi_h)
:=\mathscr D_{13}^m(\phi_h)
+\mathscr D_{22}^m(\phi_h)
+\mathscr D_{23}^m(\phi_h).
\]
The preceding bounds then imply
\begin{align}
	|\mathscr D^m(I_h\Nbhsm\phi_h)
	-\tilde{\mathscr D}^m(I_h\Nbhsm\phi_h)|
	&\lesssim
	(\tau+h^2)\|I_h\Nbhsm\phi_h\|_{L^2(\Ghsm)},
	\label{eq:Dm-tildeDm}\\
	|\tilde{\mathscr D}^m(\phi_h)|
	&\lesssim h^2\|\phi_h\|_{H^1(\Ghsm)}.
	\label{eq:tildeDm}
\end{align}

\section{Error equation and the estimates for linear and bilinear forms}
\label{Proof-THM-1}

%Throughout this section, we assume the induction hypothesis
%\eqref{eq:ind_hypo1} and identify finite element functions on different
%curves through their nodal values. The geometric decompositions and
%estimates shared with the stabilized Dziuk method are recalled from
%\cite{Bai2026}; we give the necessary adaptations for the normal mass
%bilinear form and the corrected error velocity.

\subsection{The error equation}\label{sec:J_stab}

Subtracting the consistency relation \eqref{eq:cons_eq} from the numerical
scheme \eqref{eq:BGN-stab} gives, for any $\phi_h\in[S_h(\Ghsm)]^2$,
\begin{align}\label{eq:err_eq1}
	&\int_{\Ghm}^h \frac{X_h^{m+1}-X_h^m}{\tau}
	\cdot\nbhm\,\nbhm\cdot\phi_h
	-\int_{\Ghsm}^h \frac{X_{h,*}^{m+1}-\hat X_{h,*}^m}{\tau}
	\cdot\nbhsm\,\nbhsm\cdot\phi_h
	\notag\\
	&\quad+\int_{\Ghm}\nabla_{\Ghm}X_h^{m+1}
	\cdot\nabla_{\Ghm}\phi_h
	-\int_{\Ghsm}\nabla_{\Ghsm}X_{h,*}^{m+1}
	\cdot\nabla_{\Ghsm}\phi_h
	\notag\\
	&\quad-\int_{\Ghm}\nabla_{\Ghm}X_h^m
	\cdot\nabla_{\Ghm}I_h[(I-\nbhm(\nbhm)^\top)\phi_h]
	\notag\\
	&\quad+\int_{\Ghsm}\nabla_{\Ghsm}\hat X_{h,*}^m
	\cdot\nabla_{\Ghsm}I_h[(I-\nbhsm(\nbhsm)^\top)\phi_h]
	=-\mathscr D^m(\phi_h).
\end{align}
The difference of the mass terms is
\begin{align}\label{eq:mass_diff}
	&\int_{\Ghm}^h \frac{X_h^{m+1}-X_h^m}{\tau}
	\cdot\nbhm\,\nbhm\cdot\phi_h
	-\int_{\Ghsm}^h \frac{X_{h,*}^{m+1}-\hat X_{h,*}^m}{\tau}
	\cdot\nbhsm\,\nbhsm\cdot\phi_h
	\notag\\
	&\qquad=
	\int_{\Ghsm}^h \frac{\eM-\ehm}{\tau}
	\cdot\nbhsm\,\nbhsm\cdot\phi_h
	+\mathscr J^m(\phi_h),
\end{align}
where
\begin{align}
%	\label{def-Jm-phi}
	\mathscr J^m(\phi_h)
	&:=\int_{\Ghm}^h \frac{X_h^{m+1}-X_h^m}{\tau}
	\cdot\nbhm\,\nbhm\cdot\phi_h
	\notag\\
	&\quad-\int_{\Ghsm}^h \frac{X_h^{m+1}-X_h^m}{\tau}
	\cdot\nbhsm\,\nbhsm\cdot\phi_h.
	\notag
\end{align}

For $\R^2$-valued functions $u$ and $v$ on a piecewise smooth curve
$\Sigma$, with unit normal $n_\Sigma$, define
\begin{align*}
	\mathscr A_\Sigma(u,v)
	&:=\int_\Sigma\nabla_\Sigma u\cdot\nabla_\Sigma v,\\
	\mathscr A_\Sigma^N(u,v)
	&:=\int_\Sigma[(\nabla_\Sigma u)n_\Sigma]
	\cdot[(\nabla_\Sigma v)n_\Sigma],\\
	\mathscr A_\Sigma^T(u,v)
	&:=\int_\Sigma{\rm tr}\big[
	(\nabla_\Sigma u)(I-n_\Sigma n_\Sigma^\top)
	(\nabla_\Sigma v)^\top\big],\\
	\mathscr B_\Sigma(u,v)
	&:=\int_\Sigma
	(\nabla_\Sigma\cdot u)(\nabla_\Sigma\cdot v)
	-{\rm tr}(\nabla_\Sigma u\nabla_\Sigma v).
\end{align*}
Thus $\mathscr A_\Sigma=\mathscr A_\Sigma^N+\mathscr A_\Sigma^T$.
Writing $\ud_i u:=(\nabla_\Sigma u)_i$, $i=1,2$, and using the Einstein
summation convention, set
\[
(D_\Sigma u)_{rl}:=-\ud_lu_r-\ud_ru_l+\delta_{rl}\ud_j u_j.
\]
The identity from \cite[Eq.~(5.8)]{BL2024} reads
\begin{align}
%	\label{eq:pdf-DGamma-identity}
	\int_\Sigma\nabla_\Sigma{\rm id}\cdot(D_\Sigma u)\nabla_\Sigma v
	=-\mathscr A_\Sigma^T(u,v)+\mathscr B_\Sigma(u,v). \notag
\end{align}

The discrete parabolicity decomposition in \cite[Section~5.2]{BL2024} applies without change to the stiffness
terms:
\begin{align}\label{eq:full_para}
	&\int_{\Ghm}\nabla_{\Ghm}X_h^{m+1}\cdot\nabla_{\Ghm}\phi_h
	-\int_{\Ghsm}\nabla_{\Ghsm}X_{h,*}^{m+1}
	\cdot\nabla_{\Ghsm}\phi_h
	\notag\\
	&\qquad=
	\mathscr A_{h,*}^N(\eM,\phi_h)
	+\mathscr A_{h,*}^T(\eM-\ehm,\phi_h)
	+\mathscr B^m(\ehm,\phi_h)
	+\mathscr K^m(\phi_h),
\end{align}
where
\begin{align}
	\mathscr A_{h,*}^N(u_h,v_h)
	&:=\mathscr A_{\Ghsm}^N(u_h,v_h),
	\qquad
	\mathscr A_{h,*}^T(u_h,v_h)
	:=\mathscr A_{\Ghsm}^T(u_h,v_h),
	\label{def-As-AGs}\\
	\mathscr A_{h,*}(u_h,v_h)
	&:=\mathscr A_{h,*}^N(u_h,v_h)+\mathscr A_{h,*}^T(u_h,v_h),
	\qquad
	\mathscr B^m(u_h,v_h)
	:=\mathscr B_{\Gm}(u_h^\ell,v_h^\ell).
	\label{def-Bm}
\end{align}
With
\[
\hat X_{h,\theta}^m:=(1-\theta)\hat X_{h,*}^m+\theta X_h^m,
\qquad
X_{h,\theta}^{m+1}:=(1-\theta)X_{h,*}^{m+1}+\theta X_h^{m+1},
\]
the remainder is
\begin{align}
	\mathscr K^m(\phi_h)
	&:=\int_0^1
	\big[
	\mathscr A_{\hat\Gamma_{h,\theta}^m}^N(\eM,\phi_h)
	-\mathscr A_{\Ghsm}^N(\eM,\phi_h)
	\big]\,\d\theta
	\notag\\
	&\quad+\int_0^1
	\big[
	\mathscr A_{\hat\Gamma_{h,\theta}^m}^T(\eM-\ehm,\phi_h)
	-\mathscr A_{\Ghsm}^T(\eM-\ehm,\phi_h)
	\big]\,\d\theta
	\notag\\
	&\quad+\int_0^1
	\big[
	\mathscr B_{\hat\Gamma_{h,\theta}^m}(\ehm,\phi_h)
	-\mathscr B_{\Ghsm}(\ehm,\phi_h)
	\big]\,\d\theta
	\notag\\
	&\quad+\mathscr B_{\Ghsm}(\ehm,\phi_h)
	-\mathscr B_{\Gm}((\ehm)^\ell,\phi_h^\ell)
	\notag\\
	&\quad+\int_0^1\int_{\hat\Gamma_{h,\theta}^m}
	\nabla_{\hat\Gamma_{h,\theta}^m}
	(X_{h,\theta}^{m+1}-\hat X_{h,\theta}^m)
	\cdot D_{\hat\Gamma_{h,\theta}^m}\ehm\,
	\nabla_{\hat\Gamma_{h,\theta}^m}\phi_h\,\d\theta. \notag
\end{align}

The stabilization terms have the same decomposition as in
\cite[Eq.~(4.9)]{Bai2026}; see also \cite[Eq.~(4.20)]{BL2025}:
\begin{align}\label{terms-5-6}
	&-\int_{\Ghm}\nabla_{\Ghm}X_h^m
	\cdot\nabla_{\Ghm}I_h[(I-\nbhm(\nbhm)^\top)\phi_h]
	\notag\\
	&\quad+\int_{\Ghsm}\nabla_{\Ghsm}\hat X_{h,*}^m
	\cdot\nabla_{\Ghsm}I_h[(I-\nbhsm(\nbhsm)^\top)\phi_h]
	\notag\\
	&=-\int_{\Ghm}\nabla_{\Ghm}X_h^m
	\cdot\nabla_{\Ghm}
	I_h[(I-\nbhm(\nbhm)^\top-\Tbhm)\phi_h]
	\notag\\
	&\quad+\int_{\Ghsm}\nabla_{\Ghsm}\hat X_{h,*}^m
	\cdot\nabla_{\Ghsm}
	I_h[(I-\nbhsm(\nbhsm)^\top-\Tbhsm)\phi_h]
	\notag\\
	&\quad-\int_{\Ghm}\nabla_{\Ghm}X_h^m
	\cdot\nabla_{\Ghm}I_h[(\Tbhm-\Tbhsm)\phi_h]
	\notag\\
	&\quad-\int_{\Ghm}\nabla_{\Ghm}X_h^m
	\cdot\nabla_{\Ghm}I_h(\Tbhsm\phi_h)
	+\int_{\Ghsm}\nabla_{\Ghsm}\hat X_{h,*}^m
	\cdot\nabla_{\Ghsm}I_h(\Tbhsm\phi_h)
	\notag\\
	&=:\mathscr F_1^m(\phi_h)+\mathscr F_2^m(\phi_h)
	+\mathscr F_3^m(\phi_h)
	\notag\\
	&\quad-\mathscr A_{h,*}^N(\ehm,I_h\Tbhsm\phi_h)
	-\mathscr B^m(\ehm,I_h\Tbhsm\phi_h)
	-\mathscr Q^m(I_h\Tbhsm\phi_h),
\end{align}
where the first three integrals define
$\mathscr F_1^m$, $\mathscr F_2^m$ and $\mathscr F_3^m$, respectively, and
\begin{align*}
	\mathscr Q^m(\phi_h)
	&:=\int_0^1
	\big[
	\mathscr A_{\hat\Gamma_{h,\theta}^m}^N(\ehm,\phi_h)
	-\mathscr A_{\Ghsm}^N(\ehm,\phi_h)
	\big]\,\d\theta\\
	&\quad+\int_0^1
	\big[
	\mathscr B_{\hat\Gamma_{h,\theta}^m}(\ehm,\phi_h)
	-\mathscr B_{\Ghsm}(\ehm,\phi_h)
	\big]\,\d\theta\\
	&\quad+\mathscr B_{\Ghsm}(\ehm,\phi_h)
	-\mathscr B_{\Gm}((\ehm)^\ell,\phi_h^\ell).
\end{align*}

Substituting \eqref{eq:mass_diff}, \eqref{eq:full_para} and
\eqref{terms-5-6} into \eqref{eq:err_eq1} yields
\begin{align}\label{eq:err_eq2}
	&\int_{\Ghsm}^h\frac{\eM-\ehm}{\tau}
	\cdot\nbhsm\,\nbhsm\cdot\phi_h
	+\mathscr J^m(\phi_h)
	\notag\\
	&\quad+\mathscr A_{h,*}^N(\eM,\phi_h)
	+\mathscr A_{h,*}^T(\eM-\ehm,\phi_h)
	+\mathscr B^m(\ehm,\phi_h)+\mathscr K^m(\phi_h)
	\notag\\
	&\quad+\sum_{i=1}^3\mathscr F_i^m(\phi_h)
	-\mathscr A_{h,*}^N(\ehm,I_h\Tbhsm\phi_h)
	-\mathscr B^m(\ehm,I_h\Tbhsm\phi_h)
	-\mathscr Q^m(I_h\Tbhsm\phi_h)
	\notag\\
	&=-\mathscr D^m(\phi_h).
\end{align}
The discrete $H^1$ parabolicity follows from
\begin{align*}
	\mathscr A_{h,*}^N(\eM,\eM)
	+\mathscr A_{h,*}^T(\eM-\ehm,\eM)
	=\mathscr A_{h,*}(\eM,\eM)-\mathscr A_{h,*}^T(\ehm,\eM),
\end{align*}
where the last term is controlled by the cancellation estimate
\eqref{eq:AT_em_em} below.

\subsection{Bilinear error estimates}

%We recall the geometric perturbation estimates from
%\cite[Lemma~4.1]{Bai2026}; cf. \cite[Lemma~4.2]{BL2024}.
%Their proofs use only the geometric comparison of the two curves and
%therefore apply to the present method.

The following bilinear error estimates are standard (cf. \cite[Lemma 4.2]{BL2024}). They can be proved by a fundamental theorem of calculus argument together with the norm equivalence.
\begin{lemma}\label{lemma:e-blinear}
	Assuming the induction hypothesis \eqref{eq:ind_hypo1},
	for all $f_h,g_h\in S_h(\Ghsm)$ and
	$p,q,r\in[1,\infty]$ with $1/p+1/q+1/r=1$, we have
	\begin{align*}
		\Big|\int_{\Ghm}f_hg_h-\int_{\Ghsm}f_hg_h\Big|
		&\lesssim
		\|\nabla_\Ghsm\ehm\|_{L^p(\Ghsm)}
		\|f_h\|_{L^q(\Ghsm)}\|g_h\|_{L^r(\Ghsm)},\\
		\Big|\int_{\Ghm}\nabla_\Ghm f_hg_h
		-\int_{\Ghsm}\nabla_\Ghsm f_hg_h\Big|
		&\lesssim
		\|\nabla_\Ghsm\ehm\|_{L^p(\Ghsm)}
		\|\nabla_\Ghsm f_h\|_{L^q(\Ghsm)}
		\|g_h\|_{L^r(\Ghsm)},
	\end{align*}
	and
	\begin{align*}
		&\Big|\int_{\Ghm}\nabla_\Ghm f_h\cdot\nabla_\Ghm g_h
		-\int_{\Ghsm}\nabla_\Ghsm f_h\cdot\nabla_\Ghsm g_h\Big|\\
		&\qquad\lesssim
		\|\nabla_\Ghsm\ehm\|_{L^p(\Ghsm)}
		\|\nabla_\Ghsm f_h\|_{L^q(\Ghsm)}
		\|\nabla_\Ghsm g_h\|_{L^r(\Ghsm)}.
	\end{align*}
	Moreover, \eqref{normal-intpl-2}--\eqref{normal-intpl-3} give
	\begin{align*}
		&\Big|\int_{\Ghsm}(\nhm-\nhsm)f_hg_h\Big|
		+\Big|\int_{\Ghsm}(\nbhm-\nbhsm)f_hg_h\Big|\\
		&\qquad\lesssim
		\|\nabla_\Ghsm\ehm\|_{L^p(\Ghsm)}
		\|f_h\|_{L^q(\Ghsm)}\|g_h\|_{L^r(\Ghsm)}.
	\end{align*}
\end{lemma}

\subsection{Basic estimates for linear and bilinear forms}
\label{sec:lin-bilin-est}

All $\mathscr B$-related terms vanish on curves, since the tangential
gradient has rank-one structure and hence
\[
{\rm tr}(\nabla_\Sigma u\nabla_\Sigma v)
=(\nabla_\Sigma\cdot u)(\nabla_\Sigma\cdot v).
\]

For the mass perturbation $\mathscr J^m$, we use the homotopy argument
of \cite[Eqs.~(4.28)--(4.31)]{BL2025}. Let
$\bar n_{h,\theta}^m$ be the reversely weighted normal on
$\hat\Gamma_{h,\theta}^m$. Transporting the velocity and the test function
through their nodal values gives
\begin{align}
	\mathscr J^m(\phi_h)
	&=\int_0^1\frac{\d}{\d\theta}
	\int_{\hat\Gamma_{h,\theta}^m}^h
	\frac{X_h^{m+1}-X_h^m}{\tau}
	\cdot\bar n_{h,\theta}^m\,
	\phi_h\cdot\bar n_{h,\theta}^m\,\d\theta.
\end{align}
The differentiation includes both normal vectors and the quadrature
weights. By Lemma~\ref{lemma:ud},
\begin{align}
%	\label{eq:n_t}
	\partial_\theta^\bullet\hat n_{h,\theta}^m
	=-\nabla_{\hat\Gamma_{h,\theta}^m}\ehm
	\cdot\hat n_{h,\theta}^m.
	\notag
\end{align}
Differentiating the edge-length weights in
\eqref{bar-n-hat-n}--\eqref{eq:bar_n} therefore yields, at each node,
\begin{align}
%	\label{eq:n_bar_t_2}
	|\partial_\theta^\bullet\bar n_{h,\theta}^m(p)|
	&\lesssim
	|\nabla_{\hat\Gamma_{h,\theta}^m}\ehm(p+)|
	+|\nabla_{\hat\Gamma_{h,\theta}^m}\ehm(p-)|.
	\notag
\end{align}
H\"older's inequality and norm equivalence then imply
\begin{align}\label{eq:J_est}
	|\mathscr J^m(\phi_h)|
	&\lesssim
	\Big\|\frac{X_h^{m+1}-X_h^m}{\tau}\Big\|_{L^\infty(\Ghsm)}
	\|\nabla_\Ghsm\ehm\|_{L^2(\Ghsm)}
	\|\phi_h\|_{L^2(\Ghsm)}
	\notag\\
	&\lesssim
	(1+\|\delta_\tau\ehm\|_{L^\infty(\Ghsm)})
	\|\nabla_\Ghsm\ehm\|_{L^2(\Ghsm)}
	\|\phi_h\|_{L^2(\Ghsm)},
\end{align}
where \eqref{eq:geo_rel_31}, \eqref{W1infty-g} and the boundedness of the
comparison velocity give
\begin{align}
	\Big\|\frac{X_h^{m+1}-X_h^m}{\tau}\Big\|_{L^\infty(\Ghsm)}
	\lesssim1+\|\delta_\tau\ehm\|_{L^\infty(\Ghsm)}.
\end{align}

The stiffness cancellation estimates are identical to those in
\cite[Section~4.3 and Lemma~C.6]{Bai2026}; see also
Lemma~\ref{lemma:AT-sup}. In particular,
\begin{align}\label{eq:AT_em_em}
	|\mathscr A_{h,*}^T(\ehm,\phi_h)|
	&\lesssim
	\min\Big\{
	\|\ehm\|_{L^2(\Ghsm)}\|\phi_h\|_{H^1(\Ghsm)},
	\|\ehm\|_{H^1(\Ghsm)}\|\phi_h\|_{L^2(\Ghsm)}
	\Big\},
\end{align}
and
\begin{align*}
	|\mathscr A_{h,*}^N(\ehm,I_h\Tbhsm\phi_h)|
	&\lesssim
	\min\Big\{
	\|\ehm\|_{L^2(\Ghsm)}
	\|I_h\Tbhsm\phi_h\|_{H^1(\Ghsm)},
%	\\
%	&\hspace{4.8em}
	\|\ehm\|_{H^1(\Ghsm)}
	\|I_h\Tbhsm\phi_h\|_{L^2(\Ghsm)}
	\Big\}\\
	&\quad+h\|\nabla_\Ghsm\ehm\|_{L^2(\Ghsm)}
	\|\phi_h\|_{L^2(\Ghsm)}.
\end{align*}

To apply the estimate for $\mathscr K^m$ in
\cite[Eq.~(4.14)]{Bai2026}, we account for the corrected error velocity:
\begin{align*}
	\eM-\ehm
	=\tau\delta_\tau\ehm+\tau I_h\Tsm v^m.
\end{align*}
The correction satisfies
$\|I_h\Tsm v^m\|_{W^{1,\infty}(\Ghsm)}\lesssim1$.
%In the geometric perturbation terms, its contribution is bounded by
%\[
%\tau\|\nabla_\Ghsm\ehm\|_{L^2(\Ghsm)}
%\|\nabla_\Ghsm I_h\Tsm v^m\|_{L^\infty(\Ghsm)}
%\|\nabla_\Ghsm\phi_h\|_{L^2(\Ghsm)}
%\lesssim
%\tau\|\nabla_\Ghsm\ehm\|_{L^2(\Ghsm)}
%\|\nabla_\Ghsm\phi_h\|_{L^2(\Ghsm)}.
%\]
The same argument as in \cite[Eq.~(4.14)]{Bai2026}, with this smooth
contribution estimated separately, gives
\begin{align}\label{eq:K_est}
	|\mathscr K^m(\phi_h)|
	&\lesssim
	\|\nabla_\Ghsm\ehm\|_{L^\infty(\Ghsm)}
	\|\nabla_\Ghsm\ehm\|_{L^2(\Ghsm)}
	\|\nabla_\Ghsm\phi_h\|_{L^2(\Ghsm)}
	\notag\\
	&\quad+\tau\Big(
	\|\nabla_\Ghsm\ehm\|_{L^2(\Ghsm)}
	+\|\nabla_\Ghsm\delta_\tau\ehm\|_{L^2(\Ghsm)}
	\|\nabla_\Ghsm\ehm\|_{L^\infty(\Ghsm)}
	\Big)\|\nabla_\Ghsm\phi_h\|_{L^2(\Ghsm)}.
\end{align}
The estimate for $\mathscr Q^m$ follows directly from
\cite[Eq.~(4.15)]{Bai2026}:
\begin{align}\label{eq:Q_est}
	|\mathscr Q^m(I_h\Tbhsm\phi_h)|
	&\lesssim
	\|\nabla_\Ghsm\ehm\|_{L^\infty(\Ghsm)}
	\|\nabla_\Ghsm\ehm\|_{L^2(\Ghsm)}
	\|\nabla_\Ghsm I_h\Tbhsm\phi_h\|_{L^2(\Ghsm)}.
\end{align}

Finally, the functionals $\mathscr F_i^m$, $i=1,2,3$, coincide with the
stabilization remainders in \cite[Section~4.3]{Bai2026}, since the
stabilization and the reversely weighted normals are the same.
Their estimates therefore carry over directly from \cite[Eqs.~(4.16)--(4.18)]{Bai2026}:
\begin{align}
	|\mathscr F_1^m(\phi_h)|
	&\lesssim
	h^{-3/2}\|\nabla_\Ghsm\ehm\|_{L^2(\Ghsm)}^3
	\|\phi_h\|_{L^\infty(\Ghsm)}
	+h^2\|\phi_h\|_{L^2(\Ghsm)},
	\label{Estimate-F1m}\\
	|\mathscr F_2^m(\phi_h)|
	&\lesssim h^2\|\phi_h\|_{L^2(\Ghsm)},
	\label{Estimate-F2m}\\
	|\mathscr F_3^m(\phi_h)|
	&\lesssim
	h^{-1}\|\nabla_\Ghsm\ehm\|_{L^2(\Ghsm)}^2
	\|\phi_h\|_{L^\infty(\Ghsm)}
	+\|\nabla_\Ghsm\ehm\|_{L^2(\Ghsm)}
	\|\phi_h\|_{L^2(\Ghsm)}.
	\label{Estimate-F3m}
\end{align}

\section{Tangential stability and the basic velocity estimates}\label{sec:tan_stab0}

\subsection{Stability of the tangential motion}\label{sec:tan_stab}

We begin by testing the scheme \eqref{eq:BGN-stab} with a nodally tangential finite element function $I_h\Tbhm \phi_h$. The definition of $\dtXm$ and the properties of $v^m$ then give the following estimate.
\begin{lemma}[\!\!{\cite[Eqs. (4.41)--(4.48)]{BL2025}}]\label{lemma:Lap-vm}
	We have
	\begin{align}
		\Big|\int_{\Gamma_h^m} \nabla_{\Gamma_h^m} \dtXm \cdot  \nabla_{\Gamma_h^m} I_h\Tbhm \phi_h \Big|
		\lesssim
		(h^2 + \| \nabla_\Ghsm \ehm \|_{L^2(\Ghsm)})  \| I_h\Tbhm \phi_h  \|_{H^1(\Ghsm)} .
	\end{align}
\end{lemma}
%\begin{proof}
%	See Appendix \ref{sec:Lap-vm}.
%\end{proof}
Setting $\phi_h = I_h \Tbhm\dtXm$ in Lemma~\ref{lemma:Lap-vm} yields
\begin{align}
%	\label{eq:N+TdotT}
	&\int_{\Gamma_h^m} \nabla_{\Gamma_h^m} \dtXm \cdot  \nabla_{\Gamma_h^m} I_h \Tbhm \dtXm
	%	\notag\\
	%	&
	\lesssim
	(h^2 + \| \nabla_\Ghsm \ehm \|_{L^2(\Ghsm)} ) \| I_h\bar T_h^m \dtXm \|_{H^1(\Ghsm)}  .
	\notag
\end{align}
We also use the following estimate based on orthogonality cancellation. Its proof follows the arguments in \cite[Lemma 4.6]{BL2025} and \cite[Lemma 4.5]{BGV2026}, using the super-approximation of $\nbhsm$ established in Lemma~\ref{lemma:n_bar_app}.
\begin{lemma}[\!\!{\cite[Lemma 4.6]{BL2025} and \cite[Lemma 4.5]{BGV2026}}]\label{lemma:NT_stab}
	For any $f_h,g_h \in S_h(\Ghm)^2$, we have
	\begin{align}\label{eq:NT_stab}
		&\Big| \int_{\Gamma_h^m} \nabla_{\Gamma_h^m} I_h \Nbhm f_h \cdot  \nabla_{\Gamma_h^m} I_h \Tbhm g_h \Big|
		%		\notag\\
		%		&
		\lesssim
		\| f_h \|_{L^2(\Ghsm)} \| g_h \|_{H^1(\Ghsm)}
		\notag\\
		&\quad
		+ (1 + h^{-2} \| \nabla_\Ghsm \ehm \|_{L^2(\Ghsm)})
		\| f_h \|_{L^2(\Ghsm)} \| g_h \|_{L^\infty(\Ghsm)}
		.
	\end{align}
\end{lemma}
%\begin{proof}
%	See Appendix \ref{sec:ortho-cancel}.
%\end{proof}
Combining the preceding two lemmas, we obtain
\begin{align}
%	\label{eq:grad_P_X_diff}
	&\int_{\Gamma_h^m} \nabla_{\Gamma_h^m} I_h \Tbhm\dtXm \cdot  \nabla_{\Gamma_h^m}  I_h \Tbhm \dtXm \notag\\
	&= \int_{\Gamma_h^m} \nabla_{\Gamma_h^m} \dtXm \cdot  \nabla_{\Gamma_h^m} I_h \Tbhm \dtXm \notag\\
	&\quad- \int_{\Gamma_h^m} \nabla_{\Gamma_h^m} I_h \Nbhm \dtXm \cdot  \nabla_{\Gamma_h^m} I_h \Tbhm \dtXm \notag\\
	&\lesssim  (h^2 + \| \nabla_\Ghsm \ehm \|_{L^2(\Ghsm)} ) \| I_h\bar T_h^m \dtXm \|_{H^1(\Ghsm)} \notag\\
	&\quad+ (1 + h^{-4}\| \nabla_\Ghsm \ehm \|_{L^2(\Ghsm)}^2) \| I_h \Nbhm \dtXm \|_{L^2(\Ghsm)}^2
	. \notag
\end{align}
Young's inequality and absorption then yield
\begin{align}\label{eq:tan_stab}
	&\| \nabla_\Ghsm I_h \Tbhm \dtXm \|_{L^2(\Ghsm)}
	\lesssim  h^2 +  \| \nabla_\Ghsm \ehm \|_{L^2(\Ghsm)}  \notag\\
	&\quad+ (1 + h^{-2}\| \nabla_\Ghsm \ehm \|_{L^2(\Ghsm)}) \| I_h \Nbhm \dtXm \|_{L^2(\Ghsm)}  .
\end{align}
The relation
$\dtem = \dtXm - I_h g^m $
from \eqref{eq:geo_rel_31}, together with the tangential stability estimate \eqref{eq:tan_stab} and Lemma~\ref{lemma:n_bar_app}, gives
\begin{align}
%	\label{asdfjklb}
	&\| \nabla_\Ghsm I_h \Tbhsm \dtem \|_{L^2(\Ghsm)}
	+
	\| \nabla_\Ghsm I_h \Tsm \dtem \|_{L^2(\Ghsm)}
	\notag\\
	&
	\leq
	2\| \nabla_\Ghsm I_h \bar T_{h}^m \dtem \|_{L^2(\Ghsm)}
	+
	2\| \nabla_\Ghsm I_h( \bar T_{h}^m -  \bar T_{h,*}^m) \dtem \|_{L^2(\Ghsm)}
	\notag\\
	&\quad
	+
	\| \nabla_\Ghsm I_h(\Tbhsm - \Tsm) \dtem \|_{L^2(\Ghsm)}
	\notag\\
	&\lesssim \tau + h^{2} + \| \nabla_\Ghsm \ehm \|_{L^2(\Ghsm)}  \notag\\
	&\quad+(1 + h^{-2}\| \nabla_\Ghsm \ehm \|_{L^2(\Ghsm)})\| I_h \Nbhm \dtem \|_{L^2(\Ghsm)} \notag\\
	&\quad+ h^{-3/2} \| \nabla_\Ghsm \ehm \|_{L^2(\Ghsm)}  \| \dtem \|_{L^2(\Ghsm)}  \notag\\
	&\quad+ h \| \dtem \|_{L^2(\Ghsm)}  \notag\\
	&\lesssim \tau + h^{2} + \| \nabla_\Ghsm \ehm \|_{L^2(\Ghsm)}  \notag\\
	&\quad+(1 + h^{-2}\| \nabla_\Ghsm \ehm \|_{L^2(\Ghsm)})\| I_h \Nbhsm \dtem \|_{L^2(\Ghsm)} \notag\\
	&\quad+ (h+h^{-3/2} \| \nabla_\Ghsm \ehm \|_{L^2(\Ghsm)}) \| \nabla_\Ghsm I_h \Tbhsm \dtem \|_{L^2(\Ghsm)} , \notag
\end{align}
where the final inequality uses the Poincar\'e inequality (Lemma \ref{lemma:poincare}).
%where we have estimated $\| \Tbhm - \bar T_{h,*}^m \|_{L^2(\Ghsm)}$ by using \eqref{normal-intpl-3} and decomposed the term $\| \dtem \|_{L^\infty(\Ghsm)} $ into its normal and tangential parts, respectively, and have changed $\Nbhm$ to $\Nbhsm$ by using estimate
%$$
%\|\bar N_{h}^m- \bar N_{h,*}^m \|_{L^2(\Ghsm)} \lesssim  \| \nabla_\Ghsm \ehm \|_{L^2(\Ghsm)} \lesssim h^{1.75} .
%$$
%This estimate follows from \eqref{normal-intpl-3} and \eqref{Linfty-W1infty-hat-em}, and can be used to absorb the additional perturbation term caused by changing $\Nbhm$ to $\Nbhsm$.
%Since $\| \nabla_\Ghsm \ehm \|_{L^2(\Ghsm)} \lesssim h^{1.75}$, as shown in \eqref{Linfty-W1infty-hat-em}, the last term on the right-hand side of \eqref{asdfjklb} can be absorbed by the left-hand side. This leads to the following result:
The induction hypothesis \eqref{eq:ind_hypo1} allows us to absorb the last term on the right-hand side and obtain
\begin{align}\label{eq:tan_stab_e}
	&\| \nabla_\Ghsm I_h \Tbhsm \dtem \|_{L^2(\Ghsm)}
	+
	\| \nabla_\Ghsm I_h \Tsm \dtem \|_{L^2(\Ghsm)}
		 \notag\\
		&
	\lesssim (\tau + h^{2}) + \| \nabla_\Ghsm \ehm \|_{L^2(\Ghsm)} 
	+(1 + h^{-2}\| \nabla_\Ghsm \ehm \|_{L^2(\Ghsm)})\| I_h \Nbhsm \dtem \|_{L^2(\Ghsm)} .
\end{align}
%Then, by applying the Poincar\'e inequality with $v_h=I_h \bar T_{h,*}^m \dtem$ satisfying $I_h(v_h\cdot\nbhsm)=0$ in Lemma \ref{lemma:poincare_lump}, we can control the $L^2$ norm of the tangential component $I_h \bar T_{h,*}^m \dtem$ by the left-hand side of \eqref{eq:tan_stab_e}. Since the $L^2$ norm of the normal component $I_h \bar N_{h,*}^m \dtem$ already appears on the right-hand side of \eqref{eq:tan_stab_e}, by summing up the $L^2$ norms of the tangential and normal components of $\dtem$ we obtain the following result:
Applying the Poincar\'e inequality (Lemma \ref{lemma:poincare}) once more yields the following stability estimate for the full error velocity in $L^2$:
\begin{align}\label{eq:tan_stab_e1}
	&\| \dtem \|_{L^2(\Ghsm)}
	%	 \notag\\
	%	&
	\lesssim (\tau + h^{2}) + \| \nabla_\Ghsm \ehm \|_{L^2(\Ghsm)}  \notag\\
	&\quad+(1 + h^{-2}\| \nabla_\Ghsm \ehm \|_{L^2(\Ghsm)})\| I_h \Nbhsm \dtem \|_{L^2(\Ghsm)} .
\end{align}

\subsection{Basic velocity estimates}\label{sec:bbd_vel}

To estimate the normal component of the error velocity, we test the error equation \eqref{eq:err_eq2} with the full error velocity. Choosing $\phi_h =  \delta_\tau\ehm$ gives
\begin{align}\label{eq:err1}
	& \int_{\hat\Gamma_{h, *}^{m}}^h \delta_\tau\ehm \cdot \nbhsm \, \delta_\tau\ehm \cdot \nbhsm
	=
	- \int_{\hat\Gamma_{h, *}^{m}}^h I_h\Tsm v^m \cdot \bar n_{h, *}^m \,   \dtem \cdot \bar n_{h, *}^m
	\notag\\
	&\quad
	- \mathscr D^m( \delta_\tau\ehm) - \mathscr J^m(\delta_\tau\ehm) - \mathscr B^m(\ehm,  \delta_\tau\ehm) - \mathscr K^m( \delta_\tau\ehm)
	\notag\\
	&\quad
	- \mathscr A_{h, *}^N(e_h^{m+1},  \delta_\tau\ehm ) - \mathscr A_{h, *}^T(e_h^{m+1} - \hat{e}_h^m, \delta_\tau\ehm )  \notag\\
	&\quad- \sum_{i=1}^3 \mathscr F_i^m(\delta_\tau\ehm) + \mathscr A_{h, *}^N(\hat e_h^{m}, I_h\Tbhsm \delta_\tau\ehm) + \mathscr B^m(\hat e_h^{m},  I_h\Tbhsm\delta_\tau\ehm) + \mathscr Q^m( I_h\Tbhsm\delta_\tau\ehm) \notag\\
	&\leq
	- \int_{\hat\Gamma_{h, *}^{m}}^h I_h\Tsm v^m \cdot \bar n_{h, *}^m \,   \dtem \cdot \bar n_{h, *}^m
	\notag\\
	&
	- \mathscr D^m( \delta_\tau\ehm) - \mathscr J^m(\delta_\tau\ehm) - \mathscr B^m(\ehm,  \delta_\tau\ehm) - \mathscr K^m( \delta_\tau\ehm)
	\notag\\
	&\quad
	- \mathscr A_{h, *}^N(\ehm+\tau I_h\Tsm v^m,  \delta_\tau\ehm )
	- \mathscr A_{h, *}^T(\tau I_h\Tsm v^m,  \delta_\tau\ehm )
	\notag\\
	&\quad- \sum_{i=1}^3 \mathscr F_i^m(\delta_\tau\ehm) + \mathscr A_{h, *}^N(\hat e_h^{m},  I_h\Tbhsm \delta_\tau\ehm) + \mathscr B^m(\hat e_h^{m},  I_h\Tbhsm \delta_\tau\ehm) + \mathscr Q^m( I_h\Tbhsm \delta_\tau\ehm)
	,
\end{align}
where the final inequality follows by dropping the following two nonpositive terms, which arise from testing with the full error velocity:
\begin{align*}
	- \tau \mathscr A_{h, *}^N(\delta_\tau\ehm, \delta_\tau\ehm)
	\qquad
	\mbox{and}\qquad - \tau \mathscr A_{h, *}^T( \delta_\tau\ehm, \delta_\tau\ehm) .
\end{align*}
Nodal orthogonality and Lemma~\ref{lemma:n_bar_app} give the following bound for the first term on the right-hand side of \eqref{eq:err1}:
\begin{align}
	\label{eq:mass-ortho}
	&\Big| \int_{\hat\Gamma_{h, *}^{m}}^h \dtem \cdot \bar n_{h, *}^m\,  I_h\Tsm v^m \cdot \bar n_{h, *}^m \Big|
	\notag\\
	&
	\lesssim \| I_h\Tsm - \bar T_{h,*}^m \|_{L^2_h(\Ghsm)} \| \dtem \cdot \bar n_{h, *}^m \|_{L^2_h(\Ghsm)}
	\notag\\
	&
	\lesssim h^{2} \| \dtem \cdot \bar n_{h, *}^m \|_{L^2_h(\Ghsm)} .
\end{align}
Estimates \eqref{eq:K_est} and \eqref{eq:Q_est} yield
\begin{align}
	%	\label{eq:K_est1}
	&| \mathscr K^m( \delta_\tau\ehm) | + | \mathscr Q^m(I_h \Tbhsm \delta_\tau\ehm) | \notag\\
	&\lesssim (\tau + \| \nabla_{\Ghsm} \ehm \|_{L^\infty(\Ghsm)}) \| \nabla_\Ghsm \ehm \|_{L^2(\Ghsm)} \| \delta_\tau\ehm \|_{H^1(\Ghsm)} \notag\\
	&\quad+ \tau \| \nabla_\Ghsm \ehm \|_{L^\infty(\Ghsm)}\| \nabla_\Ghsm \delta_\tau\ehm \|_{L^2(\Ghsm)} \| \nabla_\Ghsm  \delta_\tau\ehm \|_{L^2(\Ghsm)}
	. \notag
\end{align}
The definitions of $\mathscr A_{h, *}^N(\cdot,\cdot)$, $\mathscr A_{h, *}^T(\cdot,\cdot)$ and $\mathscr B^m(\cdot,\cdot)$ in \eqref{def-As-AGs}--\eqref{def-Bm} also give
\begin{align}
	%	\label{eq:AB_est}
	|\mathscr A_{h, *}^N(u_h,v_h) | + |\mathscr A_{h, *}^T(u_h,v_h) | + |\mathscr B^m(u_h,v_h) |
	&\lesssim \| u_h \|_{H^1(\Ghsm)} \| v_h \|_{H^1(\Ghsm)} , \notag
\end{align}
for any $u_h, v_h \in [S_h(\Ghsm)]^2$. Inserting these bounds, together with the estimates for $\mathscr F_i^m, i=1,2,3,$ in \eqref{Estimate-F1m}--\eqref{Estimate-F3m}, into \eqref{eq:err1} yields
\begin{align}
	&\| \delta_\tau\ehm \cdot \nbhsm \|_{L_h^2(\Ghsm)}^2
	\notag\\
	&\lesssim  (\tau + h^{2})  \| \dtem \|_{H^1(\Ghsm)}
	+
	h^2 \| \dtem \cdot \nbhsm \|_{L_h^2(\Ghsm)}
	\notag\\
	&\quad
	+ h^{-1/2} \| \nabla_\Ghsm \ehm \|_{L^2(\Ghsm)} \| \dtem \|_{L^2(\Ghsm)}^2
	\notag\\
	&\quad
	+  (\tau+\| \ehm \|_{H^1(\Ghsm)}) \| \dtem \|_{H^1(\Ghsm)}
	\notag\\
	&\quad
	+
	(\tau + \| \nabla_{\Ghsm} \ehm \|_{L^\infty(\Ghsm)}) \| \nabla_\Ghsm \ehm \|_{L^2(\Ghsm)} \| \delta_\tau\ehm \|_{H^1(\Ghsm)} \notag\\
	&\quad+ \tau \| \nabla_\Ghsm \ehm \|_{L^\infty(\Ghsm)}\| \nabla_\Ghsm \delta_\tau\ehm \|_{L^2(\Ghsm)} \| \nabla_\Ghsm  \delta_\tau\ehm \|_{L^2(\Ghsm)}
	%	\notag\\
	%	&\lesssim  h^{-1}(\tau + h^{2}) +  h^{-1} \| \nabla_\Ghsm \ehm \|_{L^2(\Ghsm)}
	\notag
	.
\end{align}
Applying the inverse inequality and Young's inequality, and then using the induction hypothesis \eqref{eq:ind_hypo1}, Lemma \ref{lemma:lump} and Lemma \ref{lemma:poincare}, we obtain
\begin{align}
	&\| I_h \Nbhsm \delta_\tau\ehm \|_{L^2(\Ghsm)}
	\sim \| \delta_\tau\ehm \cdot \nbhsm \|_{L_h^2(\Ghsm)}
	\notag\\
	&\lesssim  h^{-1}(\tau + h^{2})
	+ \epsilon^{-1} h^{-1} \| \ehm \|_{H^1(\Ghsm)}
	\notag\\
	&\quad
	+ h^{-1/4} \| \nabla_\Ghsm \ehm \|_{L^2(\Ghsm)}^{1/2} \| \dtem \|_{L^2(\Ghsm)}
	\notag\\
	&\quad
	+
	(\epsilon h + \tau^{1/2} \| \nabla_\Ghsm \ehm \|_{L^\infty(\Ghsm)}^{1/2}) \| \dtem \|_{H^1(\Ghsm)}
	\notag\\
	&\quad
	+
	\epsilon^{-1}
	h^{-1}
	(\tau + \| \nabla_{\Ghsm} \ehm \|_{L^\infty(\Ghsm)}) \| \nabla_\Ghsm \ehm \|_{L^2(\Ghsm)}
	\notag\\
	%	&\quad+ h^{-1} \tau \| \nabla_\Ghsm \ehm \|_{L^\infty(\Ghsm)}\| \nabla_\Ghsm \delta_\tau\ehm \|_{L^2(\Ghsm)}
	%	\notag\\
	%	&\quad+ h^{-1} \tau \| \nabla_\Ghsm \ehm \|_{L^\infty(\Ghsm)}\| \nabla_\Ghsm \delta_\tau\ehm \|_{L^2(\Ghsm)}
	%	\notag\\
	&\lesssim  h^{-1}(\tau + h^{2})
	+
	\epsilon^{-1} h^{-1}
	\| \ehm \|_{H^1(\Ghsm)}
	\notag\\
	&\quad+ \epsilon \| I_h\Nbhsm \delta_\tau\ehm \|_{L^2(\Ghsm)}
	%	\notag\\
	%	&\quad
	+ h^{5/8} \| \nabla_\Ghsm I_h\Tbhsm \delta_\tau\ehm \|_{L^2(\Ghsm)}
	\notag
	.
\end{align}
We now use the tangential stability estimate \eqref{eq:tan_stab_e} to bound the tangential error velocity $I_h\Tbhsm \delta_\tau\ehm$ in terms of the normal error velocity $I_h\Nbhsm \delta_\tau\ehm$. The resulting contribution involving $I_h\Nbhsm \delta_\tau\ehm$ has a sufficiently small coefficient and can therefore be absorbed into the left-hand side. This gives
\begin{align}\label{eq:vel-est-L2N}
	\| I_h \Nbhsm \delta_\tau\ehm \|_{L^2(\Ghsm)}
	\sim \| \delta_\tau\ehm \cdot \nbhsm \|_{L_h^2(\Ghsm)}
	%		\notag\\
	%	&\lesssim  h^{-1}(\tau + h^{2}) +  h^{-1} \| \nabla_\Ghsm \ehm \|_{L^2(\Ghsm)}
	%	\notag\\
	%	&\quad
	%	+
	%	h^{-1}\tau \| \nabla_\Ghsm \ehm \|_{L^\infty(\Ghsm)} \| \nabla_\Ghsm \delta_\tau\ehm \|_{L^2(\Ghsm)}
	%	\notag\\
	\lesssim  h^{-1}(\tau + h^{2}) +  h^{-1} \| \nabla_\Ghsm \ehm \|_{L^2(\Ghsm)}
	.
\end{align}
Combining this bound with the tangential stability estimate \eqref{eq:tan_stab_e} yields
\begin{align}
	&\label{eq:vel-est-H1T}
	\| \nabla_{\Ghsm} I_h \Tbhsm \delta_\tau\ehm \|_{L^2(\Ghsm)}
	+
	\| \nabla_{\Ghsm} I_h \Tsm \delta_\tau\ehm \|_{L^2(\Ghsm)}
	\notag\\
	&
	\lesssim
	\tau + h^2
	+
	\| \nabla_\Ghsm \ehm \|_{L^2(\Ghsm)}
	%	\notag\\
	%	&\quad
	+(1 + h^{-2} \| \nabla_\Ghsm \ehm \|_{L^2(\Ghsm)})
	\notag\\
	&\qquad\times
	\big( h^{-1}(\tau + h^{2}) +  h^{-1} \| \nabla_\Ghsm \ehm \|_{L^2(\Ghsm)}  \big)
	\notag\\
	&\lesssim
	h^{-1}(\tau + h^{2})
	+
	h^{-1} \| \nabla_\Ghsm \ehm \|_{L^2(\Ghsm)}
	+
	h^{-3} \| \nabla_\Ghsm \ehm \|_{L^2(\Ghsm)}^2
	%	\notag\\
	%	&\quad
	%	+
	%	(1 + h^{-1} + h^{-3}(\tau+h^2)) \| \nabla_\Ghsm \ehm \|_{L^2(\Ghsm)}
	%	\notag\\
	%	&\quad
	%	+
	%	h^{-3} \| \nabla_\Ghsm \ehm \|_{L^2(\Ghsm)}^2
	,
\end{align}
and hence the following estimate for the full error velocity in $L^2$:
\begin{align}
	\label{eq:vel-est-L2}
	\|  \delta_\tau\ehm \|_{L^2(\Ghsm)}
	&\lesssim
	\| I_h \Nbhsm \delta_\tau\ehm \|_{L^2(\Ghsm)}
	+
	\| \nabla_\Ghsm I_h \Tbhsm \delta_\tau\ehm \|_{L^2(\Ghsm)}
	\notag\\
	&\lesssim
	h^{-1}(\tau + h^{2})
	+
	h^{-1} \| \nabla_\Ghsm \ehm \|_{L^2(\Ghsm)}
	+
	h^{-3} \| \nabla_\Ghsm \ehm \|_{L^2(\Ghsm)}^2
	,
\end{align}
and the corresponding $H^1$ estimate
\begin{align}
	\label{eq:vel-est-H1}
	\|  \delta_\tau\ehm \|_{H^1(\Ghsm)}
	&\lesssim
	h^{-1}
	\| I_h \Nbhsm \delta_\tau\ehm \|_{L^2(\Ghsm)}
	+
	\| \nabla_\Ghsm I_h \Tbhsm \delta_\tau\ehm \|_{L^2(\Ghsm)}
	\notag\\
	&\lesssim
	h^{-2}(\tau + h^{2})
	+
	h^{-2} \| \nabla_\Ghsm \ehm \|_{L^2(\Ghsm)}
	+
	h^{-3} \| \nabla_\Ghsm \ehm \|_{L^2(\Ghsm)}^2
	.
\end{align}
Combining the velocity estimates \eqref{eq:vel-est-H1T} and \eqref{eq:vel-est-L2} with the induction hypothesis \eqref{eq:ind_hypo1} and the step-size condition $\tau\leq c_0 h^{2}$ gives the preliminary bounds
\begin{align}\label{eq:a-priori-delta-e}
	\| \nabla_\Ghsm I_h\Tbhsm \delta_\tau\ehm \|_{L^2(\Ghsm)}
	+
	\| \nabla_\Ghsm I_h\Tsm \delta_\tau\ehm \|_{L^2(\Ghsm)}
	\lesssim h^{1/2}
	,\qquad
	\| \delta_\tau\ehm \|_{L^2(\Ghsm)}
	\lesssim h^{1/2} .
\end{align}

%Under the step-size condition $\tau\simeq h^2$, the above estimate simplifies to
%\begin{align}\label{eq:vel-est-H1T}
%	\| \nabla_{\Ghsm} I_h \Tsm \delta_\tau\ehm \|_{L^2(\Ghsm)}
%	&\lesssim
%	h+ h^{-1} \| \nabla_\Ghsm \ehm \|_{L^2(\Ghsm)} + h^{-5/2} \| \nabla_\Ghsm \ehm \|_{L^2(\Ghsm)}^2
%	.
%\end{align}
%
%Estimates \eqref{eq:vel-est-L2N} and \eqref{eq:vel-est-H1T} control, respectively, the full velocity in $L^2$ and the tangential velocity in $H^1$. They are of vital importance in the later shape regularity analysis in Section~\ref{sec:bbd}.

\subsection{Estimates for $\eM$ and $\ehM$}

Eq. \eqref{eq:a-priori-delta-e} yields
\begin{align}
%	\label{eq:e_NT}
	\| \eM - \ehm \|_{L^2(\Ghsm)}
	\lesssim \tau + \tau \| \dtem \|_{L^2(\Ghsm)}
%	\notag\\
%	&
	\lesssim \tau  ,
	\notag
\end{align}
and, similarly, under the step-size condition $\tau\leq c_0 h^2$, we obtain
\begin{align*}
	\| e_h^{m+1}-\hat e_{h}^m\|_{H^1(\Ghsm)}
	&\lesssim \tau + \tau \| \dtem \|_{H^1(\Ghsm)}  \\
	&\lesssim \tau + h^{-2}\tau  (\tau + h^{2}) + h^{-2} \tau \| \nabla_\Ghsm \ehm \|_{L^2(\Ghsm)}
	+ h^{-3} \tau \| \nabla_\Ghsm \ehm \|_{L^2(\Ghsm)}^2
	\\
	&\lesssim \tau + \| \nabla_\Ghsm \ehm \|_{L^2(\Ghsm)} .
\end{align*}
Using the triangle inequality, these estimates imply
\begin{align}
	\| e_h^{m+1}\|_{L^2(\Ghsm)}
	&\lesssim \tau + \| \ehm \|_{L^2(\Ghsm)} , \label{hat-ehm-L2} \\
	\| e_h^{m+1}\|_{H^1(\Ghsm)}
	&\lesssim \tau + \| \nabla_\Ghsm \ehm \|_{L^2(\Ghsm)} .
	\label{hat-ehm-H1}
\end{align}

By construction, $|\ehm(p)| \leq |e_h^m(p)|$ at every finite element node $p\in\mathcal N(\Ghsm)$. This nodal inequality gives the corresponding bound in the mass-lumped discrete norm (see Appendix~\ref{sec:disc-norm}):
\begin{align}
%	\label{eq:hat-stab}
	\| \ehm \|_{L^p_h(\hat\Gamma_{h,*}^{m})} \leq \| e_h^m \|_{L^p_h(\hat\Gamma_{h,*}^{m})}
	\quad\forall p\in[1,\infty]
	. \notag
\end{align}
The geometric relations \eqref{eq:geo_rel_1}--\eqref{eq:geo_rel_2} and the inverse inequality also give the following a priori $H^1$ estimate for $\ehM$:
\begin{align}
	\| \ehM \|_{H^1(\Ghsm)}
	&\lesssim
	\| \eM \|_{H^1(\Ghsm)}
	+
	h^{-1}\| \eM \|_{L^\infty(\Ghsm)} \| \eM \|_{L^2(\Ghsm)}
	\lesssim
	\| \eM \|_{H^1(\Ghsm)}
	%	\lesssim h^{3/4}
%	\label{eq:ehM-H1-prior}
	\notag
	.
\end{align}

\subsection{Displacement estimates}\label{sec:disc-est}

The velocity estimates also allow us to compare $\Ghm, \GhM, \Ghsm, \GhsM$ and $\Gamma_{h,*}^{m+1}$. By Lemma~\ref{lemma:norm-equiv}, the $L^p$ and $W^{1,p},p\in[1,\infty],$ norms on these curves are equivalent provided that their pairwise displacements are sufficiently small in the $W^{1,\infty}$ norm.

We first obtain a preliminary bound for the displacement between consecutive consistency curves:
\begin{align}
%	\label{eq:hat_X_s_diff}
	&\| \hat X_{h,*}^{m+1} - \hat X_{h,*}^{m} \|_{L^{\infty}(\Ghsm)} \notag\\
	&\le
	\| \hat X_{h,*}^{m+1} - X_{h}^{m+1} \|_{L^{\infty}(\Ghsm)}
	+\| X_{h}^{m+1} - X_{h}^{m} \|_{L^{\infty}(\Ghsm)}
	+ \| X_{h}^{m} - \hat X_{h,*}^{m} \|_{L^{\infty}(\Ghsm)} \notag\\
	&=
	\| \hat e_{h}^{m+1} \|_{L^{\infty}(\Ghsm)}
	+ \|\eM - \ehm + \tau I_h(-H^mn^m + g^m)  \|_{L^{ \infty}(\Ghsm)}
	+\| \hat e_{h}^m \|_{L^{\infty}(\Ghsm)} \notag\\
	&\lesssim
	\| \hat e_{h}^{m+1} \|_{L^{\infty}(\Ghsm)}
	+ \|\eM \|_{L^{ \infty}(\Ghsm)}
	+\| \hat e_{h}^m \|_{L^{\infty}(\Ghsm)} + \tau  \notag\\
	&\lesssim
	\tau
	+
	\| \ehm \|_{L^{ \infty}(\Ghsm)}
	+
	\tau\| \delta_\tau\ehm \|_{L^{ \infty}(\Ghsm)}
	\lesssim h^{7/4}
	. \notag
\end{align}
To sharpen this bound, we use the decomposition
\begin{align}\label{eq:X-hat-diff-decomp}
	\hat X_{h,*}^{m+1} - \hat X_{h,*}^{m}
	&=
	I_h\Nsm(\hat X_{h,*}^{m+1} - \hat X_{h,*}^{m})
	+
	I_h\Tsm(\hat X_{h,*}^{m+1} - \hat X_{h,*}^{m})
	\notag\\
	&=
	I_h((Y^{m + 1} - {\rm id}_\Gm)\circ a^m|_\Ghsm)
	+
	\rho_h^m
	\notag\\
	&\quad
	+
	I_h\Tsm( X_{h}^{m+1} - X_{h}^{m})
	\notag\\
	&\quad
	-
	I_h\Tsm((N_*^{m+1}\circ\hat X_{h,*}^{m+1}-N_*^{m}\circ\hat X_{h,*}^{m})\ehM)
	\notag\\
	&=
	I_h((Y^{m + 1} - {\rm id}_\Gm)\circ a^m|_\Ghsm)
	+
	\rho_h^m
	\notag\\
	&\quad
	+
	I_h\Tsm( \eM - \ehm - \tau I_h\big(v^m \circ a^m|_\Ghsm \big))
	+
	\tau
	I_h\Tsm \big((v^m-H^m n^m + g^m) \circ a^m|_\Ghsm \big)
	\notag\\
	&\quad
	-
	I_h\Tsm((N_*^{m+1}\circ\hat X_{h,*}^{m+1}-N_*^{m}\circ\hat X_{h,*}^{m})\ehM)
	\notag\\
	&=
	I_h((Y^{m + 1} - {\rm id}_\Gm)\circ a^m|_\Ghsm)
	+
	\rho_h^m
	\notag\\
	&\quad
	+
	\tau I_h\Tsm\dtem
	+
	\tau I_h\Tsm \big((v^m-H^m n^m + g^m) \circ a^m|_\Ghsm \big)
	\notag\\
	&\quad
	-
	I_h\Tsm((N_*^{m+1}\circ\hat X_{h,*}^{m+1}-N_*^{m}\circ\hat X_{h,*}^{m})\ehM)
	.
\end{align}
Taking the $L^\infty$ norm in \eqref{eq:X-hat-diff-decomp} gives
\begin{align}
	\| \hat X_{h,*}^{m+1} - \hat X_{h,*}^{m} \|_{L^\infty(\Ghsm)}
	&\lesssim
	\tau
	+
	\| \rho_h^m  \|_{L^\infty(\Ghsm)}
	%	\notag\\
	%	&\quad
	+
	\tau \| I_h\Tsm\dtem  \|_{L^\infty(\Ghsm)}
	\notag\\
	&\quad
	+
	\|
	I_h\Tsm((N_*^{m+1}\circ\hat X_{h,*}^{m+1}-N_*^{m}\circ\hat X_{h,*}^{m})\ehM)
	\|_{L^\infty(\Ghsm)}
	\notag\\
	&\lesssim
	\tau
	%	 \notag\\
	%	 &\quad
	+
	\tau \| I_h\Tsm\dtem  \|_{L^\infty(\Ghsm)}
	%	 \notag\\
	%	 &\quad
	+
	h^{7/4}
	\| \hat X_{h,*}^{m+1} - \hat X_{h,*}^{m} \|_{L^\infty(\Ghsm)}
	%	 \|
	%	 I_h\Tsm\ehM
	%	 \|_{L^\infty(\Ghsm)}
	, \notag
\end{align}
where we have used Lemma~\ref{lemma:super_conv-nonlinear}, the bound $\| \ehM \|_{L^\infty(\Ghsm)}\lesssim h^{7/4}$, and the estimate
\begin{align}
	\| \rho_h^m  \|_{L^\infty(\Ghsm)}
	&\lesssim
	\tau^2 +  \| | I_h T_*^m (\hat X_{h,*}^{m + 1} -\hat X_{h,*}^{m})|^2 \|_{L^\infty(\Ghsm)}
	\notag\\
	&\lesssim
	\tau^2 + {h^{7/4}} \| I_h T_*^m (\hat X_{h,*}^{m + 1} -\hat X_{h,*}^{m}) \|_{L^\infty(\Ghsm)}
	\notag .
\end{align}
Absorbing the final term and using the velocity estimate \eqref{eq:vel-est-H1T}, we obtain
\begin{align}
%	\label{eq:X-hat-disp-Linf}
	\| \hat X_{h,*}^{m+1} - \hat X_{h,*}^{m} \|_{L^\infty(\Ghsm)}
	&\lesssim
	\tau
	. \notag
\end{align}
This improves the bound for $\rho_h^m$ to
\begin{align}
	\| \rho_h^m  \|_{L^\infty(\Ghsm)}
	\lesssim
	\tau^2 +  \| | I_h T_*^m (\hat X_{h,*}^{m + 1} -\hat X_{h,*}^{m})|^2 \|_{L^\infty(\Ghsm)}
	\lesssim
	\tau^2 . \notag
\end{align}
Next, taking the $W^{1,\infty}$ norm in \eqref{eq:X-hat-diff-decomp} and applying the inverse inequality gives
\begin{align}
	\| \hat X_{h,*}^{m+1} - \hat X_{h,*}^{m} \|_{W^{1,\infty}(\Ghsm)}
	&\lesssim
	\tau
	+
	\| \rho_h^m  \|_{W^{1,\infty}(\Ghsm)}
	%	\notag\\
	%	&\quad
	+
	\tau \| I_h\Tsm\dtem  \|_{W^{1,\infty}(\Ghsm)}
	\notag\\
	&\quad
	+
	\|
	I_h\Tsm((N_*^{m+1}\circ\hat X_{h,*}^{m+1}-N_*^{m}\circ\hat X_{h,*}^{m})\ehM)
	\|_{W^{1,\infty}(\Ghsm)}
	\notag\\
	&\lesssim
	\tau + h^{3/4} \| \hat X_{h,*}^{m+1} - \hat X_{h,*}^{m} \|_{W^{1,\infty}(\Ghsm)}
	. \notag
\end{align}
A further absorption argument yields
\begin{align}\label{eq:X-hat-disp-W1inf}
	\| \hat X_{h,*}^{m+1} - \hat X_{h,*}^{m} \|_{W^{1,\infty}(\Ghsm)}
	&\lesssim
	\tau
	.
\end{align}
Combining this bound with Lemma~\ref{lemma:super_conv-nonlinear} and the Lipschitz continuity of the normal projection yields
\begin{align}\label{eq:X-hat-disp-W1inf1}
	\| I_h (\NsM\circ \hat X_{h,*}^{m+1} -\Nsm\circ \hat X_{h,*}^{m}) \|_{W^{1,\infty}(\Ghsm)}
	&\lesssim
	\tau
	.
\end{align}

The preceding displacement estimates and Lemma~\ref{lemma:norm-equiv} therefore establish the equivalence of the $L^p$ and $W^{1,p}$ norms, for all $p\in[1,\infty]$, on $\Ghm, \GhM, \Ghsm, \GhsM$ and $\Gamma_{h,*}^{m+1}$ for finite element functions with a common nodal vector.

\section{Super-convergent velocity estimates and the convergence proof}\label{sec:bbd_vel_Abel}

\subsection{Bihomotopy family of curves}\label{sec:bihomo}

We use \(\Ghso\) as the fixed reference curve and extend the superscript
\(m+\alpha\) to denote the time \(t_{m}+\alpha\tau\).
For \(0\leq\alpha,\theta\leq1\), we define the two-parameter homotopy family
\begin{align}
	\hat X_{h,*}^{m+\alpha,\theta}
	&:=[(1-\theta)I_h+\theta]
	\left[a^{m+\alpha}\circ
	\bigl((1-\alpha)\hat X_{h,*}^{m}+\alpha\hat X_{h,*}^{m+1}\bigr)\right],
	\label{bhk1:homotopy}\\
	\hat\Gamma_{h,*}^{m+\alpha,\theta}
	&:=\hat X_{h,*}^{m+\alpha,\theta}(\Ghso).
	\notag
\end{align}
The parameter \(\alpha\) connects consecutive time levels. For fixed
\(\alpha\), the parameter \(\theta\) connects the affine interpolant to its smooth
projected parametrization, whose image at \(\theta=1\) is
\(\Gamma^{m+\alpha}\).
The four corners of this family are given by
\begin{equation}
	\begin{array}{c|c|c|c}
		(\alpha,\theta)&\text{map from }\Ghso&\text{image}&\text{test}\\ \hline
		(0,0)&\hat X_{h,*}^{m}&\Ghsm&\phi_h\text{ at }m\\
		(0,1)&a^{m}\circ\hat X_{h,*}^{m}&\Gamma^{m}&\phi_h^{\ell,m}\\
		(1,0)&\hat X_{h,*}^{m+1}&\GhsM&\phi_h\text{ at }m+1\\
		(1,1)&a^{m+1}\circ\hat X_{h,*}^{m+1}&\Gamma^{m+1}&\phi_h^{\ell,m+1}.
	\end{array}
	\label{bhk1:corners}
\end{equation}

We also introduce
\(\hat X_{h,*}^{m+\alpha}=(1-\alpha)\hat X_{h,*}^{m}+\alpha\hat X_{h,*}^{m+1}\in [S_h(\Ghsm)]^2\), which parametrizes the linearly interpolated curve $\hat \Gamma_{h,*}^{m+\alpha}$ with transport velocity $\hat v_{h,*}^m:= (\hat X_{h,*}^{m+1} - \hat X_{h,*}^{m})/\tau$.
In general, $\hat X_{h,*}^{m+\alpha,\theta}\notin [S_h(\Ghso)]^2$ unless $\theta=0$, and $\hat X_{h,*}^{m+\alpha,0} \neq \hat X_{h,*}^{m+\alpha}$ unless $\alpha=0,1$.

Let \(\phi_h^0\) denote the fixed reference realization of the nodal
vector of \(\phi_h\). We define $\phi_h^{\alpha,\theta}$ by transporting \(\phi_h^0\) through $\hat X_{h,*}^{m+\alpha,\theta}$:
\begin{equation}
	\phi_h^{\alpha,\theta}\circ
	\hat X_{h,*}^{m+\alpha,\theta}=\phi_h^0. \notag
%	\label{bhk1:test-transport}
\end{equation}
Thus \(\mda\phi_h^{\alpha,\theta}=\mdc\phi_h^{\alpha,\theta}=0\).

\begin{lemma}
	\label{bhk1:velocities}
	On the reference domain $\Ghso$, define the bihomotopy velocities by
	\begin{align}
		V_\alpha := \partial_\alpha \hat X_{h,*}^{m+\alpha,\theta},
		\qquad
		V_\theta := \partial_\theta \hat X_{h,*}^{m+\alpha,\theta},
		\qquad
		V_{\alpha,\theta} := \partial_{\alpha,\theta} \hat X_{h,*}^{m+\alpha,\theta} . \notag
	\end{align}
	These velocities satisfy the following estimates:
	\begin{align}
		\|V_\theta\|_{L^\infty(\Ghso)}
		+h\|\nabla_\Ghso V_\theta\|_{L^\infty(\Ghso)}&\lesssim h^2,
		\label{bhk1:Vtheta-bound}\\
		\|V_\alpha\|_{H^1(\Ghso)}
		+\|\nabla_\Ghso^2V_\alpha\|_{L^2(\Ghso)}
		&\lesssim\tau(1+\|\hat v_{h,*}^{m}\|_{H^1(\Ghsm)}),
		\label{bhk1:Valpha-bound}\\
		\|V_{\alpha,\theta}\|_{L^2(\Ghso)}
		+h\|\nabla_\Ghso V_{\alpha,\theta}\|_{L^2(\Ghso)}
		&\lesssim\tau h^2(1+\|\hat v_{h,*}^{m}\|_{H^1(\Ghsm)}).
		\label{bhk1:Vmixed-bound}
	\end{align}
	Here the second derivatives in $\|\nabla_\Ghso^2V_\alpha\|_{L^2(\Ghso)}$ are taken elementwise.
\end{lemma}
\begin{proof}
	We prove \eqref{bhk1:Vmixed-bound}; the proofs of \eqref{bhk1:Vtheta-bound} and \eqref{bhk1:Valpha-bound} follow by similar, simpler arguments and are omitted.
	
	Differentiating the bihomotopy map \eqref{bhk1:homotopy} gives
	\begin{align}
		V_\theta&=(1-I_h)(a^{m+\alpha}\circ \hat X_{h,*}^{m+\alpha}),\notag\\
		V_\alpha&=\tau[(1-\theta)I_h+\theta]
		\bigl[(\partial_ta)^{m+\alpha}\circ \hat X_{h,*}^{m+\alpha}
		+(\nabla a^{m+\alpha}\circ \hat X_{h,*}^{m+\alpha})\hat v_{h,*}^{m}\bigr],\notag\\
		V_{\alpha,\theta}&=\tau(1-I_h)
		\bigl[(\partial_ta)^{m+\alpha}\circ \hat X_{h,*}^{m+\alpha}
		+(\nabla a^{m+\alpha}\circ \hat X_{h,*}^{m+\alpha})\hat v_{h,*}^{m}\bigr]. \notag
%		\label{bhk1:V-formulas}
	\end{align}
	In these chain rules, \(\nabla\) denotes differentiation in the ambient Euclidean space.
	
	Let
	\begin{align}
		f = (\partial_ta)^{m+\alpha}\circ \hat X_{h,*}^{m+\alpha}
		+(\nabla a^{m+\alpha}\circ \hat X_{h,*}^{m+\alpha})\hat v_{h,*}^{m} .
	\end{align}
	Since \(\nabla_\Ghso^2 \hat X_{h,*}^{m+\alpha}=0\) and
	\(\nabla_\Ghso^2 \hat v_{h,*}^{m}=0\) on each element of the reference domain $\Ghso$,
	its second derivative is
	\begin{align*}
		\nabla_\Ghso^2 f={}&\nabla^2\partial_ta(t_{m}+\alpha\tau,\hat X_{h,*}^{m+\alpha})[\nabla_\Ghso \hat X_{h,*}^{m+\alpha},\nabla_\Ghso \hat X_{h,*}^{m+\alpha}]
		\\
		&+2\nabla^2a(t_{m}+\alpha\tau,\hat X_{h,*}^{m+\alpha})[\nabla_\Ghso \hat X_{h,*}^{m+\alpha},\nabla_\Ghso\hat v_{h,*}^{m}]\\
		&+\nabla^3a(t_{m}+\alpha\tau,\hat X_{h,*}^{m+\alpha})[\nabla_\Ghso \hat X_{h,*}^{m+\alpha},\nabla_\Ghso \hat X_{h,*}^{m+\alpha},\hat v_{h,*}^{m}]
		,
	\end{align*}
	where all coefficients are evaluated at time \(t_{m}+\alpha\tau\).
	
	Estimate \eqref{eq:X-hat-disp-W1inf} and the triangle inequality give $\| \hat X_{h,*}^{m+\alpha} \|_{W^{1,\infty}(\Ghso)}\leq C_\kl (1+\tau) \lesssim 1$.
	Consequently,
	\(\|f\|_{H_h^2(\Ghso)}\lesssim1+\|\hat v_{h,*}^{m}\|_{H^1(\Ghso)}\), and the interpolation error estimates in Lemma~\ref{lemma:Ih} yield \eqref{bhk1:Vmixed-bound}.
\end{proof}

Since $V_\theta$ and $V_{\alpha,\theta}$ are defined using $(1-I_h)$, they satisfy the following local integration by parts identities; see also Lemma~\ref{Lemma-GLW}.
\begin{lemma}
	\label{bhk1:velocities-super}
	For any piecewise smooth function $f$ on $\Ghso$, the following local integration by parts identities hold:
	\begin{align}
		\int_\Ghso \nabla_\Ghso V_\theta f &=-\int_\Ghso V_\theta \nabla_\Ghso f ,
%		\label{bhk1:Vtheta-bound-super}
		\notag
		\\
		\int_\Ghso \nabla_\Ghso V_{\alpha,\theta} f &=-\int_\Ghso V_{\alpha,\theta} \nabla_\Ghso f . \notag
%		\label{bhk1:Vmixed-bound-super}
	\end{align}
\end{lemma}

\subsection{Estimates for the discrete time derivative of $\mathscr D_{22}^m$, $\mathscr D_{23}^m$ and $\mathscr D_{13}^m$}

Let
\begin{equation}
	F(\alpha,\theta):=
	\int_{\hat\Gamma_{h,*}^{m+\alpha,\theta}}
	\nabla_{\hat\Gamma_{h,*}^{m+\alpha,\theta}}\id\cdot
	\nabla_{\hat\Gamma_{h,*}^{m+\alpha,\theta}}\phi_h^{\alpha,\theta}. \notag
%	\label{bhk1:F}
\end{equation}
The corner relations \eqref{bhk1:corners} and the fundamental theorem of calculus give
\begin{align}
	&\delta_\tau(\mathscr D_{22}^m+\mathscr D_{23}^m)(\phi_h):=\frac{(\mathscr D_{22}^{m+1}(\phi_h)+\mathscr D_{23}^{m+1}(\phi_h)) - (\mathscr D_{22}^m(\phi_h)+\mathscr D_{23}^m(\phi_h))}{\tau}
	\notag\\
	&=\tau^{-1}(F(1,0)-F(1,1))-\tau^{-1}(F(0,0)-F(0,1))
	=-\frac{1}{\tau}\int_0^1\int_0^1 F_{\alpha\theta}(\alpha,\theta)
	\,\d\theta\,\d\alpha.
	\label{bhk1:rectangle}
\end{align}
\begin{lemma}
	\label{bhk1:geometric-estimate}
	The sum of $\mathscr D_{22}^m$ and $\mathscr D_{23}^m$ satisfies the following discrete time-derivative estimate:
	\begin{equation}
		|\delta_\tau(\mathscr D_{22}^m+\mathscr D_{23}^m)(\phi_h)|
		\lesssim h^2
		(1+\|\hat v_{h,*}^{m}\|_{H^1(\Ghsm)})
		\|\phi_h\|_{H^1(\Ghsm)}.
		\label{bhk1:geo-estimate}
	\end{equation}
\end{lemma}
\begin{proof}
	On the reference domain $\Ghso$, the calculus identities in Lemma~\ref{lemma:ud} show that differentiation with respect to $\alpha$ or $\theta$ introduces the linear dependence on the velocity gradient $\nabla_\Ghso V_{\alpha}$ or $\nabla_\Ghso V_{\theta}$. Mixed differentiation $\partial_{\alpha,\theta}$ therefore produces terms that depend linearly on either $\nabla_\Ghso V_{\alpha,\theta}$ or $\nabla_\Ghso V_{\alpha}\otimes \nabla_\Ghso V_{\theta}$. Thus $F_{\alpha\theta}$ has the form
	\begin{align}
		F_{\alpha\theta}
		&=
		\int_\Ghso
		\Big(
		F_1(\hat X_{h,*}^{m+\alpha,\theta}, \nabla_\Ghso \hat X_{h,*}^{m+\alpha,\theta}, \nabla_\Ghso V_{\alpha,\theta}\otimes\nabla_\Ghso \phi_h^0)
		\notag\\
		&\quad+
		F_2(\hat X_{h,*}^{m+\alpha,\theta}, \nabla_\Ghso \hat X_{h,*}^{m+\alpha,\theta}, \nabla_\Ghso V_\theta\otimes \nabla_\Ghso V_\alpha\otimes\nabla_\Ghso \phi_h^0)
		\Big)
		\notag ,
	\end{align}
	where $F_1$ and $F_2$ are piecewise smooth in their first two arguments and linear in the third.
	Using this linearity and the fact that the second derivatives of linear finite element functions vanish elementwise, we apply Lemma~\ref{bhk1:velocities-super} to transfer the derivatives from $V_{\alpha\theta}$ and $V_\theta$ to the remaining factors without boundary terms:
	\begin{align}
		F_{\alpha\theta}
		&=
		\int_\Ghso
		\Big(
		F_3(\hat X_{h,*}^{m+\alpha,\theta}, \nabla_\Ghso \hat X_{h,*}^{m+\alpha,\theta}, \nabla_\Ghso^2 \hat X_{h,*}^{m+\alpha,\theta}, V_{\alpha,\theta}\otimes\nabla_\Ghso \phi_h^0)
		\notag\\
		&\quad+
		F_4(\hat X_{h,*}^{m+\alpha,\theta}, \nabla_\Ghso \hat X_{h,*}^{m+\alpha,\theta}, \nabla_\Ghso^2 \hat X_{h,*}^{m+\alpha,\theta}, V_\theta\otimes \nabla_\Ghso V_\alpha\otimes\nabla_\Ghso \phi_h^0)
		\notag\\
		&\quad+
		F_5(\hat X_{h,*}^{m+\alpha,\theta}, \nabla_\Ghso \hat X_{h,*}^{m+\alpha,\theta}, V_\theta\otimes \nabla_\Ghso^{2} V_\alpha\otimes\nabla_\Ghso \phi_h^0)
		\Big)
		\notag ,
	\end{align}
	where $F_3$, $F_4$ and $F_5$ are linear in their last argument and smooth in the remaining arguments.
	Using the bound $\| \hat X_{h,*}^{m+\alpha} \|_{W^{1,\infty}(\Ghso)}\leq C_\kl (1+\tau) \lesssim 1$, we obtain
	\begin{align}
		|\delta_\tau(\mathscr D_{22}^m+\mathscr D_{23}^m)(\phi_h)|
		&\lesssim
		\tau^{-1}
		\| V_{\alpha,\theta} \|_{L^2(\Ghso)}
		\| \nabla_\Ghso \phi_h\|_{L^2(\Ghso)}
		\notag\\
		&\quad+
		\tau^{-1}
		\| \nabla_\Ghso V_{\alpha} \|_{L^2(\Ghso)}
		\| V_{\theta} \|_{L^{\infty}(\Ghso)}
		\| \nabla_\Ghso \phi_h\|_{L^2(\Ghso)}
		\notag\\
		&\quad+
		\tau^{-1}
		\| \nabla_\Ghso^{2} V_{\alpha} \|_{L^2(\Ghso)}
		\| V_{\theta} \|_{L^{\infty}(\Ghso)}
		\| \nabla_\Ghso \phi_h\|_{L^2(\Ghso)}
		\notag\\
		&\lesssim h^2
		(1+\|\hat v_{h,*}^{m}\|_{H^1(\Ghsm)})
		\| \nabla_\Ghsm \phi_h\|_{L^2(\Ghsm)}.
		\notag
	\end{align}
\end{proof}

\begin{lemma}
	\label{bhk1:D13-estimate}
	The discrete time derivative of $\mathscr D_{13}^m$ satisfies
	\begin{equation}
		|\delta_\tau\mathscr D_{13}^m(\phi_h)|
		\lesssim h^2(1+\|\hat v_{h,*}^{m}\|_{H^1(\Ghsm)})
		\|\phi_h\|_{H^1(\Ghsm)}.
		\label{bhk1:D13-bound}
	\end{equation}
\end{lemma}
\begin{proof}
	Consider
	\begin{align}
		f(t_m+\alpha\tau,{\rm id}_{\hat \Gamma_{h,*}^{m+\alpha}}):=H_*(t_m+\alpha\tau,{\rm id}_{\hat \Gamma_{h,*}^{m+\alpha}})
		n_*(t_m+\alpha\tau,{\rm id}_{\hat \Gamma_{h,*}^{m+\alpha}}). \notag
%		\label{bhk1:curvature-extension}
	\end{align}
	Applying the calculus identities in Lemma~\ref{lemma:ud}, we obtain
	\begin{align}
		\delta_\tau\mathscr D_{13}^m(\phi_h)
		&=-\frac{1}{\tau}\int_0^1 \frac{\d}{\d\alpha} \left(\int_{\hat\Gamma_{h,*}^{m+\alpha}}^{h}-\int_{\hat\Gamma_{h,*}^{m+\alpha}}\right)
		f(t_m+\alpha\tau,{\rm id}_{\hat \Gamma_{h,*}^{m+\alpha}})
		\cdot\phi_h\,\d\alpha
		\notag\\
		&=-\int_0^1\left(\int_{\hat\Gamma_{h,*}^{m+\alpha}}^{h}-\int_{\hat\Gamma_{h,*}^{m+\alpha}}\right)
		\Big(
		\partial_t f
		+
		\nabla f\,\hat v_{h,*}^{m}
		+f\nabla_{\hat\Gamma_{h,*}^{m+\alpha}}\cdot \hat v_{h,*}^{m} \Big)
		\cdot\phi_h\,\d\alpha, \notag
%		\label{bhk1:D13-difference-exact}
	\end{align}
	where the integrand on the right-hand side is evaluated at $t=t_m+\alpha\tau$.
	
	The result follows from the super-approximation estimate in Lemma~\ref{lemma:super_conv}, H\"older's inequality, the smoothness of $f$, and norm equivalence on $\hat\Gamma_{h,*}^{m+\alpha}, \alpha\in[0,1]$.
\end{proof}

\begin{lemma}
	\label{bhk1:main}
	Together with \eqref{eq:X-hat-disp-W1inf}, the preceding estimates yield
	\begin{align}
		|\delta_\tau\mathscr D_{13}^m(\phi_h)|
		&\lesssim h^2
		\|\phi_h\|_{H^1(\Ghsm)},\label{bhk1:dt-D13}\\
		|\delta_\tau\mathscr D_{22}^m(\phi_h) + \delta_\tau\mathscr D_{23}^m(\phi_h)|
		&\lesssim
		h^2
		\|\phi_h\|_{H^1(\Ghsm)}.
		\label{bhk1:dt-D2}
	\end{align}
\end{lemma}

\subsection{Super-convergence of error velocity}\label{sec:bbd_vel_sup}

We now identify the summation-by-parts structure obtained by testing the error equation \eqref{eq:err_eq2} with the normal error velocity $\phi_h = I_h\Nbhsm\dtem$. This structure underlies the super-convergence estimates in Section~\ref{sec:main-err}.

Taking $\phi_h = I_h\Nbhsm\dtem$ in \eqref{eq:err_eq2} gives
\begin{align}\label{eq:err2}
	& \int_{\hat\Gamma_{h, *}^{m}}^h \Big( \frac{e_h^{m+1}-\hat e_h^m}{\tau} -  I_h \Tsm v^m\Big) \cdot \bar n_{h, *}^m\, \dtem\cdot \bar n_{h, *}^m \notag\\
	&= - \int_{\hat\Gamma_{h, *}^{m}}^h I_h\Tsm v^m \cdot \bar n_{h, *}^m \,   \dtem \cdot \bar n_{h, *}^m \notag\\
	&\quad + \int_{\hat\Gamma_{h, *}^{m}}^h \frac{e_h^{m+1}-\hat e_h^m}{\tau} \cdot \bar n_{h, *}^m\, \Big(\frac{e_h^{m+1}-\hat e_h^m}{\tau} - I_h\Tsm v^m\Big)\cdot \bar n_{h, *}^m \notag\\
	&= - \int_{\hat\Gamma_{h, *}^{m}}^h I_h\Tsm v^m \cdot \bar n_{h, *}^m \,   \dtem \cdot \bar n_{h, *}^m \notag\\
	&\quad - \mathscr D^m( I_h\Nbhsm \dtem) - \mathscr J^m( I_h\Nbhsm \dtem) - \mathscr B^m(\ehm,  I_h\Nbhsm \dtem) - \mathscr K^m( I_h\Nbhsm \dtem)  \notag\\
	&\quad
	-
	\mathscr A_{h,  *}^N\Big(e_h^{m+1},  I_h\Nbhsm \Big(\frac{e_h^{m+1}-\hat e_h^m}{\tau} - I_h\Tsm v^m\Big)\Big)
	\notag\\
	&\quad
	-
	\mathscr A_{h, *}^T\Big(e_h^{m+1} - \hat{e}_h^m,  I_h\Nbhsm \Big(\frac{e_h^{m+1}-\hat e_h^m}{\tau} - I_h\Tsm v^m\Big)\Big)  \notag\\
	&\quad- \sum_{i=1}^3 \mathscr F_i^m( I_h\Nbhsm \dtem) + \mathscr A_{h, *}^N(\hat e_h^{m}, I_h \bar T_{h,*}^m  I_h\Nbhsm \dtem)
	\notag\\
	&\quad
	+ \mathscr B^m(\hat e_h^{m}, I_h \bar T_{h,*}^m  I_h\Nbhsm \dtem) + \mathscr Q^m(I_h \bar T_{h,*}^m  I_h\Nbhsm \dtem)
	.
\end{align}
The first term on the right-hand side has been estimated in \eqref{eq:mass-ortho}.
%
%The super-approximation of $\nbhsm$ in Lemma~\ref{lemma:n_bar_app} gives
%\begin{align}
%	\Big|\int_{\hat\Gamma_{h, *}^{m}}^h \dtem \cdot \bar n_{h, *}^m\,  I_h\Tsm v^m \cdot \bar n_{h, *}^m \Big|
%	\lesssim h^{\r 2} \| \dtem \cdot \bar n_{h, *}^m \|_{L^2_h(\Ghsm)} .
%\end{align}
For notational convenience, throughout the remainder of this section,
we use the abbreviations
\begin{align*}
	&I_h\NsM\phi_h
	:=I_h[(\NsM\circ\hat X_{h,*}^{m+1})\phi_h],
	\qquad
	I_h\TsM\phi_h
	:=I_h[(\TsM\circ\hat X_{h,*}^{m+1})\phi_h],\\
	&I_h\Nsm\phi_h
	:=I_h[(\Nsm\circ\hat X_{h,*}^{m})\phi_h],
	\qquad
	I_h\Tsm\phi_h
	:=I_h[(\Tsm\circ\hat X_{h,*}^{m})\phi_h].
\end{align*}
For the $\mathscr A_{h,  *}^N$ term, we use the decomposition
\begin{align}
	&\quad \mathscr A_{h,  *}^N\Big(e_h^{m+1},  I_h\Nbhsm\dtem\Big)
	\notag\\
	&=
	\mathscr A_{h, *}\Big(I_h\Nsm e_h^{m+1}, I_h\Nsm\dtem\Big)
	\notag\\
	&\quad
	-
	\mathscr A_{h,  *}^T\Big(I_h\Nsm e_h^{m+1}, I_h\Nsm \dtem\Big)
	\notag\\
	&\quad
	+
	\mathscr A_{h,  *}^N\Big(I_h\Tsm  e_h^{m+1}, I_h\Nsm \dtem\Big)
	\notag\\
	&\quad
	+
	\mathscr A_{h,  *}^N\Big(e_h^{m+1}, I_h(\Nbhsm-\Nsm)\dtem\Big)
	\notag\\
	&=
	\mathscr A_{h,  *}\Big(I_h\NsM e_h^{m+1}, \frac{I_h\NsM\eM - I_h\Nsm e_h^m}{\tau}\Big)
	\notag\\
	&\quad-
	\mathscr A_{h,  *}\Big(I_h(\NsM-\Nsm)e_h^{m+1}, \frac{I_h\NsM\eM - I_h\Nsm e_h^m}{\tau}\Big)
	\notag\\
	&\quad
	-
	\mathscr A_{h,  *}\Big(I_h\Nsm e_h^{m+1}, \frac{I_h(\NsM-\Nsm)\eM }{\tau}\Big)
	\notag\\
	&\quad+
	\mathscr A_{h,  *}\Big( I_h\Nsm e_h^{m+1}, I_h\Nsm\Big(\frac{e_h^m -\ehm}{\tau}-I_h\Tsm v^m\Big)\Big)
	\notag\\
	&\quad
	-
	\mathscr A_{h,  *}^T\Big(I_h\Nsm e_h^{m+1}, I_h\Nsm \dtem\Big)
	\notag\\
	&\quad
	+
	\mathscr A_{h,  *}^N\Big(I_h\Tsm  e_h^{m+1}, I_h\Nsm \dtem\Big)
	\notag\\
	&\quad
	+
	\mathscr A_{h,  *}^N\Big(e_h^{m+1}, I_h(\Nbhsm-\Nsm)\dtem\Big)
	\notag\\
	&=:
	\mathscr A_{h,  *}\Big(I_h\NsM e_h^{m+1}, \frac{I_h\NsM\eM - I_h\Nsm e_h^m}{\tau}\Big)
	+
	\sum_{i=1}^6 A_i^m  .
	\notag
\end{align}
The first term on the right-hand side can be estimated using Young's inequality:
\begin{align}
	&\quad \mathscr A_{h,  *}\Big(I_h\NsM e_h^{m+1}, \frac{I_h\NsM\eM - I_h\Nsm e_h^m}{\tau}\Big)
	\notag\\
	&\geq
	\frac{\| \nabla_\GhsM I_h\NsM e_h^{m+1} \|_{L^2(\GhsM)}^2 - \| \nabla_\Ghsm I_h\Nsm e_h^{m} \|_{L^2(\Ghsm)}^2}{2\tau}
	\notag\\
	&\quad-
	\frac{ \| \nabla_\GhsM I_h\NsM e_h^{m+1} \|_{L^2(\GhsM)}^2 - \| \nabla_\Ghsm I_h\NsM e_h^{m+1} \|_{L^2(\Ghsm)}^2 }{2\tau}
	\notag\\
	&=:
	\frac{\| \nabla_\GhsM I_h\NsM e_h^{m+1} \|_{L^2(\GhsM)}^2 - \| \nabla_\Ghsm I_h\Nsm e_h^{m} \|_{L^2(\Ghsm)}^2}{2\tau}
	+
	A_7^m .
	\notag
\end{align}
We also set
\begin{align}
	A_8^m :=  \mathscr A_{h, *}^T\Big(e_h^{m+1} - \hat{e}_h^m,  I_h\Nbhsm \Big(\frac{e_h^{m+1}-\hat e_h^m}{\tau} - I_h\Tsm v^m \Big)\Big) .
	\notag
\end{align}
Estimates \eqref{eq:X-hat-disp-W1inf} and \eqref{eq:X-hat-disp-W1inf1} show that the changes from $\Nsm$ to $\NsM$ and from $\Ghsm$ to $\GhsM$ contribute a factor $\tau$. Hence
\begin{align}
	| A_1^m |
	&\lesssim
	\| \eM \|_{H^1(\Ghsm)} (\| \eM \|_{H^1(\Ghsm)} + \| \eem \|_{H^1(\Ghsm)}) ,
	\notag\\
	| A_2^m |
	&\lesssim
	\| \eM \|_{H^1(\Ghsm)} (\| \eM \|_{H^1(\Ghsm)} + \| \eem \|_{H^1(\Ghsm)}) ,
	\notag\\
	| A_7^m |
	&\lesssim
	\| \eM \|_{H^1(\Ghsm)}^2
	\notag .
\end{align}
Cancellation estimates in Lemma~\ref{lemma:AT-sup} give
\begin{align}
	|A_4^m|
	&\lesssim
	\| \eM \|_{H^1(\Ghsm)}  \| I_h\Nsm\dtem \|_{L^2(\Ghsm)}
	\notag\\
	|A_5^m|
	&\lesssim
	\| \eM \|_{H^1(\Ghsm)}  \| I_h\Nsm\dtem \|_{L^2(\Ghsm)}
	\notag\\
%	|A_6^m|
%	&\lesssim
%	h \| \eM \|_{H^1(\Ghsm)}  \| \dtem \|_{L^2(\Ghsm)}
%	\notag\\
	|A_8^m|
	&\lesssim
	(\| \eM \|_{H^1(\Ghsm)} + \| \ehm \|_{H^1(\Ghsm)}) \| I_h\Nbhsm\dtem \|_{L^2(\Ghsm)}
	\notag ,
\end{align}
and Lemma \ref{lemma:n_bar_app} and the inverse inequality yield
\begin{align}
	|A_6^m|
	\lesssim
	h \| \eM \|_{H^1(\Ghsm)}  \| \dtem \|_{L^2(\Ghsm)}
	\notag .
\end{align}
The most crucial term $A_3^m$ is controlled by Lemma~\ref{lemma:corrected-displacement-nodal}, proved below:
\begin{align}
	|A_3^m|
	&\lesssim
	\tau \| \eM \|_{H^1(\Ghsm)} .
	\notag
\end{align}
To express the upper bounds for $|A_4^m|$ and $|A_5^m|$ in terms of the discrete normal component, we use Lemma~\ref{lemma:n_bar_app}, the velocity estimate \eqref{eq:vel-est-L2}, and the induction hypothesis \eqref{eq:ind_hypo1} to obtain
\begin{align}
	\| I_h\Nsm\dtem \|_{L^2(\Ghsm)}
	&\leq
	\| I_h\Nbhsm\dtem \|_{L^2(\Ghsm)}
	+
	\| I_h(\Nbhsm-\Nsm)\dtem \|_{L^2(\Ghsm)}
	\notag\\
	&\lesssim
	\| I_h\Nbhsm\dtem \|_{L^2(\Ghsm)}
	+
	h^2
	\| \dtem \|_{L^2(\Ghsm)}
	\notag\\
	&\lesssim
	\| I_h\Nbhsm\dtem \|_{L^2(\Ghsm)}
	+
	h(\tau+h^2)
	+
	h
	\| \nabla_\Ghsm \ehm \|_{L^2(\Ghsm)}
	+
	h^{-1}
	\| \nabla_\Ghsm \ehm \|_{L^2(\Ghsm)}^2
	\notag\\
	&\lesssim
	\| I_h\Nbhsm\dtem \|_{L^2(\Ghsm)}
	+
	h(\tau+h^2)
	+
	h^{3/4}
	\| \nabla_\Ghsm \ehm \|_{L^2(\Ghsm)}
	\notag .
\end{align}
Combining these estimates, we obtain
\begin{align}\label{eq:Aim}
	\Big| \sum_{i=1}^8 A_i^m \Big|
	&\lesssim
	(\tau+h^2)^2
	+
	\epsilon
	\| \dtem \cdot \nbhsm \|_{L_h^2(\Ghsm)}^2
	\notag\\
	&\quad
	+
	\epsilon^{-1}
	\Big(\| \eM \|_{H^1(\Ghsm)}^2
	+
	\| \eem \|_{H^1(\Ghsm)}^2
	+
	\| \ehm \|_{H^1(\Ghsm)}^2\Big)
	.
\end{align}

We next turn to the consistency error and begin with the decomposition
\begin{align}
	\mathscr D^m(I_h\Nbhsm\dtem)
	&=
	\mathscr D^m(I_h\Nbhsm\dtem) - \tilde{\mathscr D}^m(I_h\Nbhsm\dtem)
	\notag\\
	&\quad
	+
	\tilde{\mathscr D}^m(I_h(\Nbhsm-\Nsm)\dtem)
	\notag\\
	&\quad
	+
	\tilde{\mathscr D}^m(I_h\Nsm\dtem)
	\notag\\
	&=:
	D_1^m
	+
	D_2^m
	+
	\tilde{\mathscr D}^m(I_h\Nsm\dtem) . \notag
\end{align}
Estimate \eqref{eq:Dm-tildeDm} gives
\begin{align}
	| D_1^m |
	\lesssim
	(\tau+h^2) \| I_h\Nbhsm\dtem \|_{L^2(\Ghsm)}
	.
	\notag
\end{align}
Estimate \eqref{eq:tildeDm}, the approximation property of $\nbhsm$ in Lemma~\ref{lemma:n_bar_app}, and the velocity estimate \eqref{eq:vel-est-L2} yield
\begin{align}
	| D_2^m |
	\lesssim
	h^3 \| \dtem \|_{L^2(\Ghsm)}
	\lesssim
	h^2(\tau+h^2)
	+
	h^2
	\| \nabla_\Ghsm \ehm \|_{L^2(\Ghsm)}
	+
	\| \nabla_\Ghsm \ehm \|_{L^2(\Ghsm)}^2
	.
	\notag
\end{align}
For the critical term $\tilde{\mathscr D}^m(I_h\Nsm\dtem)$, we isolate a discrete time difference as follows:
\begin{align}
	\tilde{\mathscr D}^m(I_h\Nsm\dtem )&= \tilde{\mathscr D}^m\Big(I_h\Nsm\Big(\frac{e_h^{m+1}-\hat e_h^m}{\tau} - I_h\Tsm v^m\Big) \Big)
	\notag\\
	&
	=
	\frac{\tilde{\mathscr D}^{m+1}(I_h\NsM \hat e_h^{m+1})- \tilde{\mathscr D}^m (I_h\Nsm \hat e_h^m)}{\tau}
	\notag\\
	&\quad-
	\frac{\tilde{\mathscr D}^{m+1}(I_h\NsM \hat e_h^{m+1})- \tilde{\mathscr D}^{m}(I_h\NsM \hat e_h^{m+1})}{\tau}
	\notag\\
	&\quad
	+
	\tilde{\mathscr D}^m\Big(\frac{I_h(\Nsm-\NsM)\hat e_h^{m+1}}{\tau} \Big)
	\notag\\
	&\quad
	+
	\tilde{\mathscr D}^m\Big(I_h\Nsm\Big(\frac{e_h^{m+1}-\hat e_h^{m+1}}{\tau} - I_h\Tsm v^m\Big) \Big)
	\notag\\
	&=:
	\frac{\tilde{\mathscr D}^{m+1}( \hat e_h^{m+1})- \tilde{\mathscr D}^{m} ( \hat e_h^m)}{\tau}
	+ D_3^m + D_4^m + D_5^m .\notag
\end{align}
The time-derivative estimates in Lemma~\ref{bhk1:main} give
\begin{align}
	| D_3^m |
	&\lesssim
	h^2 \| \ehM \|_{H^1(\Ghsm)}
	.
	\notag
\end{align}
H\"older's inequality and \eqref{eq:X-hat-disp-W1inf1} yield
\begin{align}
	| D_4^m |
	&\lesssim
	\|  \tilde{\mathscr D}^m \|_{H^{-1}(\Ghsm)} \| \ehM \|_{H^1(\Ghsm)}
	\lesssim
	(\tau+h^2) \| \ehM \|_{H^1(\Ghsm)} ,
	\notag
\end{align}
and Lemma~\ref{lemma:corrected-displacement-nodal} below gives
\begin{align}
	| D_5^m |
	&\lesssim
	\|  \tilde{\mathscr D}^m \|_{H^{-1}(\Ghsm)}
		\Big\| I_h\Nsm\Big(\frac{e_h^{m+1}-\hat e_h^{m+1}}{\tau} - I_h\Tsm v^m\Big) \Big\|_{H^1(\Ghsm)}
	\lesssim
	\tau (\tau+h^2) .
	\notag
\end{align}
Here $\|  \tilde{\mathscr D}^m \|_{H^{-1}(\Ghsm)}$ denotes the dual norm of the consistency functional acting on the finite dimensional space of test functions.
Combining these decompositions, we obtain
\begin{align}
	{\mathscr D}^m(I_h\Nbhsm\dtem )
	&=
	\frac{\tilde{\mathscr D}^{m+1}( \hat e_h^{m+1})- \tilde{\mathscr D}^{m} ( \hat e_h^m)}{\tau}
	+
	\sum_{i=1}^5 D_i^m ,
	\notag
\end{align}
with
\begin{align}\label{eq:Dim}
	\Big| \sum_{i=1}^5 D_i^m \Big|
	&\lesssim
	\epsilon^{-1}(\tau+h^2)^2
	+
	\epsilon
	\| \dtem \cdot \nbhsm \|_{L_h^2(\Ghsm)}^2
	\notag\\
	&\quad
	+
	\| \eM \|_{H^1(\Ghsm)}^2
	+
	\| \eem \|_{H^1(\Ghsm)}^2
	+
	\| \ehm \|_{H^1(\Ghsm)}^2
	.
\end{align}

We conclude this subsection by proving the super-approximation result used above to estimate $A_3^m$ and $D_5^m$.
\begin{lemma}
	\label{lemma:corrected-displacement-nodal}
	The following super-approximation estimate holds:
	\begin{align}
		\|
		I_h\Nsm\bigl(\eM-\ehM-\tau I_h\Tsm v^m\bigr) \|_{H^1(\Ghsm)}
		\lesssim \tau^2
		. \notag
%		\label{eq:corrected-displacement-nodal}
	\end{align}
	We also have
	\begin{align}
		\|
		I_h\Nsm\bigl(e_h^m-\ehm-\tau I_h\Tsm v^m\bigr) \|_{H^1(\Ghsm)}
		\lesssim \tau^2
		. \notag
%		\label{eq:corrected-displacement-nodal1}
	\end{align}
	As a consequence,
	\begin{align}
		\|
		I_h\Nsm e_h^m-\ehm \|_{H^1(\Ghsm)}
		\lesssim \tau^2
		. \notag
%		\label{eq:corrected-displacement-nodal2}
	\end{align}
\end{lemma}

\begin{proof}
	Let $\tilde Y(t;p)$ be the local flow generated by the
	time-dependent exact velocity $v$ determined by the system \eqref{subeq:system}:
	\begin{align}
		\partial_t\tilde Y(t;p)
		&=v\bigl(t,\tilde Y(t;p)\bigr),
		\qquad t\in[t_m,t_{m+1}],
		\notag\\
		\tilde Y(t_m;p)&=p,
		\qquad p\in\Gm.
		\notag
%		\label{eq:corrected-displacement-flow}
	\end{align}
	Since $v\cdot n=-H$, the flow remains on the evolving curve, so $\tilde Y(t;p)\in\Gamma(t)$.
	Define $\tilde X_{h,*}^{m+1}\in[S_h(\Ghsm)]^2$ by its nodal values
	\begin{align}
		\tilde X_{h,*}^{m+1}(p)
		=\tilde Y(t_{m+1};p),
		\qquad p\in\mathcal N(\Ghsm). \notag
%		\label{eq:corrected-displacement-flow-interpolant}
	\end{align}
	Thus the nodes of this interpolated curve lie on $\GM$.
	
	Taylor expansion of this flow gives
	\begin{align}
		\tilde X_{h,*}^{m+1}-\hat X_{h,*}^{m}
		&=\tau I_h\bigl(-H^mn^m+\Tsm v^m + I_h \tilde g^m \bigr) , 
%		\notag
		\label{eq:corrected-displacement-flow-expansions}
	\end{align}
	with
	\begin{align}
		\| I_h \tilde g^m \|_{W^{1,\infty}(\Ghsm)} \lesssim\tau. \notag
	\end{align}
	Combining
	\eqref{eq:X-id1}--\eqref{W1infty-g} and \eqref{eq:corrected-displacement-flow-expansions} yields
	\begin{align}
		\tilde X_{h,*}^{m+1}
		=X_{h,*}^{m+1}+\tau I_h\Tsm v^m+\tau I_h(-g^m+\tilde g^m) , \notag
%		\label{eq:corrected-displacement-flow-remainder}
	\end{align}
	and
	\begin{align}
		\eM-\ehM-\tau I_h\Tsm v^m
		&=\hat X_{h,*}^{m+1}-X_{h,*}^{m+1}
		-\tau I_h\Tsm v^m
		\notag\\
		&=\hat X_{h,*}^{m+1}-\tilde X_{h,*}^{m+1}+\tau I_h(-g^m+\tilde g^m).
		\label{eq:corrected-displacement-chord-identity}
	\end{align}
	Since $\hat X_{h,*}^{m+1}- \tilde X_{h,*}^{m+1}$ is an almost tangential chord, Taylor's theorem yields the nodal estimate
	\begin{align}
		| I_h \NsM (\hat X_{h,*}^{m+1}-\tilde X_{h,*}^{m+1})|
		\lesssim| I_h \TsM (\hat X_{h,*}^{m+1}-\tilde X_{h,*}^{m+1})|^2.
		\label{eq:corrected-displacement-chord-bound}
	\end{align}
	
	On $\Ghsm$, we have the identity
	\begin{align}
		I_h \TsM (\hat X_{h,*}^{m+1}-\tilde X_{h,*}^{m+1})
		&=
		I_h \TsM\big(\eM-\ehM-\tau I_h\Tsm v^m
		- \tau I_h (-g^m+\tilde g^m)\big)
		\notag\\
		&=
		I_h \Tsm(\eM-\ehm-\tau I_h\Tsm v^m)
		\notag\\
		&\quad+
		I_h (\TsM-\Tsm)(\eM-\tau I_h\Tsm v^m)
		\notag\\
		&\quad
		- \tau I_h \TsM (-g^m+\tilde g^m)
		\notag .
	\end{align}
	Consequently,
	\begin{align}
		| I_h \TsM (\hat X_{h,*}^{m+1}-\tilde X_{h,*}^{m+1})|
		&\lesssim
		\tau | I_h \Tsm\dtem|
		+
		\tau| \eM|
		+\tau^2
		\notag ,
	\end{align}
	at all finite element nodes.
	
	Identity \eqref{eq:corrected-displacement-chord-identity} and orthogonality also give, on $\Ghsm$,
	\begin{align}
		&
		I_h \Nsm\bigl(\eM-\ehM-\tau I_h\Tsm v^m\bigr)
		\notag\\
		&=
		I_h \Nsm (\hat X_{h,*}^{m+1}-\tilde X_{h,*}^{m+1})
		+
		\tau
		I_h \Nsm
		(-g^m+\tilde g^m)
		\notag\\
		&=
		I_h \NsM (\hat X_{h,*}^{m+1}-\tilde X_{h,*}^{m+1})
		\notag\\
		&\quad
		-
		I_h (\NsM-\Nsm) \NsM (\hat X_{h,*}^{m+1}-\tilde X_{h,*}^{m+1})
		\notag\\
		&\quad
		-
		I_h (\NsM-\Nsm) \TsM (\hat X_{h,*}^{m+1}-\tilde X_{h,*}^{m+1})
		\notag\\
		&\quad
		+
		\tau
		I_h \Nsm
		(-g^m+\tilde g^m)
		\notag .
	\end{align}
	Applying \eqref{eq:corrected-displacement-chord-bound} to this identity yields the nodal bound
	\begin{align}
		&
		\Big|I_h\Nsm\bigl(\eM-\ehM-\tau I_h\Tsm v^m\bigr) -\tau I_h \Nsm (-g^m+\tilde g^m)
		\notag\\
		&\quad
		+
		I_h (\NsM-\Nsm) \TsM (\hat X_{h,*}^{m+1}-\tilde X_{h,*}^{m+1})
		\Big|
		\notag\\
		&\lesssim
		(1+\tau)
		| I_h \NsM (\hat X_{h,*}^{m+1}-\tilde X_{h,*}^{m+1})|
		\notag\\
		&\lesssim
		| I_h \TsM (\hat X_{h,*}^{m+1}-\tilde X_{h,*}^{m+1})|^2
		\notag\\
		&\lesssim
		\tau^2 | I_h \Tsm\dtem|^2 + \tau^2 | \eM|^2 + \tau^4 . \notag
	\end{align}
	The triangle inequality, the inverse inequality, and \eqref{eq:corrected-displacement-chord-identity} then give
	\begin{align}
		&
		\| I_h \Nsm\bigl(\eM-\ehM- \tau I_h\Tsm v^m\bigr) \|_{H^1(\Ghsm)}
		\notag\\
		&\lesssim
		\| I_h (\NsM-\Nsm) \TsM (\hat X_{h,*}^{m+1}-\tilde X_{h,*}^{m+1}) \|_{H^1(\Ghsm)}
		+
		\tau
		\| I_h \Nsm  (-g^m+\tilde g^m) \|_{H^1(\Ghsm)}
		\notag\\
		&\quad
		+ h^{-1}\tau^2 \| I_h \Tsm\dtem \|_{L^4(\Ghsm)}^2 + h^{-1} \tau^2 \| \eM \|_{L^4(\Ghsm)}^2
		+
		h^{-1}\tau^4
		\notag\\
		&\lesssim
		\tau^2
		+\tau^2 \| I_h \Tsm\dtem \|_{H^1(\Ghsm)}
		+ h^{-1}\tau^2 \| I_h \Tsm\dtem \|_{L^4(\Ghsm)}^2 + h^{-1} \tau^2 \| \eM \|_{L^4(\Ghsm)}^2
		\notag\\
		&\lesssim
		\tau^2 , \notag
	\end{align}
	where the final inequality uses \eqref{eq:a-priori-delta-e} and \eqref{hat-ehm-L2}--\eqref{hat-ehm-H1}.
	Finally, \eqref{eq:X-hat-disp-W1inf1} yields
	\begin{align}
		&\|
		I_h\Nsm\bigl(e_h^m-\ehm-\tau I_h\Tsm v^m\bigr) \|_{H^1(\Ghsm)}
		\notag\\
		&\leq
		\|
		I_h(\Nsm-N_*^{m-1})\bigl(e_h^m-\ehm-\tau I_h\Tsm v^m\bigr) \|_{H^1(\Ghsm)}
		\notag\\
		&\quad
		+
		\|
		I_hN_*^{m-1}\bigl(e_h^m-\ehm-\tau I_h\Tsm v^m\bigr) \|_{H^1(\Ghsm)}
		\notag\\
		&\lesssim \tau^2
		, \notag
	\end{align}
	for $m\geq 1$. When $m=0$, we have $I_h\Nsm(e_h^m-\ehm-\tau I_h\Tsm v^m)=0$.
	This completes the proof.
\end{proof}

\subsection{Main error estimates}\label{sec:main-err}

Substituting the decomposition of $A_i^m$ and $D_i^m$ into \eqref{eq:err2}, we obtain
\begin{align}\label{error-estimate-1}
	&
	\frac{\| \nabla_\GhsM I_h\NsM e_h^{m+1} \|_{L^2(\GhsM)}^2 - \| \nabla_\Ghsm I_h\Nsm e_h^{m} \|_{L^2(\Ghsm)}^2}{2\tau}
	+
	\| \delta_\tau \ehm \cdot \nbhsm \|_{L_h^2(\hat\Gamma_{h,*}^{m})}^2
	\notag\\
	&\leq
	-
	\frac{\tilde{\mathscr D}^{m+1}( \hat e_h^{m+1})- \tilde{\mathscr D}^{m} ( \hat e_h^m)}{\tau}
	- \int_{\hat\Gamma_{h, *}^{m}}^h I_h\Tsm v^m \cdot \bar n_{h, *}^m \,   \dtem \cdot \bar n_{h, *}^m
	\notag\\
	&\quad
	- \sum_{i=1}^5 D_i^m
	- \mathscr J^m( I_h\Nbhsm \dtem) - \mathscr B^m(\ehm,  I_h\Nbhsm \dtem) - \mathscr K^m( I_h\Nbhsm \dtem)  \notag\\
	&\quad
	-
	\sum_{i=1}^8 A_i^m
	- \sum_{i=1}^3 \mathscr F_i^m( I_h\Nbhsm \dtem)
	\notag\\
	&\leq
	-
	\frac{\tilde{\mathscr D}^{m+1}( \hat e_h^{m+1})- \tilde{\mathscr D}^{m} ( \hat e_h^m)}{\tau}
	+
	\epsilon^{-1} C_\kl (\tau + h^{2})^2
	+ \epsilon C_\kl \| \dtem\cdot\nbhsm \|_{L_h^2(\Ghsm)}^2 \notag\\
	&\quad+ \epsilon^{-1} C_\kl \big(\| e_h^m \|_{H^1(\Ghsm)}^2 + \| \ehm \|_{H^1(\Ghsm)}^2 + \| \eM \|_{H^1(\Ghsm)}^2 + \| \ehM \|_{H^1(\Ghsm)}^2 \big) ,
\end{align}
%where the final inequality combines the preceding estimates for the remaining terms.
where the final inequality follows from \eqref{eq:mass-ortho},
\eqref{eq:J_est}, \eqref{eq:K_est}, and
\eqref{Estimate-F1m}--\eqref{Estimate-F3m}, together with the
bounds for $\sum_{i=1}^8 A_i^m$ and $\sum_{i=1}^5 D_i^m$
established in \eqref{eq:Aim} and \eqref{eq:Dim}.
%We have also used the velocity estimates
%\eqref{eq:vel-est-H1}--\eqref{eq:a-priori-delta-e},
%the induction hypothesis \eqref{eq:ind_hypo1},
%the step-size condition $\tau\leq c_0h^2$,
%norm equivalence, the inverse inequality, and Young's inequality.
The term $\mathscr B^m(\ehm,I_h\Nbhsm\dtem)$ vanishes on curves,
as noted in Section~\ref{sec:lin-bilin-est}.

Choosing the Young parameter sufficiently small allows us to absorb the term $\epsilon\| \dtem\cdot\nbhsm \|_{L_h^2(\Ghsm)}^2$ into the left-hand side. We then apply \eqref{hat-ehm-H1} to bound $\eem$ and $\eM$ in terms of $\hat e_h^{m-1}$ and $\ehm$ on the right-hand side.
Multiplying \eqref{error-estimate-1} by $\tau$ and summing over the time steps yields
\begin{align}
	& \frac{1}{2} \| \nabla_\GhsM I_h\NsM e_h^{m+1} \|_{L^2(\GhsM)}^2
	+
	\frac{1}{2}
	\sum_{j=0}^m \tau
	\| \delta_\tau \hat e_{h}^j \cdot \bar n_{h,*}^j \|_{L_h^2(\hat\Gamma_{h,*}^{j})}^2
	\notag\\
	&\leq
	-
	\tilde{\mathscr D}^{m+1}(I_h\NsM \hat e_h^{m+1})
	+
	\epsilon^{-1} C_\kl (\tau+ h^{2})^2
	+
	\epsilon^{-1} C_\kl \sum_{j=0}^{m+1} \tau \| \hat e_{h}^j \|_{H^1(\hat\Gamma_{h,*}^j)}^2 , \notag
\end{align}
where we have used $\hat e_h^{0} = 0$.
The endpoint consistency term satisfies, by \eqref{eq:tildeDm},
\begin{align}
	| \tilde{\mathscr D}^{m+1}( \hat e_h^{m+1}) |
	\leq C_\kl (\tau+h^2) \| \hat e_h^{m+1} \|_{H^1(\Ghsm)} . \notag
\end{align}
Lemma~\ref{lemma:corrected-displacement-nodal} allows us to replace $I_h\NsM e_h^{m+1}$ on the left-hand side by $\hat e_h^{m+1}$ up to a higher-order smaller remainder. Together with the normal Poincar\'e inequality (Lemma \ref{lemma:poincare}), this yields the full $H^1$ estimate
\begin{align}
	& \| \hat e_h^{m+1} \|_{H^1(\GhsM)}^2
	+
	\sum_{j=0}^m \tau
	\| \delta_\tau \hat e_h^j \cdot \bar n_{h,*}^j \|_{L_h^2(\hat\Gamma_{h,*}^{j})}^2
	\notag\\
	&\leq
	\epsilon^{-1} C_\kl (\tau+ h^{2})^2
	+
	\epsilon^{-1} C_\kl \sum_{j=0}^{m+1} \tau \|  \hat e_h^j \|_{H^1(\hat\Gamma_{h,*}^j)}^2
	+
	\epsilon C_\kl \| \ehM \|_{H^1(\Ghsm)}^2
	, \notag
\end{align}
for all $0\leq m\leq l$.
Applying the discrete Gr\"onwall inequality, we obtain
\begin{align}\label{eq:err_fin0}
	\max_{0\le m\le l+1} \| \hat e_h^{m} \|_{H^1(\Ghsm)}^2
	+
	\sum\limits_{m = 0}^{l} \tau
	\| \delta_\tau \ehm \cdot \nbhsm \|_{L_h^2(\hat\Gamma_{h,*}^{m})}^2
	\le C_\kl ( \tau +  h^{2})^2 .
\end{align}
A further application of \eqref{hat-ehm-H1}, the norm equivalence (Lemma \ref{lemma:lump}), \eqref{eq:tan_stab_e1} and \eqref{eq:err_fin0} gives
\begin{align}\label{eq:err_fin1}
	\max_{0\le m\le l+1} \| e_h^{m} \|_{H^1(\Ghsm)}^2
	+
	\sum\limits_{m = 0}^{l} \tau
	\| \delta_\tau \ehm \|_{L^2(\hat\Gamma_{h,*}^{m})}^2
	\le C_\kl ( \tau +  h^{2})^2 .
\end{align}
The uniform boundedness of $C_\kl$ is established in Section~\ref{sec:bbd}.

\subsection{Uniform boundedness of $\kappa_l$}
\label{sec:bbd}

Throughout this subsection, we view $\hat X_{h,*}^m$ and $X_h^m$ as maps from the reference curve $\Gamma_{h,\rm f}^0$ to $\Ghsm$ and $\Gamma_h^m$, respectively. We write $v^m_{\rm f} = v^m \circ a^m \circ \hat X_{h,*}^m$ and $g^m_{\rm f} = g^m \circ a^m \circ \hat X_{h,*}^m$ for the corresponding functions on $\Gamma_{h,\rm f}^0$.
The geometric relations \eqref{eq:geo_rel_4}--\eqref{eq:geo_rel_6} give
\begin{align*}
	&\|\hat X_{h,*}^{m + 1} -\hat X_{h,*}^{m}\|_{W^{1,\infty}(\Gamma_{h,\rm f}^0)} \notag\\
	&\le \| I_h [(Y^{m + 1} - {\rm id})\circ a^m \circ \hat X_{h,*}^{m} ]\|_{W^{1,\infty}(\Gamma_{h,\rm f}^0)}
	+ \| \rho_h^m\circ \hat X_{h,*}^{m} \|_{W^{1,\infty}(\Gamma_{h,\rm f}^0)} \\
	&\quad+\| I_h[ I_h(T_*^m\circ \hat X_{h,*}^{m}) I_h (N_*^{m+1}\circ \hat X_{h,*}^{m+1}-N_*^{m}\circ \hat X_{h,*}^{m})  (\hat e_{h}^{m + 1} \circ \hat X_{h,*}^{m})]\|_{W^{1,\infty}(\Gamma_{h,\rm f}^0)} \\
	&\quad+\| I_h[(T_*^m \circ \hat X_{h,*}^{m} ) (X_{h}^{m + 1} - X_{h}^{m})  ] \|_{W^{1,\infty}(\Gamma_{h,\rm f}^0)} \\
	&=: E_1^m + E_2^m + E_3^m.
\end{align*}
The stability of $I_h$ on $C^0(\Gamma_{h,\rm f}^0)\cap W^{1,\infty}(\Gamma_{h,\rm f}^0)$, the chain rule, the inverse inequality, and \eqref{eq:geo_rel_5} imply
\begin{align*}
	E_1^m
	&\le C_0 \| (Y^{m + 1} - {\rm id})\circ a^m \circ \hat X_{h,*}^{m} \|_{W^{1,\infty}(\Gamma_{h,\rm f}^0)}
	+
	\| \rho_h^m\circ \hat X_{h,*}^{m} \|_{W^{1,\infty}(\Gamma_{h,\rm f}^0)}
	\\
	&\le C_0 \| \nabla_{\hat\Gamma_{h,*}^m} [(Y^{m + 1} - {\rm id})\circ a^m] \circ \hat X_{h,*}^{m}
	\,\nabla_{\Gamma_{h,\rm f}^0} \hat X_{h,*}^{m} \|_{L^{\infty}_h(\Gamma_{h,\rm f}^0)} \\
	&\quad
	+ C_0\| Y^{m + 1} - {\rm id} \|_{L^{\infty}(\Gm)}
	+ C_0 h^{-1} \| \rho_h^m\circ \hat X_{h,*}^{m} \|_{L^{\infty}(\Gamma_{h,\rm f}^0)}
	\\
	&\le C_0\tau\|\hat X_{h,*}^{m} \|_{W^{1,\infty}(\Gamma_{h,\rm f}^0)}
	+ C_0\tau + C_0 h^{-1} \big(\tau^2 + \|I_h \Tsm (\hat X_{h,*}^{m+1} - \hat X_{h,*}^{m}) \|_{L^\infty(\Ghsm)}^2\big) \notag\\
	&\le C_0\tau\|\hat X_{h,*}^{m} \|_{W^{1,\infty}(\Gamma_{h,\rm f}^0)}
	+ C_0\tau ,
\end{align*}
where the final inequality uses \eqref{eq:X-hat-disp-W1inf} and the step-size condition $\tau\leq c_0 h^2$ for sufficiently small $h\leq h_\kl$.
The inverse inequality also gives
\begin{align*}
	E_2^m
	&\le C_0h^{-3/2} \| I_h[ I_h(T_*^m\circ \hat X_{h,*}^{m}) I_h (N_*^{m+1}\circ \hat X_{h,*}^{m+1}-N_*^{m}\circ \hat X_{h,*}^{m})  (\hat e_{h}^{m + 1} \circ \hat X_{h,*}^{m})]\|_{L^{2}_h(\Gamma_{h,\rm f}^0)} \\
	&\le C_{\kappa_l} h^{-3/2} \tau (\tau+h^{2}) 
	\leq C_0 \tau
	,
\end{align*}
where we have used \eqref{eq:X-hat-disp-W1inf1}, \eqref{eq:err_fin0} and the step-size condition $\tau\le c h^{2}$ for sufficiently small $h\leq h_\kl$.
%Combining the bounds for $E_1^m$ and $E_2^m$, we obtain
%\begin{align*}
%	\|\hat X_{h,*}^{m + 1} -\hat X_{h,*}^{m}\|_{W^{1,\infty}(\Gamma_{h,\rm f}^0)}
%	&\le
%	C_0\tau (1+\| \hat X_{h,*}^{m} \|_{W^{1,\infty}(\Gamma_{h,\rm f}^0)})
%	\notag\\
%	&\quad
%	+ C_0\| I_h[(T_*^m \circ \hat X_{h,*}^{m} ) (X_{h}^{m + 1} - X_{h}^{m})  ] \|_{W^{1,\infty}(\Gamma_{h,\rm f}^0)}  .
%\end{align*}
To estimate $E_3^m$, we use \eqref{eq:geo_rel_31} and write
\begin{align*}
	E_3^m
%	=\| I_h[(T_*^m \circ \hat X_{h,*}^{m} ) (X_{h}^{m + 1} - X_{h}^{m})  ]\|_{W^{1,\infty}(\Gamma_{h,\rm f}^0)} \notag\\
%	&
	&\leq \| I_h[(T_*^m \circ \hat X_{h,*}^{m} ) (X_{h}^{m+1} - X_{h}^{m} - \tau I_h v^m_{\rm f}) ] \|_{W^{1,\infty}(\Gamma_{h,\rm f}^0)}
	+ \tau\| I_h[(T_*^m \circ \hat X_{h,*}^{m} ) v^m_{\rm f}] \|_{W^{1,\infty}(\Gamma_{h,\rm f}^0)} \notag\\
	&=  \tau \| I_h[(T_*^m \circ \hat X_{h,*}^{m} ) (\dtem + I_h g^m_{\rm f} )]  \|_{W^{1,\infty}(\Gamma_{h,\rm f}^0)}
%	\notag\\
%	&\quad
	+ \tau\| I_h[(T_*^m \circ \hat X_{h,*}^{m} ) v^m_{\rm f}] \|_{W^{1,\infty}(\Gamma_{h,\rm f}^0)} \\
	&\leq \tau \| I_h[(T_*^m \circ \hat X_{h,*}^{m} ) \dtem ]  \|_{W^{1,\infty}(\Gamma_{h,\rm f}^0)}
%	\notag\\
%	&\quad
	+ C_0 \tau (C_{\kappa_l}\tau + 1 + \| \hat X_{h,*}^m \|_{W^{1,\infty}(\Ghso)})
	\notag\\
	&\leq C_{\kappa_l} h^{-1/2} \tau \Big( (\tau + h^{2}) + \| \nabla_\Ghsm \ehm \|_{L^2(\Ghsm)} 
		\notag\\
		&\quad
		+ (1+h^{-2}\|\nabla_\Ghsm\ehm\|_{L^2(\Ghsm)})\| I_h\Nbhsm \delta_\tau\ehm \|_{L^2(\Ghsm)}\Big)
%	\notag\\
%	&\quad
	+ C_0 \tau (1 + \| \hat X_{h,*}^m \|_{W^{1, \infty}(\Ghso)})
	\notag\\
	&\leq C_{\kappa_l} h^{-1/2} \tau \| I_h\Nbhsm \delta_\tau\ehm \|_{L^2(\Ghsm)}
%	\notag\\
%	&\quad
	+ C_0 \tau (1 + \| \hat X_{h,*}^m \|_{W^{1, \infty}(\Ghso)})  .
\end{align*}
Here, the penultimate inequality follows from the inverse inequality, \eqref{eq:tan_stab_e} and
\begin{align*}
	\| I_h g^m_{\rm f} \|_{W^{1,\infty}(\Gamma_{h,\rm f}^0)} &\le C_\kl \tau
	&&\mbox{(in view of \eqref{W1infty-g})} ,\\
	\| (\Tsm\circ \hat X_{h,*}^m) v^m_{\rm f} \|_{W^{1,\infty}(\Gamma_{h,\rm f}^0)} &\le
	C_0(1+ \|\hat X_{h,*}^m \|_{W^{1,\infty}(\Gamma_{h,\rm f}^0)} )
	&&\mbox{(chain rule of differentiation)} ,
\end{align*}
while the last inequality is derived by invoking \eqref{eq:err_fin0} and the step-size condition $\tau\le c h^{2}$ for sufficiently small $h\leq h_\kl$.

%Applying the inverse inequality under the conditions $\tau\le c h^{2}$ and $C_{\kappa_l}h^{\frac12}\le 1$, we obtain
%\begin{align*}
%	&\| I_h[(T_*^m \circ \hat X_{h,*}^{m} ) (X_{h}^{m + 1} - X_{h}^{m})  ]\|_{W^{1,\infty}(\Gamma_{h,\rm f}^0)} \notag\\
%	&\leq  C_0h^{-1/2}\| I_h[(T_*^m \circ \hat X_{h,*}^{m} ) (\eM - \ehm  - \tau I_h (\Tsm\circ\hat X_{h,*}^m) v^m_{\rm f} )]  \|_{H^{1}(\Gamma_{h,\rm f}^0)}
%	\notag\\
%	&\quad
%	+ C_0 \tau (1 + \| \hat X_{h,*}^m \|_{W^{1,\infty}(\Ghso)})  \notag\\
%	&\leq {\r C_{\kappa_l} h^{-1/2} \tau \| \delta_\tau\ehm \|_{L^2(\Ghsm)}}
%	+ C_0 \tau (1 + \| \hat X_{h,*}^m \|_{W^{1, \infty}(\Ghso)}) ,
%\end{align*}
%where the final inequality uses \eqref{eq:tan_stab_e} and \eqref{eq:err_fin1}.
%Recalling \eqref{eq:err_fin0}, the step-size condition $\tau\le c h^{2}$, and the bound $\| \nabla_\Ghsm \ehm \|_{L^2(\Ghsm)}\le C_{\kappa_l} h^{1.75}$, we record the resulting estimate as
%\begin{align}\label{eq:k_inv}
%	&\| I_h[(T_*^m \circ \hat X_{h,*}^{m} ) (X_{h}^{m + 1} - X_{h}^{m})  ]\|_{W^{1,\infty}(\Gamma_{h,\rm f}^0)} \notag\\
%	&\leq {\r C_{\kappa_l} h^{-1/2} \tau \| \delta_\tau\ehm \|_{L^2(\Ghsm)}}
%	+ C_0 \tau (1 + \| \hat X_{h,*}^m \|_{W^{1, \infty}(\Ghso)}) ,
%\end{align}
%where the mesh-size condition $C_{\kappa_l}h^{\frac12}\le 1$ is used again.
Combining the preceding bounds gives
\begin{align}
	\|\hat X_{h,*}^{m + 1} -\hat X_{h,*}^{m}\|_{W^{1,\infty}(\Gamma_{h,\rm f}^0)}
	\le C_{\kappa_l} h^{-1/2} \tau \| I_h \Nbhsm \delta_\tau\ehm \|_{L^2(\Ghsm)}
	+ C_0 \tau (1 + \| \hat X_{h,*}^m \|_{W^{1, \infty}(\Gamma_{h,\rm f}^0)}) . \notag
\end{align}
The triangle inequality and the super-convergent error estimate \eqref{eq:err_fin1} therefore yield
\begin{align}
	&\|  \hat X_{h,*}^{m+1} \|_{W^{1, \infty}(\Ghso)} - \|  \hat X_{h,*}^0 \|_{W^{1, \infty}(\Ghso)} \notag\\
	&\le \sum_{j=0}^m\|\hat X_{h,*}^{j + 1} -\hat X_{h,*}^{j}\|_{W^{1,\infty}(\Gamma_{h,\rm f}^0)} \notag\\
	&\le
	C_{\kappa_l} h^{-1/2} \sum_{j=0}^m \tau \| I_h \bar N_{h,*}^j \delta_\tau\hat e_h^j \|_{L^2(\hat\Gamma_{h,*}^j)}
	+ C_0 + \sum_{j=0}^m C_0 \tau \| \hat X_{h,*}^j \|_{W^{1,\infty}(\Gamma_{h,\rm f}^0)} \notag\\
	&\le
	C_{\kappa_l} h^{-1/2} (\tau+h^{2})
	+ C_0 + \sum_{j=0}^m C_0 \tau \| \hat X_{h,*}^j \|_{W^{1,\infty}(\Gamma_{h,\rm f}^0)}
	\notag\\
	&\le
	C_0 + \sum_{j=0}^m C_0 \tau \| \hat X_{h,*}^j \|_{W^{1,\infty}(\Gamma_{h,\rm f}^0)}
	, \notag
\end{align}
given the step-size condition $\tau\le c h^{2}$ and sufficiently small $h\leq h_\kl$.
%where \eqref{eq:tan_stab_e1} and \eqref{eq:err_fin1} give
%\begin{align}
%	C_{\kappa_l} h^{-1/2}
%	\sum_{j=0}^m \tau \| \delta_\tau\hat e_h^j \|_{L^2(\hat\Gamma_{h,*}^j)}
%	&\leq
%	C_{\kappa_l} h^{-1/2}
%	\sum_{j=0}^m \tau (\tau+h^2+\| \delta_\tau\hat e_h^j \cdot \bar n_{h,*}^j \|_{L_h^2(\hat\Gamma_{h,*}^j)})
%	\notag\\
%	&\leq
%	C_{\kappa_l} h^{-1/2} (\tau+h^2)
%	\leq C_0 .
%	\notag
%\end{align}
The discrete Gr\"onwall inequality now yields
\begin{align}
	\max_{0\le m\le l} \|  \hat X_{h,*}^{m+1} \|_{W^{1, \infty}(\Ghso)}
	&\le
	C_0 .  \notag
\end{align}
The estimate $\| (\hat X_{h,*}^{m+1})^{-1} \|_{W^{1,\infty}(\hat\Gamma_{h,*}^{m+1})}\le C_0$ follows by the argument used for \cite[Eq. (4.43)]{Bai2026}; we omit the details. Consequently,
\begin{align}\label{eq:bbd-kl}
	\kappa_{l+1} \le C_0 ,
\end{align}
with a constant $C_0$ independent of $\tau$ and $l$. Thus the quantity $\kappa_l$ defined in \eqref{PP} is uniformly bounded. The restriction on the mesh size $h$ can therefore be improved to $h\leq h_0$ for some $h_0>0$ independent of $\tau$ and $l$.
Estimates \eqref{eq:err_fin0}--\eqref{eq:err_fin1} and \eqref{eq:bbd-kl} recover the induction hypothesis \eqref{eq:ind_hypo1} at the next time level $m=l+1$.
This completes the proof of Theorem~\ref{thm:main}.

\subsection{Trajectory convergence}
\label{Proof-THM-2}

Recall the notation from Section \ref{subsec:main-result-induction}.
$\{x_{j,\#}(t):t\in[0,T]\}$ denotes the particle trajectory generated by the velocity in \eqref{subeq:system}, with initial position $x_j^0\in\Gamma_h^0$. $X_{h,\#}^m$ is the finite element function with nodal vector $(x_{1,\#}(t_m),\dots,x_{J,\#}(t_m))^\top$. Then $X_{h,\#}^m$ maps the initial curve $\Gamma_h^0$ onto the polygonal curve $\Gamma_{h,\#}^m$, which interpolates $\Gamma^m$ at the nodes $x_{j,\#}(t_m)$, $j=1,\dots,J$.

%Let $I_h$ denote the nodal interpolation operator on the initial curve $\Gamma_h^0$. 
At the nodes of $\Ghso$, we have
\begin{align*}
	X_{h,\#}^{m+1} = X_{h,\#}^{m} + \tau I_h [v^m\circ X_{h,\#}^{m}] + O(\tau^2) ,
\end{align*}
which is the nodal Taylor expansion of the flow determined by \eqref{subeq:system}. Using \eqref{eq:geo_rel_31},
the trajectory error $e_{h,\#}^m= X_{h,\#}^m - X_h^m$ therefore satisfies
\begin{align*}
	e_{h,\#}^{m+1}
	&= e_{h,\#}^{m} - (X_h^{m+1} - X_h^m - \tau I_h [v^m\circ \hat X_{h,*}^{m}])
	+ \tau I_h[v^m\circ X_{h,\#}^{m} -  v^m\circ \hat X_{h,*}^{m}] + O(\tau^2) \\
	&= e_{h,\#}^{m} - (\eM - \ehm  - \tau I_h[ (\Tsm v^m) \circ \hat X_{h,*}^{m}] + \tau I_h g_{\rm f}^m) \\
	&\quad\,
	+ \tau I_h[v^m\circ X_{h,\#}^{m} -  v^m\circ \hat X_{h,*}^{m}]  + O(\tau^2).
\end{align*}
The smoothness of $v^m$ on $\Gamma^m$ gives
$| I_h (v^m\circ X_{h,\#}^{m} -  v^m\circ \hat X_{h,*}^{m} ) |\le C( |e_{h,\#}^{m}| +|\hat e_h^m|)$; also see Lemma \ref{lemma:super_conv-nonlinear}. Hence, at the nodes of $\Ghso$, we obtain
\begin{align*}
	|e_{h,\#}^{m+1} |
	&\le (1+C\tau)|e_{h,\#}^{m}| + |\eM - \ehm  - \tau I_h[ (\Tsm v^m) \circ \hat X_{h,*}^{m}]| + C\tau^2 + C\tau|\hat e_h^m| .
\end{align*}
Taking the discrete $L^2$ norm on $\Ghso$ and using its equivalence with the continuous $L^2$ norm on $\Gamma_h^0$ (cf. Lemma \ref{lemma:lump}) yields
\begin{align*}
	\|e_{h,\#}^{m+1} \|_{L^2_h(\Ghso)}
	&\le (1+C\tau)\|e_{h,\#}^{m}\|_{L^2_h(\Ghso)}
	+ C\|\eM - \ehm  - \tau I_h[ (\Tsm v^m) \circ \hat X_{h,*}^{m}]\|_{L_h^2(\Ghso)} \notag\\
	&\quad + C\tau^2 + C\tau \|\hat e_h^m \|_{L^2_h(\Ghso)} \\
	&\le
	C \tau (\tau+h^2)
	+
	(1+C\tau)\|e_{h,\#}^{m}\|_{L^2_h(\Ghso)}
	+ C\tau \| \delta_\tau\ehm \|_{L^2(\Ghsm)}
	,
\end{align*}
where $\|\hat e_h^m \|_{L^2_h(\Ghso)}$ is controlled by \eqref{eq:err_fin0} and the uniform boundedness of the shape regularity constants.
Iterating this inequality in $m$, or equivalently applying the discrete Gr\"onwall inequality, gives
\begin{align*}
	\max_{0\leq m\leq l+1}
	\|e_{h,\#}^{m} \|_{L^2(\Ghso)}
	&\le
	C (\tau+h^2)
	+
	C \sum_{m=0}^l \tau \| \delta_\tau\ehm \|_{L^2(\Ghsm)}
	\le C(\tau+h^2) .
\end{align*}
This completes the proof of Theorem~\ref{thm:trajectory} in view of the norm equivalence between $\Ghso$ and $\Ghsm$.

\appendix
\section{Notation}
\label{section:notation}
The notation below is frequently used in this article.
\begin{longtable}{p{1.1cm}p{13cm}}
	$\Gamma^m$:
	& 
	The exact smooth curve at time level $t=t_m$.\\
	
	$\Gamma_h^m$:
	&
	The numerically computed curve at time level $t=t_m$.\\
	
	$\bfx^{m}$: 
	&
	The nodal vector $\bfx^m=(x_1^m,\dots,x_J^m)^\top$ consisting of the positions of nodes on $\Gamma_h^m$.\\
	
	$\hat\bfx_*^{m}$: 
	&
	The distance projection of $\bfx^{m}$ onto the exact curve $\Gamma^m$, i.e., $\hat\bfx_*^{m}=(\hat x_{1,*}^m,\dots,\hat x_{J,*}^m)^\top$ with $\hat x_{j,*}^m=a^m(x_j^m)$. \\
	
	$\bfx_*^{m+1}$: 
	&
	The new position of $\hat\bfx_*^{m}$ evolving under curve-shortening flow (without additional tangential motion) from $t_m$ to $t_{m+1}$. \\
	
	$\Ghsm$: 
	&
	The piecewise polynomial curve which interpolates $\Gamma^m$ at the nodes in $\hat\bfx_*^{m}$.\\
	
	$\Gamma_{h,*}^{m+1}$: 
	&
	The piecewise polynomial curve which interpolates $\Gamma^{m+1}$ at the nodes in $\bfx_*^{m+1}$.\\
	
	$X_{h}^{m}$: 
	&
	The finite element function with nodal vector $\bfx^m$. It coincides with the identity map, i.e., ${\rm id}(x)=x$, when it is considered as a function on $\Gamma_{h}^m$. \\
	
	$X_{h}^{m+1}$: 
	&
	The finite element function with nodal vector $\bfx^{m+1}$. 
	When it is considered as a function on $\Gamma_{h}^m$, it represents the local flow map from $\Gamma_{h}^m$ to $\Gamma_{h}^{m+1}$.\\
	
	$\hat X_{h,*}^{m}$: 
	&
	The finite element function with nodal vector $\hat\bfx_*^m$. It coincides with the identity map, i.e., ${\rm id}(x)=x$, when it is considered as a function on $\Ghsm$. It coincides with the discrete flow map from $\hat\Gamma_{h,*}^0$ to $\Ghsm$ when it is considered as a function on $\hat\Gamma_{h,*}^0$.\\
	
	$X_{h,*}^{m+1}$: 
	&
	The finite element function with nodal vector $\bfx_*^{m+1}$. 
	When it is considered as a function on $\Ghsm$, it represents the local flow map from $\Ghsm$ to $\Gamma_{h,*}^{m+1}$.\\
	
	$X^{m+1}$: 
	&
	The global flow map from $\Gamma^0$ to $\Gamma^{m+1}$ under curve-shortening flow. \\
	
	$Y^{m+1}$: 
	&
	The local flow map from $\Gamma^m$ to $\Gamma^{m+1}$ under curve-shortening flow. \\
	
	$\hat e_{h}^m$: 
	&
	The finite element error function with nodal vector $\hat\bfe^m=\bfx^m-\hat\bfx_*^m$.\\ 
	
	$e_{h}^{m+1}$: 
	&
	The auxiliary error function with nodal vector $\bfe^{m+1}=\bfx^{m+1}-\bfx_*^{m+1}$.\\ 
	
	$n^m$: 
	&
	The unit normal vector on $\Gamma^m$. \\
	
	$n^m_*$: 
	&
	The unit normal vector of $\Gamma^m$ inversely lifted to a neighborhood of $\Gamma^m$ (including $\Ghsm$), i.e., $n^m_*=n^m\circ a^m$. \\
	
	$\hat n_{h,*}^m$: 
	&
	The normal vector on $\Ghsm$. \\
	
	$\bar n_{h,*}^m$: 
	&
	The averaged normal vector on $\Ghsm$, which is not necessarily unit. \\
	
	$n_h^m$: 
	&
	The normal vector on $\Gamma_{h}^m$. \\
	
	$\bar n_h^m$: 
	&
	The averaged normal vector on $\Gamma_{h}^m$, which is not necessarily unit. \\
	
	$\hat \mu_{h,*}^m$: 
	&
	The co-normal vector (unit tangent vector) on $\Ghsm$. \\
	
	$\mu_{h}^m$: 
	&
	The co-normal vector (unit tangent vector) on $\Ghm$. \\
	
	$\Nsm$: 
	&
	The normal projection operator $\Nsm=n^m_* (n^m_*)^\top$ on $\Ghsm$. \\
	
	$N^m$: 
	&
	The normal projection operator $N^m=n^m (n^m)^\top$ on $\Gamma^m$. 
	Thus $N^m$ is the lift of $\Nsm$ onto $\Gamma^m$, and $N_*^m$ is the extension of $N^m$ to a neighborhood of $\Gm$. \\
	
	$\Nhsm$: 
	&
	The normal projection operator $\Nhsm=\hat n^m_{h,*} (\hat n^m_{h,*})^\top$ on $\Ghsm$. \\
	
	$\Nbhsm$: 
	&
	The averaged normal projection operator $\Nbhsm= \frac{\bar n^m_{h,*}}{|\bar n^m_{h,*}|} (\frac{\bar n^m_{h,*}}{|\bar n^m_{h,*}|})^\top$ on $\Ghsm$. \\
	
	$T_*^m$: 
	&
	The tangential projection operator $T_*^m=I - n^m_* (n^m_*)^\top$ on $\Ghsm$. \\
	
	$T^m$: 
	&
	The tangential projection operator $T^m=I - n^m (n^m)^\top$ on $\Gamma^m$. 
	Thus $T^m$ is the lift of $\Tsm$ onto $\Gamma^m$. \\
	
	$\Thsm$: 
	&
	The tangential projection operator $\Thsm=I - \nhsm (\nhsm)^\top$ on $\Ghsm$. \\
	
	$\Tbhsm$: 
	&
	The averaged tangential projection operator $\Tbhsm=I - \frac{\bar n^m_{h,*}}{|\bar n^m_{h,*}|} (\frac{\bar n^m_{h,*}}{|\bar n^m_{h,*}|})^\top$ on $\Ghsm$. \\
	
	$\mathcal{N}(\Gamma_h^m)$: 
	&
	The collection of nodes of $\Gamma_h^m$. 
	%	\\
	%	
	%	$\mathcal{N}_b(\Gamma_h^m)$: 
	%	&
	%	The collection of endpoints (boundary points) of the elements of $\Gamma_h^m$. 
\end{longtable}

\section{Surface calculus formulas}

Let $\Gamma\subset\mathbb R^2$ be a smooth curve, possibly with
boundary. For $u\in C^\infty(\Gamma)$, we denote by
$\ud_i u$, $i=1,2$, the $i$th Cartesian component of
$\nabla_\Gamma u$. The corresponding product rule, chain rule, integration-by-parts formula, commutator identities, and evolution formulas are collected below. We refer to \cite[Lemma~3.14]{BL2025} and the references therein for their proofs.
\begin{lemma}\label{lemma:ud}
	Let $\Gamma$ and $ \Gamma^\prime$ be two smooth curves that are possibly open, such as smooth pieces of some finite element curves, and let $f, h \in C^\infty(\Gamma)$ and $g\in C^\infty(\Gamma^\prime; \Gamma)$ be given functions. Then the following results hold. 
	\begin{itemize}
		\item[1.]  $\ud_i(fh) = \ud_i f h + f\ud_i h$ on $\Gamma$.
		\item[2.] $\ud_i(f\circ g) = (\ud_j f\circ g)\, \ud_i g_j$ on $\Gamma'$.
		\item[3.] $\int_{\Gamma}f \ud_i h = -\int_{\Gamma}\ud_i f h + \int_{\Gamma}f h H n_i + \int_{\partial\Gamma}f h \mu_i$ where $n, \mu$ are the normal and co-normal (tangential) directions, respectively, and $H:=\ud_i n_i$ (with the Einstein notation) is the mean curvature, i.e. the trace of the second fundamental form.
		\item[4.] $\ud_i \ud_j f = \ud_j \ud_i f + n_i H_{jl} \ud_l f -  n_j H_{il} \ud_l f$, where $H_{ij} := \ud_i n_j = \ud_j n_i$.
		\item[5.] If $\Gamma$ evolves under the velocity field $v$, and $G_T := \bigcup_{t\in [0, T]}\Gamma(t) \times \{t\}$, then 
		$$\md(\ud_i f) = \ud_i (\md f )- (\ud_i v_j - n_i n_l \ud_j v_l)\ud_j f \quad\forall\, f \in C^2(G_T) ,$$
		where $\md$ denotes the material derivative with respect to $v$.
		\item[6.] If $f, h \in C^2(G_T)$ then 
		$$\frac{\d}{\d t}\int_{\Gamma} f h = \int_{\Gamma} \md f h + \int_{\Gamma} f\md h + \int_{\Gamma} f h (\nabla_\Gamma\cdot v).$$
		The divergence is defined as $\nabla_\Gamma\cdot v := \ud_i v_i$, which coincides with the intrinsic divergence on the curve if $v$ is a tangential vector field on $\Gamma$. Since the Lagrange interpolation commutes with the material time derivative, it is straightforward to check in the local coordinates that an analogous result also holds for the mass lumping integral, i.e., 
		$$
		\frac{\d}{\d t}\int_{\Gamma_h}^h \tilde f \tilde h = \int_{\Gamma_h}^h \md \tilde f \tilde h + \int_{\Gamma_h}^h \tilde f\md \tilde h + \int_{\Gamma_h}^h \tilde f \tilde h (\nabla_{\Gamma_h}\cdot v_h) ,
		$$
		where $\Gamma_h$ is a finite element curve moving with polynomial velocity $v_h\in [S_h(\Gamma_h)]^2$ (mass lumping is well defined on $\Gamma_h$), and $\tilde f,\tilde h$, defined on $\bigcup_{t\in [0, T]}\Gamma_h(t)\times\{t\}$, are continuous in space and continuously differentiable along the mesh trajectories.  
		\item[7.] The evolution of the unit normal vector $n$ of the curve $\Gamma$ with respect to the velocity field $v$ satisfies the following relation: 
		$$
		\md n_i = -\ud_i v_j n_j .
		$$
	\end{itemize}
\end{lemma}

\section{Super-approximation estimates}
\label{sec:super}

This appendix collects super-approximation estimates for mass lumping
and the projection-error framework; see
\cite{BL2024,BL2025,BGV2027}. Throughout this appendix,
\[
0\leq m\leq l,
\qquad
l\in\big\{0,\ldots,\lfloor T/\tau\rfloor\big\},
\qquad
p,q,r\in[1,\infty].
\]
\begin{lemma}\label{lemma:super_conv}
	The following estimates hold for any piecewise smooth function $f$ and finite element functions $\phi_h,v_h, w_h\in S_h(\Ghsm)$:
	\begin{align*}
		\| (1 - I_h)(f \phi_h) \|_{L^p(\Ghsm)} &\lesssim \| f \|_{W_h^{2,\infty}(\Ghsm)} h \| \phi_h \|_{L^{p}(\Ghsm)} , \notag\\
		\| \nabla_\Ghsm (1 - I_h)(f \phi_h) \|_{L^p(\Ghsm)} &\lesssim \| f \|_{W_h^{2,\infty}(\Ghsm)} h \| \phi_h \|_{W^{1,p}(\Ghsm)} ,\\
		\| (1 - I_h)(v_h w_h) \|_{L^p(\Ghsm)} &\lesssim  h^2 \| v_h \|_{W^{1,q}(\Ghsm)} \| w_h \|_{W^{1,r}(\Ghsm)} ,\\
		\| \nabla_{\Ghsm}(1 - I_h)(v_h w_h) \|_{L^p(\Ghsm)} &\lesssim  h \| v_h \|_{W^{1,q}(\Ghsm)} \| w_h \|_{W^{1,r}(\Ghsm)} ,
	\end{align*}
	whenever $1/p=1/q+1/r$.
\end{lemma}

The following bounds are a direct consequence of Lemma~\ref{lemma:super_conv}.
\begin{lemma}\label{lemma:T<=N2}
	We have
	\begin{align}
		\| \Tsm \ehm \|_{L^p(\Ghsm)}
		&\lesssim
		h\| \ehm \|_{L^p(\Ghsm)}  , \notag\\
		\| \Tsm \ehm \|_{W^{1,p}(\Ghsm)}
		&\lesssim
		h\| \ehm \|_{W^{1,p}(\Ghsm)}
		\notag .
	\end{align}
\end{lemma}

A nonlinear version of Lemma~\ref{lemma:super_conv} follows from the chain and product rules in Lemma~\ref{lemma:ud}.
\begin{lemma}\label{lemma:super_conv-nonlinear}
	Let $f\in W^{3,\infty}(D)$ be defined on a connected open set $D\subset \R^2$, and let $\phi_h, \psi_h \in [S_h(\Ghso)]^2$ be vector-valued finite element functions with ranges contained in $D$. Then
	\begin{align*}
		\| (1 - I_h)(f\circ\phi_h - f\circ\psi_h) \|_{L^p(\Ghso)} &\leq C h \| \phi_h - \psi_h \|_{L^{p}(\Ghso)} ,
		\notag\\
		\| \nabla_\Ghso (1 - I_h)(f\circ\phi_h - f\circ\psi_h) \|_{L^p(\Ghso)} &\leq C h \| \phi_h - \psi_h \|_{W^{1,p}(\Ghso)}  .
	\end{align*}
	The triangle inequality and Lipschitz continuity consequently give
	\begin{align*}
		\| I_h(f\circ\phi_h - f\circ\psi_h) \|_{L^p(\Ghso)} &\leq C \| \phi_h - \psi_h \|_{L^{p}(\Ghso)} ,
		\notag\\
		\| I_h (f\circ\phi_h - f\circ\psi_h) \|_{W^{1,p}(\Ghso)} &\leq C \| \phi_h - \psi_h \|_{W^{1,p}(\Ghso)}  .
	\end{align*}
	The constant $C$ depends on $\| f \|_{W^{3,\infty}(D)}$, $\| \phi_h \|_{W^{1,\infty}(\Ghso)}$ and $\| \psi_h \|_{W^{1,\infty}(\Ghso)}$.
\end{lemma}

The approximation properties of Gauss--Lobatto quadrature yield the following estimate.
\begin{lemma}\label{lemma:super_conv2}
	Let $f$ be smooth on each element $K$ of $\Ghsm$. Suppose that the pullback $f\circ F_K $ vanishes at all Gauss--Lobatto points of the reference segment $K_{\rm f}^0$ for every element $K$ of $\Ghsm$. Then
	\begin{align}
		\Big|\int_\Ghsm f \d\xi \Big|
		\lesssim h^{2} \| f \|_{W^{2,1}_h(\Ghsm)} , \notag
	\end{align}
	where $\|\cdot\|_{W^{2,1}_h(\Ghsm)}$ denotes the piecewise $W^{2,1}$ norm.
\end{lemma}

The following estimate follows directly from Lemma~\ref{lemma:super_conv2}.
\begin{lemma}\label{Lemma-GLW}
	For any smooth function $f$ on $\Gm$, we have
	\begin{align*}
		\Big| \int_{\Gm} \nabla_{\Gm} (f - (I_h f^{-\ell})^\ell) \cdot  \nabla_{\Gm}  \phi_h^\ell \Big|
		\lesssim
		h^{2} \|f\|_{H^{2}(\Gm)} \|\phi_h\|_{H^1(\Ghsm)}
		\quad\forall\,\phi_h\in S_h(\Ghsm) .
	\end{align*}
\end{lemma}

Lemma~\ref{lemma:T<=N2} also yields the following cancellation estimate for the tangential stiffness bilinear form; see \cite[Lemma C.6]{Bai2026}.
\begin{lemma}\label{lemma:AT-sup}
	We have
	\begin{align}
		&| \mathscr A_{h, *}^T(I_h\Nsm\psi_h,\phi_h) |
		+
		| \mathscr A_{h, *}^T(I_h\Nbhsm\psi_h,\phi_h) |
		+
		| \mathscr A_{h, *}^N(I_h\Tsm\psi_h,\phi_h) |
		+
		| \mathscr A_{h, *}^N(I_h\Tbhsm\psi_h,\phi_h) |
		\notag\\
		&\lesssim \min\Big\{\| \psi_h \|_{L^2({\Ghsm})}  \| \phi_h \|_{H^1({\Ghsm})} , \| \psi_h \|_{H^1({\Ghsm})}  \| \phi_h \|_{L^2({\Ghsm})}\Big\}
		. \notag
	\end{align}
\end{lemma}

\section{Discrete norms and Poincar\'{e} inequalities}\label{sec:disc-norm}

The positivity of the Gauss--Lobatto quadrature weights ensures that the discrete $L^p$ quantity
\begin{align}
	\| v \|_{L^p_h(\Ghsm)}  &:= \Big( \int_\Ghsm^h |v|^p \Big)^{\frac1p}
	= \Big( \sum\limits_{K \subset \Ghsm} \int_{K_{\rm f}^0} I_{K_{\rm f}^0} \big( |v \circ F_K |^p |\nabla_{K_{\rm f}^0} F_K| \big) \Big)^{\frac1p}\quad p\in[1,\infty),
	\notag\\
	\| v \|_{L^\infty_h(\Ghsm)}  &:= \| v \|_{L^\infty(\Ghsm)}  , \notag
\end{align}
defines a norm on the finite element space $S_h(\Ghsm)$. Indeed, $\| v \|_{L^p_h(\Ghsm)}=0$ if and only if $v=0$ at every node of $\Ghsm$. The same discrete $L^p$ quantity is also well defined for piecewise continuous functions on $\Ghsm$. The following lemma summarizes its basic properties; see \cite[Lemma 3.7 and Eq. (3.46)]{BL2025}.

\begin{lemma}\label{lemma:lump}
	For $p\in[1,\infty]$, the following relations hold for every finite element function $v_h\in S_h(\Ghsm)$ and all piecewise continuous functions $w_1, w_2, w_3$ on $\Ghsm$:
	\begin{align}
		\| v_h \|_{L^p_h(\Ghsm)} &\sim \| v_h \|_{L^p(\Ghsm)} , \notag\\
		\| \nabla_\Ghsm v_h \|_{L^p_h(\Ghsm)} &\sim \| \nabla_\Ghsm v_h \|_{L^p(\Ghsm)} , \notag\\
		\Big|\int_\Ghsm^h w_1 w_2 w_3 \Big| &\lesssim \| w_1 \|_{L^\infty(\Ghsm)} \| w_2 \|_{L^2_h(\Ghsm)} \| w_3 \|_{L^2_h(\Ghsm)}  \notag .
	\end{align}
	Moreover,
	\begin{align}
		\| I_h\Nbhsm \ehm \|_{L^p(\Ghsm)}
		\sim
		\| \ehm \cdot \nbhsm \|_{L_h^p(\Ghsm)} . \notag
	\end{align}
\end{lemma}

Following the argument in \cite[Section 3.6]{BL2025}, we obtain the following Poincar\'e inequalities.
\begin{lemma}[The Poincar\'e inequality]
	\label{lemma:poincare}
	For sufficiently small $h$, the following Poincar\'e inequalities hold:
	\begin{align}
		\int_\Ghsm | v |^2 &\lesssim \int_\Ghsm | I_h\Nbhsm v |^2 + \int_\Ghsm | \nabla_\Ghsm v |^2
		\quad\forall\, v\in H^1(\Ghsm)^2 ,
		\notag
		\\
		\int_\Ghsm | v |^2 &\lesssim \int_\Ghsm | I_h \Tbhsm v |^2 + \int_\Ghsm | \nabla_\Ghsm v |^2
		\quad\forall\, v\in H^1(\Ghsm)^2 , \notag
	\end{align}
\end{lemma}

Applying these estimates with $v$ equal to $I_h\Nsm \ehm$, $I_h\Tsm \ehm$, $I_h\Nbhsm \ehm$ and $I_h\Tbhsm \ehm$, respectively, and using Lemmas~\ref{lemma:norm-equiv} and~\ref{lemma:super_conv2} to absorb the resulting perturbation terms, we obtain
\begin{align}
	&\int_\Ghsm | I_h\Nsm \ehm |^2 \lesssim  \int_\Ghsm | \nabla_\Ghsm I_h\Nsm \ehm |^2 ,
	\qquad
	\int_\Ghsm | I_h\Tsm \ehm |^2 \lesssim  \int_\Ghsm | \nabla_\Ghsm I_h\Tsm \ehm |^2 ,
	\notag\\
	&\int_\Ghsm | I_h\Nbhsm \ehm |^2 \lesssim  \int_\Ghsm | \nabla_\Ghsm I_h\Nbhsm \ehm |^2 ,
	\qquad
	\int_\Ghsm | I_h\Tbhsm \ehm |^2 \lesssim  \int_\Ghsm | \nabla_\Ghsm I_h\Tbhsm \ehm |^2 .
	\notag
\end{align}

\section*{Acknowledgments}
The author thanks Yupei Xie for discussions on the two-parameter homotopy argument.

\renewcommand{\refname}{\bf References}

\bibliographystyle{abbrv}
\bibliography{MCF}

\end{document}